\documentclass[10pt,reqno]{amsart}
\usepackage{amsmath,amssymb,mathrsfs,graphicx}
\usepackage{ifthen}
\usepackage{a4wide,color}
\usepackage{caption}
\usepackage{sidecap}
\usepackage{rotating}
\usepackage{colortbl}
\usepackage{float}
\usepackage{comment}
\usepackage{longtable}

\usepackage{babel}

\numberwithin{equation}{section}

\numberwithin{figure}{section}

\theoremstyle{plain}

\theoremstyle{plain}

\theoremstyle{remark}

\theoremstyle{plain}

\allowdisplaybreaks
\let \pa \partial

\let \eps \varepsilon
\newcommand{\N}{{\mathbb N}}
\newcommand{\R}{{\mathbb R}}

\newtheorem{theorem}{Theorem}[section]
\newtheorem{lemma}{Lemma}[section]

\newtheorem{proposition}{Proposition}[section]
\newtheorem{remark}{Remark}[section]

\let \eps=\varepsilon

\let \la=\lambda

\let \pa=\partial

\let \om=\omega

\makeatother
\providecommand{\lemmaname}{Lemma}
\providecommand{\propositionname}{Proposition}
\providecommand{\remarkname}{Remark}
\providecommand{\theoremname}{Theorem}

\begin{document}

\title[1D Boltzmann equation with hard potential]{
Time-asymptotic expansion and fluid approximation of the 1D Boltzmann equation with hard potentials
}
\author{Hui-Tzu Chang}
\address{Hui-Tzu Chang, Department of Mathematics, National Cheng Kung
University, Tainan, Taiwan}
\email{L18101022@gs.ncku.edu.tw}

\author{Haitao Wang}
\address{Haitao Wang, School of Mathematical Sciences, Institute of Natural
Sciences, MOE-LSC, CMA-Shanghai, Shanghai Jiao Tong University, Shanghai,
China}
\email{haitallica@sjtu.edu.cn}

\author{Kung-Chien Wu}
\address{Kung-Chien Wu, Department of Mathematics, National Cheng Kung
University, Tainan, Taiwan and National Center for Theoretical Sciences,
National Taiwan University, Taipei, Taiwan}
\email{kungchienwu@gmail.com}

\begin{abstract}

We establish a quantitative large-time comparison between the macroscopic Boltzmann solution and the full nonlinear compressible Navier–Stokes solution in one space dimension. For Maxwellian molecules and hard potentials with angular cutoff, we consider small, spatially compactly supported perturbations of a global Maxwellian with polynomial velocity decay, without assuming spatial differentiability. Nonlinear diffusion waves persist at leading order, making comparison through linearized equations alone insufficient. We construct a coupled system for the leading nonlinear fluid waves and derive cancellation identities from the exact kinetic spectral representation to control interactions between distinct characteristic families. Combined with a particle–fluid decomposition, these yield pointwise estimates separating the leading waves, their interactions, and a faster-decaying remainder. Matching the nonlinear wave dynamics then yields space-time pointwise error estimates and an \(L_x^1\) convergence rate arbitrarily close to \((1+t)^{-1/2}\).

\end{abstract}


\maketitle

\section{Introduction}

\subsection{The model}

We consider the one-dimensional Boltzmann equation (BE)
\begin{equation}
\left\{
\begin{array}{l}
\partial _{t}F+\xi^{1}\pa _{x}F=Q(F,F)\text{,} \\[4mm]
F(0,x,\xi )=F_{0}(x,\xi )\text{,}%
\end{array}%
\right. \quad (t,x,\xi )\in {\mathbb{R}}^{+}\times {\mathbb{R}}\times {%
\mathbb{R}}^{3}\text{,}  \label{bot.1.a}
\end{equation}%
where $F(t,x,\xi )$ is the particle velocity distribution at time
$t\geq0$, position $x\in {\mathbb{R}}$, and microscopic velocity
$\xi =(\xi^{1},\xi^{2},\xi^{3})\in {\mathbb{R}}^{3}$.
The bilinear collision operator is given by
\begin{equation*}
Q(F,G)=\frac{1}{2}\int_{{\mathbb{R}}^{3}\times S^{2}}B(|\xi _{\ast }-\xi
|,\Omega )\left\{ F_{\ast }^{\prime }G^{\prime }+G_{\ast }^{\prime
}F^{\prime }-F_{\ast }G-FG_{\ast }\right\} d\xi _{\ast }d\Omega \text{,}
\end{equation*}%
with collision kernel
\begin{equation*}
B(|\xi -\xi _{\ast }|,\Omega)=|\xi -\xi _{\ast }|^{\gamma }\beta (\vartheta ).
\end{equation*}%
We use the usual notation $F=F\left( t,x,\xi \right) $,
$F_{\ast}=F_{\ast
}\left( t,x,\xi _{\ast }\right) $,
$F^{\prime }=F(t,x,\xi' )$, and
$F_{\ast
}^{\prime }=F\left( t,x,\xi _{\ast }^{\prime }\right) $.
The post-collisional velocities satisfy
\begin{equation*}
\xi ^{\prime }=\xi -[(\xi -\xi _{\ast })\cdot \Omega ]\Omega \text{,}\quad
\xi _{\ast }^{\prime }=\xi_{\ast } +[(\xi -\xi _{\ast })\cdot \Omega ]\Omega \text{,
\ \ }\Omega \in S^{2}\text{.}
\end{equation*}%

Throughout this paper, we consider \textit{Maxwellian molecules}
$\left( \gamma
=0\right) $ and \textit{hard potentials} ($0<\gamma< 1$).
The angular kernel $\beta (\vartheta )$ satisfies Grad's angular
cutoff assumption
\begin{equation*}
0<\beta (\vartheta )\leq C\left\vert \cos \vartheta \right\vert \text{,}
\end{equation*}%
for some constant $C>0$, where the collision angle $\vartheta $
is defined by
\begin{equation*}
\cos \vartheta =\frac{|(\xi -\xi _{\ast })\cdot \Omega |}{|\xi -\xi _{\ast }|%
}\text{.}
\end{equation*}
The case $\gamma=1$ with
$\beta (\vartheta )=\lvert \cos \vartheta\rvert$ is the hard-sphere model.

It is known that Global Maxwellians are equilibrium of the Boltzmann equation.
We study perturbations of a global Maxwellian and write
\begin{equation}
F=\mathcal{M}+\sqrt{\mathcal{M}}f\text{, \ \ \ }F_{0}(x,\xi )=\mathcal{M}%
+\varepsilon \sqrt{\mathcal{M}}f_{0}\text{,}  \label{F}
\end{equation}%
where $\varepsilon >0$ is sufficiently small and $f$ is the
perturbation. The global Maxwellian $\mathcal{M}$ is normalized by
\begin{equation*}\displaystyle
\mathcal{M}= (2\pi)^{-3/2}e^{-\frac{\lvert \xi \rvert^2}{2}}\text{.}
\end{equation*}%
Substituting \eqref{F} into \eqref{bot.1.a}, we find that $f$ satisfies
\begin{equation}
\left\{
\begin{array}{l}
\partial _{t}f+\xi^{1}\pa _{x}f=Lf+\Gamma \left( f,f\right),
\vspace {3mm}
\\
f|_{t=0}=\varepsilon f_{0}\text{, \ \ }\left( x,\xi \right) \in \mathbb{R}\times \mathbb{R}^{3}\text{,}%
\end{array}%
\right.  \label{Linearized}
\end{equation}%
where the linearized collision operator and the nonlinear collision
term are defined by
\begin{equation*}
Lf=2  \mathcal{M}^{-1/2}Q(  \mathcal{M}, \mathcal{M}^{1/2}f),\qquad \Gamma \left(
f,f\right) = \mathcal{M}^{-1/2}Q(  \mathcal{M}^{1/2}f, \mathcal{M}^{1/2}f) \text{.}
\end{equation*}%

We impose even symmetry in each transverse velocity variable
$\xi^2$ and $\xi^3$ on the initial data and work within this symmetry class.
Within this class, the null space of $L$ is three-dimensional,
with the orthonormal basis
\begin{equation*}
	\mathrm{Ker}(L)=\mbox{span}\left\{ \chi _{0},\chi _{1},\chi _{4}\right\} =\mbox{span} \left\{ \mathcal{M}%
	^{1/2},\ \xi^{1}\mathcal{M}^{1/2},\ \frac{1}{\sqrt{6}}(|\xi |^{2}-3)%
	\mathcal{M}^{1/2}\right\} .
\end{equation*}%
We use the macro--micro decomposition associated with this null
space. Let $\mathbb{P}_{0}$ be the orthogonal projection onto
$\mathrm{Ker}(L)$ with respect to the $L_{\xi}^{2}$ inner product,
and set $\mathbb{P}_{1}\equiv\mathrm{Id}-\mathbb{P}_{0}$.

The Euler characteristic speeds and directions are determined by
$\mathbb{P}_{0}\xi^{1}$ acting on $\mathbb{P}_{0}L^{2}_{\xi}$:
\begin{equation}
\left\{
\begin{array}{l}
\mathbb{P}_{0}\xi^{1}\mathbb{P}_{0}E_{j}=\la_{j}E_{j}\text{, } \\[2mm]
\la_{1}=-\sqrt{\frac{5}{3}}\text{,}\quad \la_{2}=0\text{,}%
\,\quad \la_{3}=\sqrt{\frac{5}{3}}\text{,} \\
E_{1}=\sqrt{\frac{3}{10}}\chi _{0}-\sqrt{\frac{1}{2}}\chi_{1}+\sqrt{\frac{1}{5}}\chi _{4}\text{,} \\[2mm]
E_{2}=-\sqrt{\frac{2}{5}}\chi _{0}+\sqrt{\frac{3}{5}}\chi _{4}\text{,} \\[2mm]
E_{3}=\sqrt{\frac{3}{10}}\chi _{0}+\sqrt{\frac{1}{2}}\chi_{1}+\sqrt{\frac{1}{5}}\chi _{4}\,\text{.}
\end{array}%
\right.  \label{Eigen function}
\end{equation}%
We use these coordinates for the fluid systems in
Subsection~\ref{subsec:main-results}.

\subsection{Notations}
In order to proceed, we introduce some notations that will be frequently used throughout the paper. We denote $\left\langle \xi \right\rangle ^{\beta }=(1+|\xi
|^{2})^{\beta /2}$, $\beta \in {\mathbb{R}}$. For the microscopic variable $%
\xi $, we denote
\begin{equation*}
	|g|_{L_{\xi }^{q}}=\Big(\int_{{\mathbb{R}}^{3}}|g|^{q}d\xi \Big)^{1/q}%
	\hbox{
		if }1\leq q<\infty \hbox{,}\quad \quad |g|_{L_{\xi }^{\infty }}=\sup_{\xi
		\in {\mathbb{R}}^{3}}|g(\xi )|\hbox{,}
\end{equation*}%
and the weighted norms can be defined by
\begin{equation*}
	|g|_{L_{\xi ,\beta }^{q}}=\Big(\int_{{\mathbb{R}}^{3}}\left\vert
	\left\langle \xi \right\rangle ^{\beta }g\right\vert ^{q}d\xi \Big)^{1/q}%
	\hbox{
		if }1\leq q<\infty \hbox{,}\quad \quad |g|_{L_{\xi ,\beta }^{\infty
	}}=\sup_{\xi \in {\mathbb{R}}^{3}}\left\vert \left\langle \xi \right\rangle
	^{\beta }g(\xi )\right\vert \hbox{.}
\end{equation*}
and
\begin{equation*}
	|g|_{L_{\xi }^{\infty }(m)}=\sup_{\xi \in {\mathbb{R}}^{3}}\left \{ |g(\xi
	)|m(\xi)\right \} \text{,}
\end{equation*}%
where $m$ is a weight function.

The $L_{\xi }^{2}$ inner product in ${\mathbb{R}}^{3}$ will be denoted by $%
\left\langle \cdot ,\cdot \right\rangle _{\xi }$, i.e.,
\begin{equation*}
	\left\langle f,g\right\rangle _{\xi }=\int_{\mathbb{R}^3} f(\xi )\overline{g(\xi )}d\xi %
	\hbox{.}
\end{equation*}%
For the Boltzmann equation, the natural norm in $\xi $ is $|\cdot
|_{L_{\sigma }^{2}}$, which is defined as
\begin{equation*}
	|g|_{L_{\sigma }^{2}}^{2}=\left|\left \langle \xi \right \rangle ^{\frac{%
			\gamma }{2}}g\right|_{L_{\xi }^{2}}^{2}\text{.}
\end{equation*}%

For the space variable $x$, we have similar notations, namely,
\begin{equation*}
	|g|_{L_{x}^{q}}=\left( \int_{{\mathbb{R}}}|g|^{q}dx\right) ^{1/q}%
	\hbox{
		if }1\leq q<\infty \hbox{,}\quad \quad |g|_{L_{x}^{\infty }}=\sup_{x\in {%
			\mathbb{R}}}|g(x)|\hbox{.}
\end{equation*}

Finally, with $\mathcal{X}$ and $\mathcal{Y}$ being two normed spaces, we
define%
\begin{equation*}
	\left\Vert g\right\Vert _{\mathcal{XY}}=\left\vert \left\vert g\right\vert _{%
		\mathcal{Y}}\right\vert _{\mathcal{X}}\hbox{,}
\end{equation*}%
and for simplicity, we denote
\begin{equation*}
	\left\Vert g\right\Vert _{L^{2}}=\left\Vert g\right\Vert _{L_{\xi
		}^{2}L_{x}^{2}}=\left( \int_{\mathbb{R}^{3}}\left\vert g\right\vert
	_{L_{x}^{2}}^{2}d\xi \right) ^{1/2}\hbox{.}
\end{equation*}%

The domain decomposition plays an important role in our analysis, so we
	introduce a cut-off function $\chi :{\mathbb{R}}\rightarrow {\mathbb{R}}$,
	which is smooth and even, and non-increasing on $[0,\infty)$, $\chi (s)=1$ for $|s|\leq \frac{1}{2}$, $\chi
	(s)=0$ for $|s|\geq 1$ and $0\leq \chi \leq 1$. Moreover, we define $\chi
	_{R}(s)=\chi (s/R)$ for positive $R$.

The Fourier transform is used frequently in this paper. Let $f:\R\to \mathbb{C}$ be a function, we define the Fourier transform of $f$ as
	$$
	\hat{f}(\eta)=\int_{\R}e^{-ix\eta}f(x)dx\,,
	$$
	and the inverse Fourier transform of $f$ is defined by
	$$
	\check{f}(x)=(2\pi)^{-1}\int_{\R}e^{ix\eta}f(\eta)d\eta\,.
	$$

For any two functions $f(x,t)$ and $g(x,t)$, we define the space-time convolution by
\begin{equation*}
f(x,t)\ast _{x,t}g(x,t)=\int_{0}^{t}\int_{{\mathbb{R}}}f(x-y,t-\tau
)g(y,\tau )dyd\tau .
\end{equation*}
For simplicity of notations, hereafter, we abbreviate $\leq C$ to $\lesssim $%
, where $C$ is a constant depending only on fixed numbers. A subscript on $\lesssim$ indicates that the implicit constant may
depend on the specified parameter, in addition to the other fixed
parameters.

\subsection{Review of previous works}
We review the pointwise wave analysis of the Boltzmann equation and its
relation to the compressible Navier--Stokes (NS) equations. These two
questions are closely related in one space dimension, where nonlinear
diffusion waves contribute at leading order.

Liu and Yu \cite{[Yu]} developed a pointwise Green's-function approach to
the one-dimensional hard-sphere Boltzmann equation. Their analysis
separates particlelike singular waves from fluidlike waves and describes
the large-time behavior of small, localized perturbations without assuming
spatial smoothness of the initial data. The approach has also been used
for initial-boundary value problems \cite{[LiuYu-07]}, pointwise shock
stability \cite{[Yu-10]}, and invariant manifolds \cite{[LiuYu-13]}.
Related pointwise analyses include the three-dimensional linearized BE
\cite{[LiuYu1],[LiuYu2]}, the one-dimensional Vlasov--Poisson--Boltzmann
and Vlasov--Maxwell--Boltzmann systems
\cite{[LiYangZhong-20],[LiYangZhong-23]}, and three-dimensional hard-sphere
solutions with polynomial velocity tails \cite{[LinWangWu2]}.

For hard potentials, the linear pointwise analysis of Lee, Liu, and Yu
\cite{[LeeLiuYu-07]} provides a starting point. The spatial decay depends
on the potential power $\gamma$ and the initial velocity weight. The loss
of velocity weight in these estimates prevents their direct use in a
nonlinear iteration with a fixed velocity weight. In \cite{[Haitao]}, an
enhanced Mixture Lemma yields a particle--fluid decomposition: the
particle part carries the loss of velocity weight and decays rapidly,
while the fluid part has the regularity and velocity decay needed for the
nonlinear estimates. This leads to the space-time wave structure of the
three-dimensional nonlinear BE, including Huygens, diffusion, and Riesz
waves. In that setting, nonlinear contributions are of lower order than
the leading linear waves.

The one-dimensional problem requires a different treatment of the
leading nonlinear terms. This is already illustrated by the Burgers
equation
\[
\theta_t+\theta\theta_x=\theta_{xx}.
\]
For localized data, its leading diffusion wave is generally nonlinear.
Although it has the same temporal decay order as a heat wave, its
nonlinear profile cannot in general be replaced by the linear heat
evolution with an error of smaller asymptotic order; see
\cite{[Liu-17],[Liu-21]}.

For systems of viscous conservation laws, the flux and the physical
viscosity couple distinct characteristic families. The analysis in
\cite{[Liu-21],[Zeng]} separates the leading diagonal waves from the
remaining interactions. According to whether a characteristic field is
genuinely nonlinear or linearly degenerate, the leading wave is described
by a convected Burgers equation or a heat equation:
\begin{equation}\label{int-Burger}
 \theta_{it}+\la_i\theta_{ix}+c_{ii}(\theta_i^2)_x
 =\mu_i\theta_{ixx},
\end{equation}
where $\la_i,c_{ii},\mu_i$ are constants with $\mu_i>0$, and $c_{ii}=0$
in the linearly degenerate case. To illustrate the nonlinear coupling,
consider a model with diagonal constant linear viscosity. After
subtracting the leading waves, the remainder equation has the schematic
form
\begin{multline}\label{int-offdiagonal}
 v_{it}+\la_i v_{ix}=\mu_i v_{ixx}
 -\sum_{(j,k)\neq(i,i)}C^i_{jk}(\theta_j\theta_k)_x\\
 +\bigl(O(1)|\boldsymbol\theta||\boldsymbol v|
       +O(1)|\boldsymbol v|^2+O(1)|\boldsymbol\theta|^3\\
 \qquad+O(1)(|\boldsymbol\theta|+|\boldsymbol v|)
              (|\boldsymbol\theta_x|+|\boldsymbol v_x|)\bigr)_x.
\end{multline}
For physical viscosity, off-diagonal linear viscous terms must also be
retained; the comparison with diagonal artificial viscosity is part of
the fluid approximation below.

In the Duhamel representation of \eqref{int-offdiagonal}, the terms
$C^i_{jj}(\theta_j^2)_x$, $j\neq i$, couple a heat kernel with a source
traveling at a different speed. The cancellation mechanism in
\cite[Lemma~3.4]{[Zeng]} uses integration by parts along these distinct
characteristic directions. Its application requires derivative estimates
for the waves and their remainders. The resulting pointwise bounds
separate a leading diffusion wave from algebraic coupling tails of the
form
\begin{equation}\label{eq:est-vis}
 (1+t)^{-1/2}e^{-\frac{(x-\la_i t)^2}{1+t}}
 +\psi_i^{3/2}+\sum_{j\neq i}\widetilde\psi_j^{3/2},
\end{equation}
where the coupling profiles are defined in \eqref{coupling1} and
\eqref{coupling2}; see \cite[Sections~3 and~4]{[Zeng]}.
Koike \cite{[Kai]} further developed time-asymptotic expansions with
pointwise remainder estimates for one-dimensional viscous compressible
flow.

Motivated by this structure, Liu and Yu \cite[Section~10]{[Yu]} introduced
decoupled kinetic Burgers equations and studied the associated fluid and
nonfluid remainders. In the notation used here, their defining integral
equation takes the form
\begin{equation}\label{int-kinetic-Burger}
 \tilde\phi_i(t,x)=\varepsilon\langle E_i,\mathbb G^t f_0\rangle_\xi
 +\int_0^t\Big\langle E_i,\mathbb G^{t-\tau}
       \Gamma(\tilde\phi_iE_i,\tilde\phi_iE_i)(\tau)\Big\rangle_\xi
       d\tau.
\end{equation}
Here $\mathbb G^t$ is the full linearized solution operator defined in
\eqref{Green}, and $E_i$ is given in \eqref{Eigen function}.
Equation \eqref{int-kinetic-Burger} recalls the formulation in
\cite{[Yu]}; the branchwise kinetic Burgers waves used in this paper are
defined in \eqref{kinetic-burger}. The kinetic Burgers formulation and
the analysis of interactions in \cite{[Yu]} are important antecedents of
our wave decomposition. Such estimates describe the difference from
kinetic diffusion waves, but do not by themselves estimate the
difference between the nonlinear BE and NS solutions. A coarser closure
without separating the leading kinetic Burgers waves may yield bounds
such as $\sum_{j=1}^3\psi_j^1$; see
\cite{[LiYangZhong-20],[LiYangZhong-23]}. This upper bound is not
integrable in $x$ and therefore does not give the $L_x^1$-asymptotic
control sought here.

The fluid-like wave structure of the Boltzmann equation suggests a close connection with the compressible Navier--Stokes equations, as formally indicated by the Chapman--Enskog expansion. This motivates the study of the time-asymptotic relationship between their nonlinear solutions. For the three-dimensional problem, Kawashima, Matsumura, and Nishida
\cite{[Kawashima-Matsumura-Nishida]} established time-asymptotic
equivalence between BE and NS.

In corresponding macroscopic variables, the comparison takes the
schematic form
\[
 \|f_{\rm BE}-u_{\rm NS}\|
 \leq\|f_{\rm BE}-f_{\rm LBE}\|
     +\|f_{\rm LBE}-u_{\rm LNS}\|
     +\|u_{\rm LNS}-u_{\rm NS}\|,
\]
where the subscripts in the latter terms indicate the respective linearized equations. In that
three-dimensional setting, the nonlinear corrections decay faster than
the leading linear waves, allowing the comparison to be reduced to the
linearized systems. In one dimension, nonlinear effects in genuinely
nonlinear characteristic families persist at leading order. Thus
comparison through the linearized equations alone does not yield the
desired lower-order error.

Our objective is therefore not only to describe similar wave structures
for BE and NS, but to control the difference between their nonlinear
solutions. We construct a coupled system that retains the leading
interactions of the three characteristic families and estimate the
remaining terms. This decomposition is used to compare the macroscopic
BE solution with both the quadratic NS approximation and the full NS
system with physical viscosity. The quantitative pointwise and $L_x^1$
comparisons are stated in Theorem~\ref{cor:full-NS-comparison}.

\subsection{Main results}
\label{subsec:main-results}

We establish two main results. The first describes the leading fluid
waves, their off-diagonal corrections, and the microscopic component,
with faster-decaying bounds for the exact remainder. The second uses
this structure to establish a quantitative time-asymptotic approximation
of the macroscopic Boltzmann solution by the quadratic and full
compressible Navier--Stokes systems.

To state these results, we introduce the fluid-wave profile
\begin{equation}\label{fluid-def}
\theta^{\alpha}_{j}(x,t)=(1+t)^{-\frac{\alpha}{2}}\left(1+\frac{\vert x-\lambda_{j} t\vert^{2}}{1+t}\right)^{-N},
\end{equation}
where $N$ may be chosen arbitrarily large, and the coupling-wave
profiles
\begin{equation}\label{coupling1}
\psi^{\alpha}_{j}(x,t)=\left(1+t+\vert x-\lambda_{j} t\vert^{2}\right)^{-\frac{\alpha}{2}}
\end{equation}
and
\begin{equation}\label{coupling2}
\tilde{\psi}^{\alpha}_{j}(x,t)=\left((1+t)^{2}+\vert x-\lambda_{j} t\vert^{3}\right)^{-\frac{\alpha}{3}}\,,
\end{equation}
Here $\la_j$ is the characteristic speed defined in
\eqref{Eigen function}. The profiles $\theta_j^\alpha$ describe
localization near the Euler characteristics, while $\psi_j^\alpha$
and $\tilde\psi_j^\alpha$ retain the algebraic tails of wave
interactions.

We first state the result on the wave structure.
\begin{theorem}[Pointwise structure of the Boltzmann solution]
\label{thm:pointwise-structure}
Let $0\leq\gamma<1$, $\beta>3/2+\gamma$, $N>4$,
$N_0>2N(1-\gamma)$, and fix $0<\delta_0<1/4$.
Choose $k$ and $\hat\beta$ as in \eqref{eq-hat-beta}.
Assume that $f_0$ is even in each of $\xi^2$ and $\xi^3$ and
is compactly supported in $x$ for all $\xi$, namely,
\begin{equation*}
f_{0}\left( x,\xi \right) \equiv 0\text{ for }\left\vert x\right\vert \geq 1%
\text{, }\xi \in \mathbb{R}^{3}\text{,}
\end{equation*}%
and
\[
 f_0\in L_x^\infty L^\infty_{\xi,\hat\beta},\qquad
 \|f_0\|_{L_x^\infty L^\infty_{\xi,\hat\beta}}
 =\mathop{\mathrm{ess\,sup}}_{x,\xi}
       \langle\xi\rangle^{\hat\beta}|f_0(x,\xi)|<\infty.
\]
For $\varepsilon>0$ sufficiently small, let $f$ be the solution of
\eqref{Linearized}. Then $f$ admits the macro--micro decomposition
$f=\mathbb{P}_{0}f+\mathbb{P}_{1}f$, with macroscopic component
$$
\mathbb{P}_{0}f=\sum_{i=1}^{3}(\phi_{i}+R_{i})E_{i}\,.
$$
Here $\Phi=\sum_{i=1}^3\phi_iE_i$ and $V^\perp$ are the leading
fluid and nonfluid waves defined in \eqref{leading fluid} and
\eqref{leading-nonfluid}, respectively, and
$R_i=\langle E_i,R\rangle_\xi$, where $R=f-\Phi-V^\perp$
is the exact remainder.

The leading fluid waves admit the decomposition
$\phi_{i}=\tilde{\phi}_{i}+V_{i,A}+V_{i,B}$, where
$\tilde{\phi}_{i}$ is the diagonal kinetic Burgers wave defined in
\eqref{expression tilde-phi-i},
and $V_{i,A}$ and $V_{i,B}$ are the non-smooth and smooth
components of the off-diagonal wave, respectively. They satisfy
\begin{equation*}
		\left\{\begin{array}{l}
			\displaystyle    \left\vert \pa^{s}_{x}\tilde{\phi}_{i}(t,x)\right\vert\lesssim\varepsilon\theta^{s+1}_{i}\,,\quad s\in \N\cup\left\{0\right\}\\ \\
			\displaystyle  \left\vert V_{i,A}(t,x)\right\vert \lesssim \varepsilon(1+t+\vert x\vert)^{-\frac{N_{0}}{1-\gamma}}\,,\\ \\
			\displaystyle \left\vert V_{i,B}(t,x)\right\vert \lesssim \varepsilon \Bigg(\psi_{i}^{\frac{3}{2}}+\sum_{ k\neq i}\tilde{\psi}_{k}^{\frac{3}{2}}\Bigg)\,,\\ \\
			\displaystyle \left\vert \partial_{x}V_{i,B}(t,x)\right\vert\lesssim\varepsilon (1+t)^{-\frac{1}{2}}\Bigg(\psi_{i}^{\frac{3}{2}}+\sum_{ k\neq i}\tilde{\psi}_{k}^{\frac{3}{2}}\Bigg)\,.
		\end{array}
		\right.
	\end{equation*}
The fluid remainder admits the decomposition
$R_{i}=R_{i,A}+R_{i,B}$, where $R_{i,A}$ and $R_{i,B}$ are
its non-smooth and smooth components, respectively. Moreover,
\begin{equation*}
		\left\{\begin{array}{l}
			\displaystyle  \left\vert R_{i,A}(t,x)\right\vert \lesssim \varepsilon^{2}(1+t+\vert x\vert)^{-\frac{N_{0}}{1-\gamma}}\,,\\ \\
			\displaystyle \left\vert R_{i,B}(t,x)\right\vert\lesssim \varepsilon^{2}\left[\theta_{i}^{2-\delta_{0}}+(1+t)^{-\frac{1}{4}+\frac{\delta_{0}}{4}}\left(\psi_{i}^{\frac{3}{2}}+\sum_{k\neq i}\tilde{\psi}_{k}^{\frac{3}{2}}\right)\right]\,,\\ \\
			\displaystyle \left\vert \partial_{x}R_{i,B}(t,x)\right\vert\lesssim \varepsilon^{2}\left[\theta_{i}^{3-\delta_{0}}+(1+t)^{-\frac{3}{4}+\frac{\delta_{0}}{4}}\left(\psi_{i}^{\frac{3}{2}}+\sum_{k\neq i}\tilde{\psi}_{k}^{\frac{3}{2}}\right)\right]\,.
		\end{array}
		\right.
	\end{equation*}
The microscopic component $\mathbb{P}_{1}f$ satisfies
$$
\left\vert \mathbb{P}_{1}f\right\vert_{L^{\infty}_{\xi,\beta}} \lesssim \varepsilon\left[(1+t+|x|)^{-\frac{N_{0}}{1-\gamma}}+\sum_{l=1}^{3}\theta^{2}_{l}\right]+\varepsilon^{2}\left[\sum_{l=1}^{3}(1+t)^{-\frac{1}{2}}\psi_{l}^{\frac{3}{2}}\right].
$$
\end{theorem}

The estimates distinguish the diagonal diffusion waves from the
corrections caused by other characteristic families. For each fixed
$C>0$, on $|x-\lambda_i t|\leq C\sqrt{1+t}$ they give
\[
 \begin{aligned}
 |\tilde\phi_i(t,x)|&\lesssim\varepsilon(1+t)^{-1/2},\qquad
 |V_{i,B}(t,x)|\lesssim\varepsilon(1+t)^{-3/4},\\
 |R_{i,B}(t,x)|&\lesssim_{\delta_0}
                  \varepsilon^2(1+t)^{-1+\delta_0/2}.
 \end{aligned}
\]
These upper bounds display the hierarchy along each characteristic;
the full estimates also keep the coupling tails away from that
characteristic. In particular, the remainder is smaller both in
perturbation size and in temporal order than the displayed leading
fluid bounds.

The initial perturbation need not be differentiable in $x$, and the
velocity hypothesis is a finite polynomial weight on $f_0$ in
\eqref{F}. The compact-support assumption removes an initial spatial
tail from the displayed estimates. The order $N$ is fixed before the
choices of $N_0$, $k$, and $\hat\beta$; faster prescribed spatial decay
requires a corresponding velocity weight through \eqref{eq-hat-beta}.
We do not claim optimality of these localization and weight assumptions.

We next compare the macroscopic Boltzmann solution directly with the
two nonlinear fluid systems. The first is the approximation obtained
by the formal Chapman--Enskog expansion calculation in Subsection~4.1; the second retains the
full Euler flux and the state-dependent physical viscosity.

The quadratic Navier--Stokes approximation with frozen physical
viscosity is
\[
 \partial_t\omega_i+\lambda_i\partial_x\omega_i
 =\sum_{j=1}^3A_{ij}\partial_x^2\omega_j
       +\sum_{k,l=1}^3B_i^{kl}\partial_x(\omega_k\omega_l)\,,
\]
where
\[
 \begin{aligned}
 A_{ij}&=\langle\mathbb P_1\xi^1E_i,
                       (-L)^{-1}\mathbb P_1\xi^1E_j\rangle_\xi\,,\\
 B_i^{kl}&=\langle L^{-1}\mathbb P_1\xi^1E_i,
            \Gamma(E_k,E_l)\rangle_\xi\,.
 \end{aligned}
\]
We also consider the full compressible Navier--Stokes system
\[
 \partial_t\omega_i^f+\partial_x\mathcal F_i(\omega^f)
 =\partial_x\left(
       \sum_{j=1}^3\mathcal A_{ij}(\omega^f)
                          \partial_x\omega_j^f\right)\,,
\]
where $\mathcal F$ and $\mathcal A$ denote the full Euler flux and
physical viscosity matrix of the compressible Navier--Stokes
system, expressed in the equilibrium characteristic coordinates
corresponding to $E_i$. We use the physical constitutive relations
and structural setting of Subsection~\ref{sec43:full-comparison},
with transport coefficients matched to the Boltzmann collision
operator. After subtracting the equilibrium flux, the coefficient
matching is
\[
 \begin{aligned}
 \mathcal F(0)&=0,\qquad
 D_z\mathcal F(0)=\operatorname{diag}(\lambda_1,\lambda_2,\lambda_3),\\
 \frac12D_z^2\mathcal F_i(0)[z,z]&=-\sum_{k,l=1}^3B_i^{kl}z_kz_l,\qquad
 \mathcal A(0)=(A_{ij}).
 \end{aligned}
\]
Here $i=1,2,3$, and the two systems have the common initial data
\[
 \omega_i(0,x)=\omega_i^f(0,x)
 =\varepsilon\big\langle\check\chi\ast_xf_0,E_i\big\rangle_\xi.
\]
These are smoothed macroscopic initial data. Since $\chi(0)=1$,
the smoothing preserves their spatial integrals.

\begin{theorem}[Time-asymptotic fluid approximation]
\label{cor:full-NS-comparison}
Assume the hypotheses of Theorem~\ref{thm:pointwise-structure}.
Let $\omega$ and $\omega^f$ solve the quadratic and full NS systems
defined above, with the common initial data.
For sufficiently small $\varepsilon>0$, the following estimate
holds for $i=1,2,3$, $x\in\mathbb R$, and $t\geq0$:
\begin{equation*}
 \begin{aligned}
 &\left|\left\langle\mathbb P_0f,E_i\right\rangle_\xi-\omega_i\right|
  +\left|\left\langle\mathbb P_0f,E_i\right\rangle_\xi-\omega_i^f\right|
 \\
 &\quad\lesssim_{\delta_0}\varepsilon\left[
 (1+t+|x|)^{-\frac{N_0}{1-\gamma}}+\theta_i^{2-\delta_0}
                  +\sum_{j\ne i}\theta_j^2\right]
 \\
 &\qquad+\varepsilon^2(1+t)^{-\frac{1-\delta_0}{4}}
       \left(\psi_i^{3/2}+\sum_{j\ne i}\tilde\psi_j^{3/2}\right)
       +\varepsilon\mathcal P_i(t,x),
 \end{aligned}
\end{equation*}
where
\begin{equation*}
 \begin{aligned}
 \mathcal P_i(t,x)
 ={}&(1+t)^{-\frac14+\frac{\delta_0}{2}}
       \sum_{j=1}^3\left[
       \frac{1}{(1+t+|x-\lambda_j(1+t)|^2)^{3/4}}
       \right.
 \\
 &\qquad\left.
       +\frac{1}{(1+t+|x-\lambda_j(1+t)|)^{1/2}
                       (1+t+|x-\lambda_j(1+t)|^2)^{3/8}}
       \right]
 \\
 &+(1+t)^{-1/2}|x-\lambda_i(1+t)|^{-1+\frac{\delta_0}{2}}
       \mathbf1_{\left\{
       \frac{2\sqrt{1+t}}{\delta_0}\leq|x-\lambda_i(1+t)|
       \leq\frac{\delta_0(1+t)}{2}\right\}}.
 \end{aligned}
\end{equation*}
The last term in $\mathcal P_i$ is understood to vanish outside
its indicated set, including when the set is empty.
In particular,
\begin{equation*}
 \begin{aligned}
 &\left\|\left\langle\mathbb P_0f(t),E_i\right\rangle_\xi
                         -\omega_i(t)\right\|_{L_x^1}
 +\left\|\left\langle\mathbb P_0f(t),E_i\right\rangle_\xi
                         -\omega_i^f(t)\right\|_{L_x^1}
 \\
 &\qquad\lesssim_{\delta_0}
           \varepsilon(1+t)^{-\frac12+\frac{\delta_0}{2}}.
 \end{aligned}
\end{equation*}
\end{theorem}

The first estimate justifies the quadratic approximation derived in
Subsection~4.1. The second gives the approximation by the full nonlinear
NS system. The estimate controls the differences of solutions, rather
than only giving separate bounds with similar wave profiles.
The $L_x^1$ errors tend to zero at the stated rate as $t\to\infty$;
this is a large-time comparison at fixed collision scale, not a
small-Knudsen-number limit. The profile $\mathcal P_i$ records the
comparison of artificial and physical viscosity. The kinetic wave
terms remain explicit through $\theta$, $\psi$, and $\tilde\psi$.

\subsection{Key ingredients and ideas of the proof}
The proof combines a coupled nonlinear wave decomposition with a
quantitative comparison of the kinetic and fluid dynamics. We first
construct the leading fluid part by
\begin{equation}\label{eq:leading}
 \Phi=\varepsilon\mathbb P_0\mathbb G^t f_0
 +\int_0^t\mathbb P_0\mathbb G^{t-\tau}
        \Gamma(\Phi,\Phi)(\tau)\,d\tau,
 \qquad \Phi=\sum_{i=1}^3\phi_iE_i.
\end{equation}
This closed three-component system retains the quadratic interactions
of all characteristic families before the exact remainder is
estimated. We separate $\phi_i=\tilde\phi_i+V_i$, where the diagonal
kinetic Burgers wave is defined by
\begin{equation}\label{kinetic-burger}
 \tilde\phi_i(t,x)=\varepsilon\langle E_i,\mathbb G_i^t f_0\rangle_\xi
 +\int_0^t\Big\langle E_i,\mathbb G_i^{t-\tau}
       \Gamma(\tilde\phi_iE_i,\tilde\phi_iE_i)(\tau)\Big\rangle_\xi
       d\tau.
\end{equation}
Here $\mathbb G_i^t$ is the individual fluid spectral branch in
\eqref{Green-fluid}, rather than the full operator $\mathbb G^t$
used in \eqref{int-kinetic-Burger}. This isolates the leading self-interaction of the $i$-th family,
while $V_i$ retains the remaining spectral contributions and
interactions between families. The waves $\tilde\phi_i$ are not solutions
of a local Burgers equation; their quantitative asymptotic comparison
with the NS Burgers waves is established in
Lemma~\ref{lemm-mathcal-K}.

The cancellation for interactions of different speeds is central to
the off-diagonal estimates. The microscopic character of the
collision source, $\mathbb P_0\Gamma=0$, gives an extra low-frequency
factor under projection onto a fluid branch. With the spectral
notation of Section~\ref{sec:pre},
\[
 \langle\Gamma(E_j,E_k),e_i(-\eta)\rangle_\xi=O(|\eta|),
 \qquad \sigma_i(\eta)+i\lambda_i\eta=O(|\eta|^2).
\]
The first property yields a spatial derivative in the convolution,
while the second gives the second-order low-frequency structure of the
characteristic derivative. For the NS diffusion waves, the corresponding
identities follow directly from the local heat and Burgers equations,
and the cancellation is obtained by integration by parts in physical
space. For the Boltzmann equation, we derive the required kernel and
source identities from the Fourier spectral representation, retaining
the exact multipliers rather than replacing the kinetic waves by local
Burgers solutions. We then use these identities in the characteristic
integration by parts for the space-time convolution, keeping the
endpoint terms. Thus the different-speed mechanism is common to both
arguments, but the kinetic case requires this spectral preparation.
The resulting cancellation is applied not only to squares of the
diagonal waves, but also to mixed sources involving off-diagonal waves
and remainders; see Proposition~\ref{prop-V-i} and
Proposition~\ref{prop-remainder}.

To make this argument compatible with initial data without spatial
differentiability, we use the particle--fluid decomposition of
\cite{[Haitao]}. The nonlinear classification in
\eqref{sec3-split-A}--\eqref{sec3-split-B} separates a non-smooth
$A$-component from a smooth $B$-component. The former is controlled
by fast space-time decay without differentiating it, while spatial
derivatives are estimated on the latter. The decomposition
\eqref{eq-V-iB} of $V_{i,B}$ isolates the diagonal spectral contribution.
Its characteristic derivative is estimated from the exact defining
equation. This provides the regular components needed for cancellation,
without differentiating the full kinetic solution.

We also retain the leading microscopic wave $V^\perp$ defined in
\eqref{leading-nonfluid}, including its linear contribution
$\varepsilon\mathbb P_1\mathbb G^t f_0$. The exact remainder
$R=f-\Phi-V^\perp$ then has zero initial data. In its equation
\eqref{eq-R-exact}, each source contains either the remainder or a
leading microscopic factor. Applying the same regularity splitting
and different-speed cancellation yields the faster-decaying
remainder in Theorem~\ref{thm:pointwise-structure}. Keeping
$V^\perp$ separate is important: the full microscopic component is
generally of order $\varepsilon$, whereas the remainder has
quadratic size.

The macro--micro calculation in Subsection~4.1 gives the quadratic
flux and frozen physical viscosity in terms of the collision operator.
We compare the kinetic leading waves with the artificial-viscosity
system \eqref{artificial}. Their transport speeds, diffusion rates,
and leading quadratic coefficients agree. This allows cancellation
of the leading terms in the difference equations. More precisely, Lemmas~\ref{lemm-mathcal-K}
and~\ref{sec4:refined-comparison} compare $\tilde\phi_i$ with
$\tilde u_i$ and $V_{i,B_1}$ with $v_i$, respectively. The
components $V_{i,B_2}$, $V_{i,A}$, and $R_i$ are estimated
separately. The comparison of artificial and physical viscosity
in \cite[Remark~8.9]{[Zeng]} then gives the approximation by the
quadratic NS solution $\omega$.

To reach the full NS system, we introduce the full-flux
artificial-viscosity solution $u^f$ in
\eqref{sec43:full-artificial}. Proposition~\ref{sec43:prop-flux-comparison}
controls the correction $u^f-u$ due to all higher-order Euler terms,
with amplitude $O(\varepsilon^3)$. The two full-flux systems have the
same Euler flux, so the artificial--physical viscosity comparison
applies to $u^f$ and $\omega^f$ as stated in
Proposition~\ref{sec43:prop-viscosity-comparison}. Combining these
comparisons with the kinetic remainder estimates proves
Theorem~\ref{cor:full-NS-comparison}. The nonlinear wave interactions
are therefore retained throughout the comparison, rather than
being discarded by linearization.

\subsection{Organization of the paper}
Section~\ref{sec:pre} collects the collision-operator estimates and
linear particle--fluid decomposition. Section~\ref{wave} constructs
the leading waves and controls the remainder, proving
Theorem~\ref{thm:pointwise-structure}.
Section~\ref{fluid-approximation} gives the formal fluid approximation
and proves Theorem~\ref{cor:full-NS-comparison} by comparison with the
quadratic and full NS systems. The convolution estimates are collected
in the Appendix.

\section{Preliminaries}\label{sec:pre}

We collect the collision estimates and the linear decomposition used in
Sections~\ref{wave} and~\ref{fluid-approximation}.

\subsection{Basic properties of $L$ and $\Gamma$}

The linearized collision operator has the decomposition
$L=-\nu+K$, where $K=-K_1+K_2$ is an integral operator and
\begin{equation}\label{Colli-freq}
 \nu(\xi)=\int_{\mathbb R^3}\int_{\mathbf S^2}
 \mathcal M(\xi_*)B(|\xi-\xi_*|,\Omega)\,d\Omega\,d\xi_*.
\end{equation}
We first recall the collision-frequency and kernel estimates; see
\cite{[Grad]} and \cite[Lemma~2.1]{[Haitao]}.

\begin{lemma}[Collision frequency and kernels]\label{lem:collision-kernels}
The operator $L$ is self-adjoint in $L^2_\xi$. There are positive
constants $\nu_{0},\nu_{1}$ such that
\begin{equation}\label{coercivity}
 \langle g,Lg\rangle_\xi\leq-\nu_{0}|\mathbb P_1g|_{L^2_\sigma}^2,
 \qquad g\in L^2_\sigma,
\end{equation}
where the left-hand side is understood as a quadratic form, and
\begin{equation}\label{nu-gamma}
 \nu_{0}\langle\xi\rangle^\gamma\leq\nu(\xi)
       \leq \nu_{1}\langle\xi\rangle^\gamma.
\end{equation}
For the integral operator $K$,
\begin{equation*}
Kf=-K_{1}f+K_{2}f=\int_{{\mathbb{R}}^{3}}-k_{1}(\xi ,\xi _{\ast })f(\xi
_{\ast })d\xi _{\ast }+\int_{{\mathbb{R}}^{3}}k_{2}(\xi ,\xi _{\ast })f(\xi
_{\ast })d\xi _{\ast }\text{,}
\end{equation*}%
the kernels $k_{1}(\xi ,\xi _{\ast })$ and $k_{2}(\xi ,\xi _{\ast })$
satisfy
\begin{equation*}
k_{1}(\xi ,\xi _{\ast })\lesssim |\xi -\xi _{\ast }|^{\gamma }\exp \left \{ -%
\frac{1}{4}\left( |\xi |^{2}+|\xi _{\ast }|^{2}\right) \right \} \text{,}
\end{equation*}%
and%
\begin{equation*}
k_{2}(\xi ,\xi _{\ast })=a\left( \xi ,\xi _{\ast },\kappa \right) \exp
\left( -\frac{(1-\kappa )}{8}\left[ \frac{\left( \left \vert \xi \right
\vert ^{2}-\left \vert \xi _{\ast }\right \vert ^{2}\right) ^{2}}{\left
\vert \xi -\xi _{\ast }\right \vert ^{2}}+\left \vert \xi -\xi _{\ast
}\right \vert ^{2}\right] \right) \text{,}
\end{equation*}%
for any $0<\kappa <1$, together with
\begin{equation*}
a(\xi ,\xi _{\ast },\kappa )\leq C_{\kappa }|\xi -\xi _{\ast }|^{-1}(1+|\xi
|+|\xi _{\ast }|)^{\gamma -1}\text{.}
\end{equation*}
\end{lemma}

The following velocity-weight gain is
\cite[Proposition~2.1]{[Haitao]}.

\begin{lemma}\label{basic-K}
For $\tau\in\mathbb R$ and $1\leq q\leq\infty$,
\begin{equation}\label{eq:K-weight-gain}
 |Kg|_{L^q_{\xi,\tau+2-\gamma}}
       \lesssim |g|_{L^q_{\xi,\tau}}.
\end{equation}
\end{lemma}

For the nonlinear term, we use the mixed-weight form of the
gain--loss estimate in \cite[Lemma~5, p.~730]{[Guo2010]}.

\begin{lemma}[Bilinear collision estimate]\label{basic-Gamma}
Let $\tau\geq0$ and $h_1,h_2\in L^\infty_{\xi,\tau}$. Then
\begin{equation}\label{eq:Gamma-mixed-weight}
 \begin{aligned}
 |\nu^{-1}\Gamma(h_1,h_2)|_{L^\infty_{\xi,\tau}}
 \lesssim{}&|h_1|_{L^\infty_{\xi,\tau}}|h_2|_{L^\infty_\xi}
       +|h_1|_{L^\infty_\xi}|h_2|_{L^\infty_{\xi,\tau}}.
 \end{aligned}
\end{equation}
The symmetrized collision operator satisfies
\begin{equation}\label{eq:collision-invariants}
 L\mathbb P_0=\mathbb P_0L=0,
 \qquad \mathbb P_0\Gamma(h_1,h_2)=0,
\end{equation}
whenever the collision moments are defined.
\end{lemma}

\subsection{Review of the linearized Boltzmann equation}
Consider
\begin{equation}\label{eq:linear0}
 \partial_th+\xi^1\partial_xh=Lh,\qquad h|_{t=0}=h_0.
\end{equation}
Its solution operator is
\begin{equation}\label{Green}
 \mathbb G^th_0=\frac1{2\pi}\int_{\mathbb R}
 e^{ix\eta}e^{(L-i\xi^1\eta)t}\widehat h_0(\eta)\,d\eta.
\end{equation}
We combine the one-dimensional spectral analysis of
\cite{[LeeLiuYu-07]} with the particle--fluid construction in
\cite[Section~3]{[Haitao]}. In the latter construction the transport
variable is restricted to $x\in\mathbb R$; the collision operator
and velocity integrations remain three-dimensional.

\subsubsection{Long wave--short wave decomposition}

The low-frequency spectral expansion is classical; see
\cite{[EP], [LeeLiuYu-07], [YangYu]}.
For the smooth dependence of the eigenvalues and the
$L_\xi^2$-valued eigenvectors on the real frequency parameter,
we use the formulation in \cite[Theorem~2.4, p.~812]{[Zhong]}.

We use $L^{-1}$ on the orthogonal complement of all five collision
invariants in the full velocity space. Its restriction to the even
class is the inverse on $\mathbb P_1L^2_\xi$, as in \cite{[LiuYu-04]}.

\begin{proposition}[Low-frequency spectrum]\label{prop:low-frequency-spectrum}
For real $|\eta|\leq\delta$, with $\delta>0$ sufficiently small,
the fluid spectrum of $L-i\xi^1\eta$ in the even class consists
of three branches. Their eigenvectors may be normalized by
\begin{equation}\label{eq:spectral-normalization}
 e_j(0)=E_j,\qquad
 e_j(-\eta)=\overline{e_j(\eta)},\qquad
 \langle e_j(-\eta),e_l(\eta)\rangle_\xi=\delta_{jl}.
\end{equation}
\begin{equation}\label{eq:finite-spectral-expansion}
 \begin{aligned}
 \sigma_j(\eta)&=-i\lambda_j\eta-A_j\eta^2
                         +\eta^3R_{\sigma,j}(\eta),
                         \\[2mm]
 e_j(\eta)&=E_j+\eta e'_j(0)+\eta^2R_{e,j}(\eta),
 \end{aligned}
\end{equation}
where $A_j>0$, $\operatorname{Re}\sigma_j(\eta)\leq-c\eta^2$, and
\begin{equation*}
\left\{\begin{array}{l}
\displaystyle A_j=\langle\mathbb P_1\xi^1E_j,
(-L)^{-1}\mathbb P_1\xi^1E_j\rangle_\xi\,,
\\ \\
\displaystyle e'_{j}(0)=\sum_{i=1}^{3}\eps_{j}^{i}E_{i}+e_{j}^{\perp}\text{,} \\
\\
\displaystyle \eps_{j}^{i}=-i\frac{\left<E_{i},\mathbb{P}_{0}\xi^{1}L^{-1}\mathbb{P}_{1}\xi^{1}E_{j}\right>_{\xi}}{\la_{i}-\la_{j}}\,,\quad j\neq i\text{,}\quad \hbox{and}\quad \eps_{j}^{j}=0\,,\\ \\
\displaystyle e_{j}^{\perp}=iL^{-1}\mathbb{P}_{1}\xi^{1}E_{j}\,.
\end{array}
\right.
\end{equation*}%
In particular, for each fixed finite order $M$ and $b\geq0$,
\[
 R_{\sigma,j}\in C^M([-\delta,\delta];\mathbb C),
 \qquad
 R_{e,j}\in
 C^M([-\delta,\delta];L_\xi^2\cap L^\infty_{\xi,b})\,,
\]
and
\[
 e'_j(0)\in L_\xi^2\cap L^\infty_{\xi,b}\,,\qquad
|L^{-1}\mathbb P_1\xi^1E_j|_{L^\infty_{\xi,b}}<\infty.
 \]

\end{proposition}

Using the cutoff assumption, write
\begin{align*}
 e^{(L-i\xi^1\eta)t}
 ={}&\chi_\delta(\eta)\sum_{j=1}^3e^{\sigma_j(\eta)t}
               |e_j(\eta)\rangle\langle e_j(-\eta)|+\chi_\delta(\eta)e^{(L-i\xi^1\eta)t}\Pi_\eta^\perp
 +(1-\chi_\delta(\eta))e^{(L-i\xi^1\eta)t}\\
 =:{}&\widehat{\mathbb G}_{L;0}
       +\widehat{\mathbb G}_{L;\perp}+\widehat{\mathbb G}_S\,,
\end{align*}
where
\[
 \Pi_\eta h=\sum_{j=1}^3e_j(\eta)
                     \langle h,e_j(-\eta)\rangle_\xi,
 \qquad \Pi_\eta^\perp=I-\Pi_\eta\,.
\]
These are complementary spectral projections; they need not be
orthogonal when $\eta\ne0$.
The corresponding physical-space operators are defined by Fourier
inversion, and
\begin{equation}\label{Green-fluid}
 \widehat{\mathbb G}_{L;0}(\eta,t)
 =\sum_{j=1}^3\widehat{\mathbb G}_j(\eta,t)
 :=\chi_\delta(\eta)\sum_{j=1}^3e^{\sigma_j(\eta)t}
                    |e_j(\eta)\rangle\langle e_j(-\eta)|.
\end{equation}
The complementary low-frequency spectrum and the spectrum on
$|\eta|\geq\delta/2$ are separated from the imaginary axis in
a left half-plane; see \cite[Section~2.2]{[Haitao]}.

The Fourier multiplier argument in \cite{[LeeLiuYu-07]} gives the
following branchwise bounds. The weighted version and the extra
decay for microscopic data are reviewed in
\cite[Proposition~2.2]{[Haitao]}; here the spatial Fourier
variable is one-dimensional.

\begin{lemma}[Fluid estimates]\label{Prop-linear-fluid}
Let $h_0\in L_x^\infty L^2_\xi$ vanish for $|x|\geq1$.
For every fixed $\alpha\in\mathbb N\cup\{0\}$,
$\hat Q\geq0$, and $N>0$, the fluid profiles in \eqref{fluid-def} satisfies
\begin{align*}
&\left|\left<E_{j}, \pa_{x}^{\alpha}\mathbb{G}_{j}^{t}h_{0}\right>_{\xi}\right|\lesssim \theta_{j}^{\alpha+1} \Vert h_{0}\Vert _{L_{x}^{\infty }L_{\xi }^{2}}%
\text{,}
\\[2mm]
&\left|\left<E_{i}, \pa_{x}^{\alpha}\mathbb{G}_{j}^{t}h_{0}\right>_{\xi}\right|\lesssim \theta_{j}^{\alpha+2} \Vert h_{0}\Vert _{L_{x}^{\infty }L_{\xi }^{2}}%
\text{,}\notag
\end{align*}%
\begin{align*}
&\left|\left< E_{j}, \pa_{x}^{\alpha}\mathbb{G}_{j}^{t}\mathbb{P}
_{1}h_{0}\right>_{\xi}\right|\lesssim \theta_{j}^{\alpha+2} \Vert h_{0}\Vert _{L_{x}^{\infty
}L_{\xi }^{2}}\,,
\\[2mm]
&\left|\left< E_{i}, \pa_{x}^{\alpha}\mathbb{G}_{j}^{t}\mathbb{P}
_{1}h_{0}\right>_{\xi}\right|\lesssim \theta_{j}^{\alpha+3} \Vert h_{0}\Vert _{L_{x}^{\infty
}L_{\xi }^{2}}\,,
\end{align*}%
for $j\neq i$ and
\begin{equation*}
\left\vert \pa_{x}^{\alpha}\mathbb{P}
_{1}\mathbb{G}_{j}^{t}h_{0}\right\vert _{L_{\xi}^{\infty }\left(\langle\xi\rangle^{\hat{Q}}\right)}\lesssim \theta_{j}^{\alpha+2}\Vert h_{0}\Vert _{L_{x}^{\infty
}L_{\xi }^{2}}\,,
\end{equation*}
\begin{equation*}
\left\vert \pa_{x}^{\alpha}\mathbb{P}
_{1}\mathbb{G}_{j}^{t}\mathbb{P}
_{1}h_{0}\right\vert _{L_{\xi}^{\infty }\left(\langle\xi\rangle^{\hat{Q}}\right)}\lesssim \theta_{j}^{\alpha+3} \Vert h_{0}\Vert _{L_{x}^{\infty
}L_{\xi }^{2}}\,.
\end{equation*}
 Constants may depend on the fixed orders and weights.
\end{lemma}

\subsubsection{Singular--regular decomposition}\label{sec:linear}

Following \cite[Section~3.1]{[Haitao]}, extract the singularity by
the Picard terms
\begin{equation}\label{eq:picard-iteration}
 \begin{cases}
 (\partial_t+\xi^1\partial_x+\nu)h^{(0)}=0,
       &h^{(0)}|_{t=0}=h_0,\\[2mm]
 (\partial_t+\xi^1\partial_x+\nu)h^{(m)}=Kh^{(m-1)},
       &h^{(m)}|_{t=0}=0,\quad m\geq1.
 \end{cases}
\end{equation}
For a fixed sufficiently large finite $\hat N$, define directly
\begin{equation}\label{eq:particle-fluid-decomposition}
 \begin{aligned}
 \mathbb G_*^th_0&=\sum_{m=0}^{\hat N}h^{(m)},\qquad
 \mathbb G_\sharp^th_0=\mathbb G^th_0-\mathbb G_*^th_0,\\
 \mathcal G^th_0&=\mathbb G_\sharp^th_0-\mathbb G_{L;0}^th_0
 =\mathbb G_{L;\perp}^th_0+\mathbb G_S^th_0-\mathbb G_*^th_0.
 \end{aligned}
\end{equation}
Thus the singular part is a finite sum, while
\begin{equation}\label{eq:regular-part-source}
 (\partial_t+\xi^1\partial_x-L)(\mathbb G_\sharp^th_0)
       =Kh^{(\hat N)},\qquad
 \mathbb G_\sharp^0h_0=0.
\end{equation}
By the Enhanced Mixture Lemma \cite[Lemma~3.2]{[Haitao]},
increasing $\hat N$ gives any fixed amount of regularity and
polynomial decay for $Kh^{(\hat N)}$.

\subsubsection{Particle--fluid estimates}

The singular estimate below is the one-dimensional form of
\cite[Proposition~3.1, (3.3)]{[Haitao]}.
For the regular estimate, we use
\cite[Proposition~2.3 and Lemma~3.2]{[Haitao]} and the
remainder construction in \cite[Sections~3.2--3.3]{[Haitao]}
with one spatial variable.

\begin{proposition}[Particle--fluid estimates]\label{PF-estimate}
Assume $h_0\in L^\infty_{\xi,Q}(L_x^\infty\cap L_x^1)$,
$Q>3/2$, and $h_0(x,\xi)=0$ for $|x|\geq1$.
Fix $\hat Q>0$, $\sigma\geq\frac{1}{2}$, and $0<\theta\leq Q$.
There is a finite $\hat N$ such that the decomposition
\eqref{eq:particle-fluid-decomposition} satisfies
\begin{equation}\label{est-G-t-ast}
 |\mathbb G_*^th_0|_{L^\infty_{\xi,Q-\theta}}
 \lesssim e^{-ct}\langle x\rangle^{-\frac{\theta}{1-\gamma}}
                   \|h_0\|_{L^\infty_{\xi,Q}L_x^\infty}
\end{equation}
and
\begin{equation}\label{est-mathcal-G-t}
 |\partial_x^a\mathcal G^th_0|_{L^\infty_{\xi,\hat Q}}
 \lesssim e^{-ct}\langle x\rangle^{-\frac{\sigma}{1-\gamma}}
                   \|h_0\|_{L^\infty_{\xi,Q}L_x^\infty},
 \qquad a=0,1,
\end{equation}
for $t\geq0$. No spatial derivative of $h_0$ is required in
\eqref{est-mathcal-G-t}. In particular, its right-hand side may
be replaced by
$(1+t+|x|)^{-\frac{\sigma}{1-\gamma}}
\|h_0\|_{L^\infty_{\xi,Q}L_x^\infty}$.

For \eqref{est-G-t-ast} alone, $Q\geq0$ and
$0\leq\theta\leq Q$ suffice. This estimate also holds for each
fixed $h^{(m)}$, with a constant depending on $m$.
One index $\hat N$ may be chosen for any finite collection of
the stated weights and derivative orders. The fluid part of
$\mathbb G_\sharp^t$ is estimated by
Lemma~\ref{Prop-linear-fluid}.
\end{proposition}

For the regular estimate, include one spatial derivative in the
Enhanced Mixture Lemma. The velocity interpolation in
\cite[Section~3.2]{[Haitao]} gives exponential decay at every fixed
output weight. Combining this estimate with the stronger spatial
weight available from \cite[Section~3.3]{[Haitao]} yields the
product decay in \eqref{est-mathcal-G-t}. The choice of $\hat N$
is made after fixing the largest weight and derivative order.

\subsubsection{Localized and inhomogeneous estimates}

We state the source-term consequences of the preceding results.
For a velocity space $\mathcal X$, define
\begin{equation}\label{eq:local-source-envelope}
 F_{\mathcal X}^{\#}[S](y)
  =\mathop{\operatorname{ess\,sup}}_{|z-y|\leq2}
                                      |S(z,\cdot)|_{\mathcal X}.
\end{equation}
Let $F\geq0$ be a dominated function satisfying, uniformly in $t$ and $y$,
\[
 |S(t,y,\cdot)|_{\mathcal X}\leq C_S F(y,t),\qquad
 \sup_{|z-y|\leq2}F(z,t)\leq C_F F(y,t).
\]
Then
\begin{equation}
\label{est-F-chi}
 F_{\mathcal X}^{\#}[S(t)](y)\leq C_S C_F F(y,t).
\end{equation}
Each profile in \eqref{fluid-def}--\eqref{coupling2}, as well as
$(1+t+|y|)^{-q}$ for each fixed $q>0$, has this local
comparability. The same holds for their finite products and finite
sums with nonnegative coefficients depending only on $t$.
Thus localization preserves these scalar upper bounds.

The characteristic vectors in \eqref{Eigen function} also give
\[
 |\langle h,E_i\rangle_\xi|
 +|\mathbb P_1h|_{L^\infty_{\xi,b}}
 \lesssim_b |h|_{L^\infty_{\xi,b}},\qquad b\geq0,
 \quad i=1,2,3,
\]
since they are polynomial multiples of $\mathcal M^{1/2}$.
Consequently, the weighted operator bounds below also hold for
their macroscopic coefficients and after applying $\mathbb P_1$
to their outputs.

\begin{proposition}[Source estimates]\label{prop:linear-source-estimates}
The following bounds hold whenever their right-hand sides are finite.

\noindent\textup{(i)} With the parameters of
Proposition~\ref{PF-estimate},
\begin{align}
 |\mathbb G_*^tS(x)|_{L^\infty_{\xi,Q-\theta}}
 &\lesssim\int_{\mathbb R}e^{-ct}
    \langle x-y\rangle^{-\frac{\theta}{1-\gamma}}
       F_{L^\infty_{\xi,Q}}^{\#}[S](y)\,dy,
 \label{eq:localized-particle-source}\\
 |\partial_x^a\mathcal G^tS(x)|_{L^\infty_{\xi,\hat Q}}
 &\lesssim\int_{\mathbb R}e^{-ct}
    \langle x-y\rangle^{-\frac{\sigma}{1-\gamma}}
       F_{L^\infty_{\xi,Q}}^{\#}[S](y)\,dy,
 \qquad a=0,1.
 \label{eq:localized-regular-source}
\end{align}
The fluid bounds in Lemma~\ref{Prop-linear-fluid} have the
same localized form with their respective $\theta_j$ profiles
and $F_{L^2_\xi}^{\#}[S]$. Scalar localization preserves
$\mathbb P_0S=0$.

\noindent\textup{(ii)} For $b\geq0$.
If $\|\nu^{-1}S(s)\|_{L_x^\infty L^\infty_{\xi,b}}$ is locally bounded, then for the fixed finite singular sum,
\begin{equation}\label{eq:particle-Duhamel-source}
 \left\|\int_0^t\mathbb G_*^{t-s}S(s)\,ds
       \right\|_{L_x^\infty L^\infty_{\xi,b}}
 \lesssim\sup_\xi\int_0^t
 e^{-c_0\nu(\xi)(t-s)}\nu(\xi)\left\Vert\nu^{-1}S(s)\right\Vert_{L_x^\infty L^\infty_{\xi,b}}\,ds.
\end{equation}
If the source has the corresponding spatial derivatives, the same
estimate applies to them through
$\partial_x^a\mathbb G_*^tS=\mathbb G_*^t\partial_x^aS$.

\noindent\textup{(iii)} If $b>\frac{3}{2}+\gamma$ and
$\mathbb P_0S=0$, then
\begin{equation}\label{eq:regular-microscopic-source}
 \|\partial_x^a\mathbb G_\sharp^tS\|_{L_x^\infty L^\infty_{\xi,b}}
 \lesssim(1+t)^{-(1+a)/2}
          \|\nu^{-1}S\|_{L_x^\infty L^\infty_{\xi,b}},
 \qquad a=0,1.
\end{equation}
No compact support or spatial derivative of $S$ is required in
\eqref{eq:regular-microscopic-source}.
\end{proposition}
\begin{proof} We give a sketch proof of this proposition.

Part~\textup{(i)} is the unit-interval localization of
Lemma~\ref{Prop-linear-fluid} and Proposition~\ref{PF-estimate};
compare the source convolutions in \cite[Section~4.1]{[Haitao]}.
Equation~\eqref{est-F-chi} allows the scalar source majorant to be
inserted directly in these convolutions.
For a spatial indicator, retain the envelope or enlarge its region
by the localization radius.

Part~\textup{(ii)} follows from the finite Picard representation
\eqref{eq:picard-iteration} and Lemma~\ref{basic-K}. The zeroth
term retains $\nu(\xi)e^{-\nu(\xi)(t-s)}$; the remaining
finitely many terms use the boundedness of $K\nu$ on
$L^\infty_{\xi,b}$. This is the damped-transport estimate used
for collision sources; compare \cite[(232)--(237)]{[Guo2010]}.
For part~\textup{(iii)}, integrate the microscopic fluid bounds
in space and use
\begin{equation}\label{eq:source-L2-conversion}
 |S|_{L^2_\xi}\lesssim|\nu^{-1}S|_{L^\infty_{\xi,b}},
 \qquad b>\frac32+\gamma,
\end{equation}
which follows from \eqref{nu-gamma} and the velocity integral in
three dimensions.
The regular contribution uses \eqref{eq:localized-regular-source} with
$Q=b-\gamma>3/2$, $\hat Q=b$, and an integrable spatial
weight. These observations give the source forms without changing
the linear decomposition. At the highest velocity weight,
parts~\textup{(ii)}--\textup{(iii)} apply with $b=\hat\beta$ and
the source norm
$\|\nu^{-1}S\|_{L_x^\infty L^\infty_{\xi,\hat\beta}}$.
\end{proof}

\section{Leading waves and remainder of the 1D Boltzmann equation}\label{wave}
In this section, we construct and analyze the leading fluid waves $\Phi$ and the leading non-fluid waves $V^{\perp}$ associated with the perturbation $f$ in \eqref{Linearized}.
The exact remainder $R=f-\Phi-V^\perp$ is further decomposed into its fluid and non-fluid parts. According to  Proposition \ref{PF-estimate}, the Green function $\mathbb{G}^{t}$ for the linearized Boltzmann equation can be split into the singular part $\mathbb{G}_{\ast}^{t}$ and the regular part $\mathbb{G}_{\sharp}^{t}$ , we introduce a classification procedure that decomposes each integral expression into non-smooth and smooth components. Now, we introduce this classification procedure as follows :
\\
Consider an integral expression of the form :
\[
I=\mathbb{G}^{t}I_{0}+\int_{0}^{t}\mathbb{G}^{t-\tau}\Gamma(F,H)(\tau)d\tau\,.
\]
Suppose that $F$ and $H$ admit the decompositions :
$$F=F_A+F_B\,, \qquad H=H_A+H_B\,,$$
where $F_{A}$, $H_{A}$ be their non-smooth part and $F_{B}$, $H_{B}$ be their smooth part.
Then the nonlinear term $\Gamma(F,H)$ can be decomposed as
\[
 \Gamma(F,H)=\Gamma(F_A,H)+\Gamma(F_B,H_A)+\Gamma(F_B,H_B)\,.
\]
Therefore, the term $I$ can be classified to  the non-smooth part $I_{A}$ and the smooth part $I_{B}$ as follows :
\[
I=I_{A}+I_{B}\,,
\]
where the non-smooth part of $I_{A}$ can be defined by
\begin{equation}\label{sec3-split-A}
 I_{A}=\mathbb{G}^{t}_{\ast}I_{0}+\int_0^t\mathbb G_*^{t-\tau}
 [\Gamma(F_A,H)+\Gamma(F_B,H_A)](\tau)\,d\tau,
\end{equation}
while the smooth component $I_{B}$ can be defined by
\begin{equation}\label{sec3-split-B}
 I_{B}=\mathbb{G}_{\sharp}^{t}I_{0}+\int_0^t\mathbb G_\sharp^{t-\tau}\Gamma(F,H)(\tau)\,d\tau
 +\int_0^t\mathbb G_*^{t-\tau}\Gamma(F_B,H_B)(\tau)\,d\tau.
\end{equation}
This classification procedure will be applied to the leading waves and the exact remainder in Proposition \ref{prop-V-i}-\ref{prop-remainder}.

    \subsection{Leading waves}
    We define the leading fluid waves by
    \begin{equation}
    \label{leading fluid}
     \Phi=\varepsilon \mathbb{P}_{0}\mathbb{G}^{t}f_{0}+\int_{0}^{t}\mathbb{P}_{0}\mathbb{G}^{t-\tau}\Gamma\left(\Phi,\Phi\right)(\tau)d\tau .
    \end{equation}
    The waves $\Phi$ can be represented as
    \begin{equation*}
    \Phi=\sum_{i=1}^{3}\phi_{i}E_{i},
    \end{equation*}
    where
    \begin{equation*}
    \phi_{i}(t,x)=\varepsilon \left\langle E_{i},\mathbb{G}^{t}f_{0}\right\rangle_{\xi}+\int_{0}^{t}\left\langle E_{i},\mathbb{G}^{t-\tau}\sum_{j,k=1}^{3}\Gamma\left(\phi_{j}E_{j},\phi_{k}E_{k}\right)(\tau)\right\rangle_{\xi}d\tau
    \end{equation*}
and define the leading non-fluid waves by
 \begin{equation}
    \label{leading-nonfluid}
     V^{\perp}=\varepsilon \mathbb{P}_{1}\mathbb{G}^{t}f_{0}+\int_{0}^{t}\mathbb{P}_{1}\mathbb{G}^{t-\tau}\Gamma\left(\Phi,\Phi\right)(\tau)d\tau .
    \end{equation}
  We decompose $\phi_{i}$ as
  $$\phi_{i}(t,x)=\tilde{\phi}_{i}(t,x)+V_{i}(t,x)\,,$$
   where the main part of the fluid waves $\tilde{\phi}_{i}(t,x)$ is the so called ``kinetic Burgers waves'':
    \begin{equation}
    \tilde{\phi}_{i}(t,x)=\varepsilon \langle E_{i},\mathbb{G}_{i}^{t}f_{0}\rangle_{\xi}+\int_{0}^{t}\Big\langle E_{i},\mathbb{G}_{i}^{t-\tau}\Gamma\left(\tilde{\phi}_{i}E_{i},\tilde{\phi}_{i}E_{i}\right)(\tau)\Big\rangle_{\xi}d\tau\,.
    \label{expression tilde-phi-i}
    \end{equation}
    The off-diagonal part $V_{i}$ then satisfies
 \begin{align}
    \quad V_{i}(t,x)&=\varepsilon \langle E_{i},\mathbb{G}_{*}^{t}f_{0}\rangle_{\xi}+\varepsilon \langle E_{i},\left(\mathbb{G}^{t}-\mathbb{G}_{*}^{t}-\mathbb{G}_{i}^{t}\right)f_{0}\rangle_{\xi}\notag
    \\
    &\quad+\sum_{\substack{1\leq j,k\leq 3 \\ (j,k)\neq (i,i)}}\int_{0}^{t}\Big\langle E_{i},\mathbb{G}^{t-\tau}\Gamma\left(\tilde{\phi}_{j}E_{j},\tilde{\phi}_{k}E_{k}\right)(\tau)\Big\rangle_{\xi}d\tau
     \notag
   \\
    &\quad +\int_{0}^{t}\Big\langle E_{i},\left(\mathbb{G}^{t-\tau}-\mathbb{G}_{i}^{t-\tau}\right)\Gamma\left(\tilde{\phi}_{i}E_{i},\tilde{\phi}_{i}E_{i}\right)(\tau)\Big\rangle_{\xi}d\tau\notag
    \\
     &\quad +\sum_{1\leq j,k\leq 3}2\int_{0}^{t}\Big\langle E_{i},\mathbb{G}^{t-\tau}\Gamma\left(\tilde{\phi}_{j}E_{j},V_{k}E_{k}\right)(\tau)\Big\rangle_{\xi}d\tau\notag
     \\
     &\quad+\sum_{1\leq j,k \leq 3}\int_{0}^{t}\Big\langle E_{i},\mathbb{G}^{t-\tau}\Gamma\left(V_{j}E_{j},V_{k}E_{k}\right)(\tau)\Big\rangle_{\xi}d\tau
     \label{expression V-i}
    \end{align}

    \begin{proposition}
    \label{lemm tilde-phi-ast,sharp}
    For each fixed finite integer $s\in \mathbb{N}\cup\left\{0\right\}$ and $\varepsilon>0$
    sufficiently small after the fixed choices above, $\tilde\phi_i$
    of \eqref{expression tilde-phi-i} satisfies
     \begin{equation}
     \label{est-tilde-phi-i}
         \left\vert \pa^{s}_{x}\tilde{\phi}_{i}(t,x)\right\vert\lesssim_{s}
         \varepsilon\theta^{s+1}_{i}\,.
     \end{equation}
    \end{proposition}
    \begin{proof} The convolution estimates used in this proof are collected in Table \ref{table-1} in the Appendix.
    For any fixed $s\in\mathbb{N}\cup\left\{0\right\}$, we take the q-th $x$-derivative of $\tilde{\phi}_{i}$ for $0\leq q\leq s$ to obtain
    \begin{align*}
    \partial_{x}^{q}\tilde{\phi}_{i}&=\varepsilon\langle E_{i},\partial_{x}^{q}\mathbb{G}^{t}_{i}f_{0}\rangle_{\xi}+\int_{0}^{\frac{t}{2}}\Bigg\langle E_{i},\partial_{x}^{q}\mathbb{G}^{t-\tau}_{i}\Gamma\left(\tilde{\phi}_{i}E_{i},\tilde{\phi}_{i}E_{i}\right)(\tau)\Bigg\rangle_{\xi}d\tau
    \\
    &\quad +\int_{\frac{t}{2}}^{t}\Bigg\langle E_{i},\mathbb{G}^{t-\tau}_{i}\left[\partial_{x}^{q}\Gamma\left(\tilde{\phi}_{i}E_{i},\tilde{\phi}_{i}E_{i}\right)(\tau)\right]\Bigg\rangle_{\xi}d\tau
    \\
    &\equiv B_{0}+B_{1}+B_{2}\,.
    \end{align*}
    Clearly,
    \begin{equation*}
        \vert B_{0}\vert \lesssim \varepsilon \theta_{i}^{1+q}\,.
    \end{equation*}
 Based on the above estimates, we assume the ansatz :
    \begin{equation}
    \label{asatz-tilde-phi}
    \left\vert \partial_{x}^{q}\tilde{\phi}_{i}\right\vert
    \leq  \varepsilon M\theta_{i}^{1+q}\,, \qquad 0\leq q\leq s\,.
    \end{equation}
   For $B_{1}$, we have
    \begin{align*}
    \vert B_{1}\vert \lesssim
    \varepsilon^{2}M^{2}\left[\int_{0}^{\frac{t}{2}}\int_{-\infty}^{\infty}
    \theta_{i}^{2+q}(t-\tau,x-y)\theta_{i}^{2}(\tau,y)dyd\tau\right]
    \lesssim \varepsilon^{2}M^{2}\theta_{i}^{1+q}\,.
    \end{align*}
   As for $B_{2}$, one yields
    \begin{align*}
    \vert B_{2}\vert\lesssim  \varepsilon^{2} M^{2}\int_{\frac{t}{2}}^{t}\int_{-\infty}^{\infty}\theta_{i}^{2}(t-\tau,x-y)\theta_{i}^{2+q}(\tau,y)dyd\tau
    \lesssim \varepsilon^{2}M^{2}\theta_{i}^{1+q}\,.
    \end{align*}
Combining the component estimates, the known forcing has amplitude
$C_{0}\varepsilon$, while the terms containing the $\tilde\phi_{i}$ have amplitude at most
$C_{1}\varepsilon^2M^{2}$, with the corresponding profiles in \eqref{asatz-tilde-phi}.
Here $C_{0}\geq1$ and $C_{1}$ depend only on the fixed parameters $s$ and the
initial-data norm, not on $M$, $\varepsilon$, or the terminal time.
The output coefficient is therefore bounded by
\[
C_{0}\varepsilon+C_{1}\varepsilon^{2}M^{2}=\varepsilon \left(C_{0}+C_{1}\varepsilon M^{2}\right)
\]
Choose $M\geq 4 C_{0}$ sufficiently large and $0<\varepsilon\leq1$ sufficiently small that
\[
C_{1}\varepsilon M^{2}\leq\frac14 M.
\]
The output is then bounded by $\frac{1}{2}M\varepsilon$, completing the proof.
    \end{proof}
    \begin{proposition}
    \label{prop-V-i}
    For $\varepsilon>0$ sufficiently small, $V_{i}$ in \eqref{expression V-i} can be decomposed into a non-smooth part $V_{i,A}$ and a smooth part $V_{i,B}$ satisfying
    \begin{equation}
    \left\vert V_{i,A}(t,x)\right\vert \lesssim \varepsilon(1+t+\vert x\vert)^{-\frac{N_{0}}{1-\gamma}}\quad
 \text{for any } \ N_{0}>2N(1-\gamma),
    \label{est-V-i-A}
    \end{equation}
     \begin{equation}
    \left\vert V_{i,B}(t,x)\right\vert \lesssim \varepsilon \Bigg(\psi_{i}^{\frac{3}{2}}+\sum_{ k\neq i}\tilde{\psi}_{k}^{\frac{3}{2}}\Bigg),
    \label{est-V-i-B}
    \end{equation}
    \begin{equation}
    \left\vert \partial_{x}V_{i,B}(t,x)\right\vert\lesssim\varepsilon (1+t)^{-\frac{1}{2}}\Bigg(\psi_{i}^{\frac{3}{2}}+\sum_{ k\neq i}\tilde{\psi}_{k}^{\frac{3}{2}}\Bigg)\,.
    \label{est-der-V-i-B}
    \end{equation}
    \end{proposition}

\begin{proof}
    The convolution estimates used in this proof are collected in Tables \ref{table-1}--\ref{table-2} in the Appendix.
    Write $V_{i}=V_{i,A}+V_{i,B}$. For convenience, throughout this paper denoting $$\Phi_{A}=\sum_{m=1}^{3}V_{m,A}E_{m}\,,\qquad \Phi_{B}=\sum_{m=1}^{3}(\tilde\phi_{m}+V_{m,B})E_{m}\,.$$ Applying
\eqref{sec3-split-A}--\eqref{sec3-split-B} to each term in
\eqref{expression V-i}, we define
\begin{align}
\label{eq-V-iA}
 V_{i,A}&=\varepsilon\langle E_{i},\mathbb G_*^{t}f_{0}\rangle_{\xi}+\int_{0}^{t}\left\langle E_{i},\mathbb{G}_{\ast}^{t-\tau}g_{A}(\tau)\right\rangle_{\xi}d\tau
 =I_{0}+I_{1}\,,
\end{align}
where
\begin{equation}
\label{eq-gA}
g_{A}(t)=2\Gamma\left(\Phi_{A},\Phi_{B}\right)+\Gamma\left(\Phi_{A},\Phi_{A}\right)\,,
\end{equation}
and we further split the smooth fluid component as
\begin{equation}
\label{eq-V-iB}
V_{i,B}=V_{i,B_{1}}+V_{i,B_{2}}\,,
\end{equation}
\begin{align}
V_{i,B_{1}}&=
\int_{0}^{t}\left\langle E_{i},\mathbb{G}_{i}^{t-\tau}\Bigg(\sum_{\substack{1\leq j,k\leq 3\\ j\neq k}}\Gamma\left(\tilde\phi_{j}E_{j},\tilde\phi_{k}E_{k}\right)\Bigg)(\tau)\right\rangle_{\xi}d\tau\notag
\\
&\quad+\sum_{\substack{1\leq j\leq 3\\j\neq i}}\int_{0}^{t}\left\langle E_{i},\mathbb{G}_{i}^{t-\tau}\Gamma\left(\tilde\phi_{j}E_{j},\tilde\phi_{j}E_{j}\right)(\tau)\right\rangle_{\xi}d\tau\notag
\notag
\\
&\quad+\int_{0}^{t}\left\langle E_{i},\mathbb{G}_{i}^{t-\tau}2\Gamma\left(\tilde\phi_{i}E_{i},V_{i,B}E_{i}\right)\Bigg)(\tau)\right\rangle_{\xi}d\tau\notag
\\
&\quad+\int_{0}^{t}\left\langle E_{i},\mathbb{G}_{i}^{t-\tau}\Bigg(\sum_{\substack{1\leq j,k\leq 3\\ j\neq k}}2\Gamma\left(\tilde\phi_{j}E_{j},V_{k,B}E_{k}\right)\Bigg)(\tau)\right\rangle_{\xi}d\tau\notag
\\
&\quad+\sum_{\substack{1\leq j\leq 3 \\j\neq i}}\int_{0}^{t}\left\langle E_{i},\mathbb{G}_{i}^{t-\tau}2\Gamma\left(\tilde\phi_{j}E_{j},V_{j,B}E_{j}\right)(\tau)\right\rangle_{\xi}d\tau\notag
\\
&\quad+\sum_{1\leq j,k\leq 3}\int_{0}^{t}\left\langle E_{i},\mathbb{G}_{i}^{t-\tau}\Gamma\left(V_{j,B}E_{j},V_{k,B}E_{k}\right)(\tau)\right\rangle_{\xi}d\tau
+\int_{0}^{t}\left\langle E_{i},\mathbb{G}_{i}^{t-\tau}g_{A}(\tau)\right\rangle_{\xi}d\tau
\notag\\
&=J_{1}+\cdots+J_{7}\,,
\label{eq-V-i-B1}
\end{align}
and
\begin{align}
V_{i,B_{2}}&=\varepsilon \left\langle E_{i},\left(\mathbb{G}_{\sharp}^{t}-\mathbb{G}_{i}^{t}\right)f_{0}\right\rangle_{\xi}+\sum_{1\leq j,k\leq 3}\int_{0}^{t}\left\langle E_{i},\left(\mathbb{G}_{\ast}^{t-\tau}+\mathcal{G}^{t-\tau}\right)\Gamma\left(\tilde\phi_{j}E_{j},\tilde\phi_{k}E_{k}\right)(\tau)\right\rangle_{\xi}d\tau\notag
\\
&\quad+\sum_{\substack{1\leq l\leq 3 \\l\neq i}}\int_{0}^{t}\left\langle E_{i},\mathbb{G}_{l}^{t-\tau}\Bigg(\sum_{\substack{1\leq j,k\leq 3\\ j\neq k}}\Gamma\left(\tilde\phi_{j}E_{j},\tilde\phi_{k}E_{k}\right)\Bigg)(\tau)\right\rangle_{\xi}d\tau\notag
\\
&\quad+\sum_{\substack{1\leq l\leq 3 \\l\neq i}}\int_{0}^{t}\left\langle E_{i},\mathbb{G}_{l}^{t-\tau}\Gamma\left(\tilde\phi_{l}E_{l},\tilde\phi_{l}E_{l}\right)(\tau)\right\rangle_{\xi}d\tau\notag
\\
&\quad+\sum_{\substack{1\leq l\leq 3 \\l\neq i}}\int_{0}^{t}\left\langle E_{i},\mathbb{G}_{l}^{t-\tau}\Bigg(\sum_{\substack{1\leq k\leq 3\\k\neq l}}\Gamma\left(\tilde\phi_{k}E_{k},\tilde\phi_{k}E_{k}\right)(\tau)\Bigg)\right\rangle_{\xi}d\tau\notag
\\
&\quad+\int_{0}^{t}\left\langle E_{i},\Bigg[\sum_{\substack{1\leq l\leq 3 \\l\neq i}}(\mathbb{G}_{l}^{t-\tau})+\mathcal{G}^{t-\tau}\Bigg]g_{A}(\tau)\right\rangle_{\xi}d\tau+\int_{0}^{t}\left\langle E_{i},\left(\mathbb{G}_{\ast}^{t-\tau}+\mathcal{G}^{t-\tau}\right)g_{B}(\tau)\right\rangle_{\xi}d\tau
\notag\\
&\quad+\sum_{\substack{1\leq l\leq 3 \\l\neq i}}\int_{0}^{t}\left\langle E_{i},\mathbb{G}_{l}^{t-\tau}2\Gamma\left(\tilde\phi_{l}E_{l},V_{l,B}E_{l}\right)(\tau)\right\rangle_{\xi}d\tau\notag
\\
&\quad+\sum_{\substack{1\leq l\leq 3 \\l\neq i}}\int_{0}^{t}\left\langle E_{i},\mathbb{G}_{l}^{t-\tau}\Bigg(\sum_{\substack{1\leq j,k\leq 3\\ j\neq k}}2\Gamma\left(\tilde\phi_{j}E_{j},V_{k,B}E_{k}\right)\Bigg)(\tau)\right\rangle_{\xi}d\tau\notag
\\
&\quad+\sum_{\substack{1\leq l\leq 3 \\l\neq i}}\int_{0}^{t}\left\langle E_{i},\mathbb{G}_{l}^{t-\tau}\Bigg(\sum_{\substack{1\leq k\leq 3\\k\neq l}}2\Gamma\left(\tilde\phi_{k}E_{k},V_{k,B}E_{k}\right)(\tau)\Bigg)\right\rangle_{\xi}d\tau\notag
\\
&\quad+\sum_{\substack{1\leq l\leq 3 \\l\neq i}}\int_{0}^{t}\left\langle E_{i},\mathbb{G}_{l}^{t-\tau}\Bigg(\sum_{1\leq j,k\leq 3}\Gamma\left(V_{j,B}E_{j},V_{k,B}E_{k}\right)(\tau)\Bigg)\right\rangle_{\xi}d\tau
\notag\\
&=L_{1}+\cdots +L_{11}\,,
\label{eq-V-i-B2}
\end{align}
where
\begin{equation}
\label{eq-gB}
g_{B}(t)=\sum_{1\leq j,k\leq 3}\left[2\Gamma\left(\tilde\phi_{j}E_{j},V_{k,B}E_{k}\right)+\Gamma\left(V_{j,B}E_{j},V_{k,B}E_{k}\right)\right]\,.
\end{equation}
The first part $V_{i,B_{1}}$ contains all contributions involving the diagonal spectral kernel $\mathbb G_i$.

    Throughout this proof, $s$ and $n$ denote differentiation orders with, $s\in \left\{0,1\right\}$ and $n\in \mathbb{N}\cup \left\{0\right\}$. We first estimate the terms involving only the kinetic Burgers waves, beginning with the non-smooth component $V_{i,A}$. In view of \eqref{eq-V-iA}--\eqref{eq-V-i-B2}, it follows that
 $$|I_{0}|\lesssim \varepsilon(1+t+\vert x\vert)^{-\frac{N_{0}}{1-\gamma}}\quad
 \text{for any } \ N_{0}>2N(1-\gamma)\,.$$
 We next turn to the first component $V_{i,B_{1}}$ of the smooth part $V_{i,B}$.
 For $J_{1}$, using these inequality for $j\ne k$, we obtain
\begin{align}
    \theta^1_j\theta^1_k
    &\lesssim (1+t)^{-1}\left(\psi_j^{2N}+\psi_k^{2N}\right),
       \qquad j\ne k\,,
    \label{est-theta-1-j-times-theta-1-k}
    \end{align}
and hence we have
    \begin{equation*}
    \vert \partial_{x}^{s}J_{1}\vert \lesssim  \varepsilon^{2}\left[\theta_{i}^{2+s}\right]\ast_{x,t}\Bigg[\sum_{l=1}^{3}(1+t)^{-1}\psi_l^{2N}\Bigg]\lesssim \varepsilon^{2}(1+t)^{-\frac{s}{2}}\Bigg(\psi_{i}^{\frac{3}{2}}+\sum_{j\neq i}\tilde\psi_{j}^{\frac{3}{2}}\Bigg).
    \end{equation*}
For $J_{2}$, in view of the expansion of $\widehat{\mathbb{G}}_{i}$ in \eqref{Green-fluid}, we have
\[
 \widehat{\mathbb G_i^t h}(\eta)
 =\chi_{\delta}(\eta)e^{\sigma_i(\eta)t}e_i(\eta)
   \langle\hat h(\eta),e_i(-\eta)\rangle_\xi.
\]
Thus, we define the kernel
\begin{equation}
\label{def-H-ij}
 H_{ij}(t,x)=\frac1{2\pi}\int_{\mathbb R}e^{ix\eta+\sigma_i(\eta)t}
 \frac{\chi_{\delta}(\eta)\langle e_i(\eta),E_i\rangle_\xi
       \langle\Gamma(E_j,E_j),e_i(-\eta)\rangle_\xi}{i\eta}\,d\eta\,.
\end{equation}
By Proposition \ref{prop:low-frequency-spectrum} and the Fourier multiplier argument used in Lemma \ref{Prop-linear-fluid}, the kernel satisfies $$|\partial_x^nH_{ij}|\lesssim\theta_i^{1+n}\,.$$
We obtain the following exact identity for each component of $J_{2}$:
\begin{equation}
\label{eq-J-2-ij}
J_{2}^{ij}= \int_{0}^{t}\langle E_i,\mathbb G_i^{t-\tau}
     \Gamma(\tilde\phi_jE_j,\tilde\phi_jE_j)(\tau)\rangle_\xi
 d\tau=\int_{0}^{t}\partial_xH_{ij}(t-\tau)*_x(\tilde\phi_j)^2(\tau)d\tau.
\end{equation}
This follows from the fact that
$\Gamma(E_j,E_j)$ is microscopic.  Moreover, there exists a kernel $H_{1}$ such that
$$
 (\partial_t+\lambda_i\partial_x)H_{ij}=\partial_x H_1,
 \qquad |\partial_x^nH_1|\lesssim\theta_i^{2+n}\,,
 $$
since
$$\widehat{(\partial_{t}+\lambda_{i}\partial_{x})H_{ij}}=(\sigma_{i}(\eta)+i\lambda_{i}\eta)\widehat{H_{ij}}\,, \qquad\sigma_i(\eta)+i\lambda_i\eta=-A_{i}\eta^{2}+\eta^{3}R_{\sigma,i}(\eta)\,.$$

Fix $j\ne i$. For $t\geq4$, we split the component $J_{2}^{ij}$ as $$J_2^{ij}=\int_{0}^{\sqrt{t}}\cdots+\int_{\sqrt{t}}^{t}\cdots\equiv J^{ij}_{2\left\{[0,\sqrt{t}]\right\}}+J^{ij}_{2\left\{[\sqrt{t},t]\right\}}\,.$$ Its early part satisfies
\[
 J^{ij}_{2\left\{[0,\sqrt{t}]\right\}}\lesssim\varepsilon^2
 \int_0^{\sqrt t}\theta_i^2(t-\tau)*_x\theta_j^2(\tau)\,d\tau
 \lesssim\varepsilon^2\theta_i^{3/2}.
\]
For the remaining part, differentiating the spatial convolution with respect to $\tau$ gives
\begin{align*}
 &(\lambda_j-\lambda_i)\partial_xH_{ij}(t-\tau)*_x(\tilde\phi_j)^2(\tau)
 \\
 &=-\partial_\tau\big[H_{ij}(t-\tau)*_x(\tilde\phi_j)^2(\tau)\big]-[(\partial_t+\lambda_i\partial_x)H_{ij}](t-\tau)
       *_x(\tilde\phi_j)^2(\tau)\\
 &\quad+H_{ij}(t-\tau)*_x
       [(\partial_\tau+\lambda_j\partial_x)(\tilde\phi_j)^2](\tau).
\end{align*}
The scalar defining equation \eqref{expression tilde-phi-i} also gives,
using its actual spectral multipliers,
\begin{align}
 \widehat{(\partial_t+\lambda_j\partial_x)\tilde\phi_j}
 ={}&(\sigma_j(\eta)+i\lambda_j\eta)\widehat{\tilde\phi_j}\notag\\
 &+\chi_{\delta}(\eta)\langle e_j(\eta),E_j\rangle_\xi
   \langle\Gamma(E_j,E_j),e_j(-\eta)\rangle_\xi
   \widehat{(\tilde\phi_j)^2}.
   \label{eq-tilde-j-par t+lamda j par x}
\end{align}
The two multipliers are of orders two and one in $\eta$, respectively. Applying the previously established spatial derivative estimates to these localized multipliers, we obtain
\begin{equation*}
|\partial_x^n(\partial_t+\lambda_j\partial_x)\tilde\phi_j|
 \lesssim(\varepsilon+\varepsilon^2)\theta_j^{3+n}.
\end{equation*}
Consequently,
\begin{align}
\label{tilde-phi-j-cancell-1}
 |(\tilde\phi_j)^2|&\lesssim\varepsilon^2\theta_j^2,&
 |\partial_x(\tilde\phi_j)^2|&\lesssim\varepsilon^2\theta_j^3,&
 |(\partial_t+\lambda_j\partial_x)(\tilde\phi_j)^2|
 &\lesssim\varepsilon^2\theta_j^4.
\end{align}

Integrating this identity
from $\sqrt t$ to $t$, and moving one spatial derivative from its second
term to the source, we obtain
\begin{align*}
 \left\vert(\lambda_{j}-\lambda_{i})J^{ij}_{2\{[\sqrt t,t]\}}\right\vert
 &=\left\vert-\big[H_{ij}(t-\tau)*_x(\tilde{\phi}_{j})^{2}(\tau)\big]_{\sqrt t}^{t}
   -\int_{\sqrt t}^t(\partial_{t}+\lambda_{i}\partial_{x})H_{ij}(t-\tau)*_x(\tilde\phi_{j})^{2}(\tau)\,d\tau\right.\\
 &\left.\quad+\int_{\sqrt t}^t H_{ij}(t-\tau)*_x
              (\partial_\tau+\lambda_j\partial_x)(\tilde\phi_j)^2(\tau)\,d\tau\right\vert
              \\
 &\lesssim\varepsilon^2\Big[\theta_j^2
   +\theta_i^1(t-\sqrt t)*_x\theta_j^2(\sqrt t)+\int_{\sqrt t}^t\theta_i^2(t-\tau)*_x\theta_j^3(\tau)\,d\tau\\
 &+\int_{\sqrt t}^t\theta_i^1(t-\tau)*_x\theta_j^4(\tau)\,d\tau\Big]\\
 &\lesssim\varepsilon^2
       \big(\psi_i^{3/2}+\tilde\psi_j^{3/2}\big).
\end{align*}
The above argument improves the estimate of $J_{2}^{ij}$, and we refer to this procedure as the `` Cancellation Mechanism ". We summarize this procedure in Remark \ref{remark-cancellation mechanism}.
 For $\partial_{x}J_{2}$, applying this Cancellation Mechanism gives
\begin{equation}
\label{est-derivative-tilde-J}
\left\vert\partial_{x}J_{2}\right\vert\lesssim \varepsilon^{2}(1+t)^{-\frac{1}{2}}\Bigg(\psi_{i}^{\frac{3}{2}}+\sum_{ k\neq i}\tilde{\psi}_{k}^{\frac{3}{2}}\Bigg).
\end{equation}
Here the $x$-derivative is placed on the Green function over $[0,\frac{t}{2}]$ and transferred to the source over $[\frac{t}{2},t]$. We next turn to the second component $V_{i,B_{2}}$. By Lemma \ref{Prop-linear-fluid}, Proposition \ref{PF-estimate}, and \eqref{est-theta-1-j-times-theta-1-k}, we obtain
\[
\vert \partial_{x}^{s}L_{1}\vert\lesssim \varepsilon \Bigg[\sum_{j\neq i}(\theta_{j}^{2+s})+(1+t+\vert x\vert)^{-\frac{N_{0}}{1-\gamma}}\Bigg]\,, \qquad \text{for any }N_{0}>2N(1-\gamma)\,,
\]
For $L_{2}$, Lemma~\ref{basic-Gamma} and the Gaussian decay of
$E_l,E_m$ give, for every fixed $b\geq0$,
\begin{equation}\label{eq:fixed-microscopic-vectors}
 |\Gamma(E_l,E_m)|_{L^\infty_{\xi,b}}<\infty.
\end{equation}
Taking $Q=\beta+N_0$, $\theta=N_0$, $\hat Q=\beta$, and
$\sigma=\max\{N_0,1/2\}$ in Proposition~\ref{prop:linear-source-estimates}\textup{(i)},
we transfer $\partial_x^s$ to the source by translation invariance.
Using \eqref{est-F-chi}, we obtain
\begin{align*}
\vert \partial_{x}^{s}L_{2}\vert
&\lesssim \varepsilon^2\int_0^t\int_{\mathbb R}
e^{-c(t-\tau)}\langle x-y\rangle^{-\frac{N_0}{1-\gamma}}
\sum_{m=1}^3\theta_m^{2+s}(y,\tau)\,dy\,d\tau
\\
&\lesssim \varepsilon^2\sum_{m=1}^3\theta_m^{2+s}(x,t)\,.
\end{align*}
Here $N_0/(1-\gamma)>2N>1$ preserves the spatial tails, and the
exponential factor absorbs the fixed time weights and characteristic
displacements. For $L_{3}$, we obtain
\[
\vert \partial_{x}^{s}L_{3}\vert\lesssim\varepsilon^{2}(1+t)^{-\frac{1+s}{2}}\Bigg[\sum_{m=1}^{3}\psi_{m}^{\frac{3}{2}}\Bigg]\,.
\]
For each component of $L_{4}$, we transfer the extra factor $\eta$, arising from $\langle e_{l}(\eta),E_{i}\rangle_{\xi}$ and $\partial_{x}^{s}\mathbb{G}_{l}^{t-\tau}$, to the source on the time interval $[\frac{t}{2},t]$. Thus, for $l\neq i$,
\begin{align*}
\vert \partial_{x}^{s}L_{4}^{l}\vert&=\left\vert \partial_{x}^{s}\int_{0}^{t}\left\langle E_{i},\mathbb{G}_{l}^{t-\tau}\Gamma\left(\tilde\phi_{l}E_{l},\tilde\phi_{l}E_{l}\right)(\tau)\right\rangle_{\xi}d\tau\right\vert
\\
&\lesssim \varepsilon^2\left[\int_{0}^{\frac{t}{2}}\theta_{l}^{3+s}(t-\tau)\ast_{x}\theta_{l}^{2}(\tau)d\tau+\int_{\frac{t}{2}}^{t}\theta_{l}^{2}(t-\tau)\ast_{x}\theta_{l}^{3+s}(\tau)d\tau\right]
\\
&\lesssim \varepsilon^{2}\theta_{l}^{2+s}\,.
\end{align*}

The structure of $L_{5}$ is the same as that of $J_{2}$, except for an additional derivative on the Green's function, which arises from projecting $\mathbb{G}^{t}_{i}$ onto a different direction $E_{j}$. Therefore, applying the Cancellation Mechanism, with the $x$-derivative placed on the Green's function over $[0,\frac{t}{2}]$ and on the source over $[\frac{t}{2},t]$, we obtain
\begin{equation*}
|\partial_{x}^{s}L_{5}|\lesssim \varepsilon^{2}(1+t)^{-\frac{1+s}{2}}\Bigg(\sum_{m=1}^{3}\psi_{m}^{\frac{3}{2}}\Bigg)\,.
\end{equation*}

Motivated by these estimates, we assume that the off-diagonal component $V_{i}$ satisfies the following ansatz:
\begin{equation}
\label{asatz-V-iA}
\vert V_{i,A}\vert \leq M \varepsilon(1+t+\vert x\vert)^{-\frac{N_{0}}{1-\gamma}}\,, \qquad\text{for any } N_{0}>2N(1-\gamma)\,,
\end{equation}
\begin{equation}
\label{asatz-V-i-B1}
\vert \partial_{x}^{s}V_{i,B_{1}}\vert \leq M\varepsilon(1+t)^{-\frac{s}{2}}\Bigg(\psi_{i}^{\frac{3}{2}}+\sum_{k\neq i}\tilde\psi_{k}^{\frac{3}{2}}\Bigg)\,,
\end{equation}
\begin{equation}
\label{asatz-V-i-B2}
\vert\partial_{x}^{s} V_{i,B_{2}}\vert \leq M\varepsilon\left[\sum_{m=1}^{3}\theta_{m}^{2+s}+(1+t)^{-\frac{1+s}{2}}\psi_{m}^{\frac{3}{2}}\right]\,,
\end{equation}
where the constant $M$ will be chosen later. Moreover, we have the inequality
\begin{equation}
\label{}
(1+t)^{-\frac{1}{4}}\psi_{m}^{\frac{3}{2}}\lesssim \tilde\psi_{m}^{\frac{3}{2}}\,, \qquad \text{for any } m\in\left\{1,2,3\right\}\,,
\end{equation}
Hence, the ansatz \eqref{asatz-V-i-B1}-\eqref{asatz-V-i-B2} implies \eqref{est-V-i-B}--\eqref{est-der-V-i-B}.
Under these assumptions, we return to the estimates of $V_{i,A}$ through $V_{i,B_{2}}$.
In $V_{i,A}$, applying Proposition~\ref{prop:linear-source-estimates}\textup{(i)}
and \eqref{est-F-chi} to $g_A$ in \eqref{eq-gA}, the same exponentially
damped convolution preserves $(1+t+|x|)^{-\frac{N_0}{1-\gamma}}$.
Thus, using the ansatz \eqref{asatz-V-iA}--\eqref{asatz-V-i-B2}, we obtain
$$
|I_{1}|\lesssim M^{2}\varepsilon^{2}(1+t+\vert x\vert)^{-\frac{N_{0}}{1-\gamma}}\quad
 \text{for any } \ N_{0}>2N(1-\gamma)\,,
$$
and hence $V_{i,A}$ satisfies the ansatz \eqref{asatz-V-iA}.

For $V_{i,B_{1}}$, Lemma \ref{Prop-linear-fluid} gives
\[
\vert J_{3}\vert \lesssim \varepsilon^2M^2\theta_{i}^{2}\ast_{x,t}\theta_{i}^{\frac{5}{2}}\lesssim\varepsilon^2M^2\theta_{i}^{\frac{3}{2}}\,,
\]
\[
\vert \partial_{x}J_{3}\vert \lesssim \varepsilon^2M^2\left[\int_{0}^{\frac{t}{2}}\theta_{i}^{3}(t-\tau)\ast_{x}\theta_{i}^{\frac{5}{2}}(\tau)d\tau+\int_{\frac{t}{2}}^{t}\theta_{i}^{2}(t-\tau)\ast_{x}\theta_{i}^{\frac{7}{2}}(\tau)d\tau\right]\lesssim\varepsilon^2M^2\theta_{i}^{\frac{5}{2}}\,,
\]
where the final inequality follows by transferring one derivative to the source on $[\frac{t}{2},t]$.
For $J_{4}$, note that for $j\neq k$
\begin{equation}
\label{eq-theta-psi-j-neq-k}
\theta^{1}_{j}\psi_{k}^{\frac{3}{2}}\lesssim (1+t)^{-\frac{1}{2}}\psi_{j}^{3}\,,
\end{equation}
we obtain
\begin{equation*}
\vert J_{4}\vert \lesssim \varepsilon^{2}M^2\left[\theta_{i}^{2}\ast_{x,t}\Bigg(\sum_{l=1}^{3}(1+t)^{-\frac{1}{2}}\psi_{l}^{3}+\theta_{l}^{3}\Bigg)\right]\lesssim\varepsilon^{2}M^2\left(\psi_{i}^{\frac{3}{2}}+\sum_{k\neq i}\tilde\psi^{\frac{3}{2}}_{k}\right)\,,
\end{equation*}
and
\begin{align*}
\vert \partial_{x}J_{4}\vert
&\lesssim\left|\sum_{\substack{1\leq j,k\leq 3\\ j\neq k}}\pa_{x}\left(\int_{0}^{t/2}+\int_{t/2}^{t}\right)\Big\langle E_{i},\mathbb{G}^{t-\tau}_{i}\Gamma\left(\tilde\phi_{j}E_{j},V_{k,B}E_{k}\right)(\tau)\Big\rangle_{\xi}d\tau\right|
\notag\\
&\lesssim\varepsilon^{2}M^2(1+t)^{-\frac{1}{2}}\left[\theta_{i}^{2}*_{x,t}\sum_{l=1}^{3}\left((1+t)^{-\frac{1}{2}}\psi_{l}^{3}+\theta_{l}^{3}\right)\right]
\notag\\
&\lesssim \varepsilon^{2}M^2(1+t)^{-\frac{1}{2}}\left(\psi_{i}^{\frac{3}{2}}+\sum_{k\neq i}\tilde\psi^{\frac{3}{2}}_{k}\right)\,.
\end{align*}
We denote each component of $J_{5}$ by
\[
J_{5}^{j}=\int_{0}^{t}\left\langle E_{i},\mathbb{G}_{i}^{t-\tau}2\Gamma\left(\tilde\phi_{j}E_{j},V_{j,B}E_{j}\right)(\tau)\right\rangle_{\xi}d\tau\,.
\]
For the contribution of $V_{j,B_{2}}$, we immediately obtain
\begin{align*}
\vert \partial_{x}^{s}J_{5}^{j}\vert_{V_{j,B}=V_{j,B_{2}}}\vert
 &\lesssim\varepsilon^2M^2\left[\int_{0}^{\frac{t}{2}}\theta_{i}^{2+s}(t-\tau)\ast_{x}\theta_{j}^{3}(\tau)d\tau+\int_{\frac{t}{2}}^{t}\theta_{i}^{2}(t-\tau)\ast_{x}\theta_{j}^{3+s}(\tau)d\tau\right]
 \\[2mm]
 &\lesssim\varepsilon^2M^2(1+t)^{-\frac{s}{2}}\left(\psi_{i}^{\frac{3}{2}}+\tilde\psi_{j}^{\frac{3}{2}}\right)\,.
\end{align*}
For the contribution of $V_{j,B_{1}}$, we use the Cancellation Mechanism. For $t\geq 4$, split $J_{5}^{j}\vert_{V_{j,B}=V_{j,B_{1}}}$ into $$J_{5}^{j}\vert_{V_{j,B}=V_{j,B_{1}}}=\int_{0}^{\sqrt{t}}\cdots+\int_{\sqrt{t}}^{t}\cdots\equiv J^{j}_{5\left\{[0,\sqrt{t}]\right\}}\bigg\vert_{V_{j,B}=V_{j,B_{1}}}+J^{j}_{5\left\{[\sqrt{t},t]\right\}}\bigg\vert_{V_{j,B}=V_{j,B_{1}}}\,.$$ Its first part satisfies
\[
\Bigg\vert\partial_{x}^{s}J^{j}_{5\left\{[0,\sqrt{t}]\right\}}\bigg\vert_{V_{j,B}=V_{j,B_{1}}}\Bigg\vert\lesssim \varepsilon^{2}M^2\int_{0}^{\sqrt{t}}\theta_{i}^{2+s}(t-\tau)\ast_{x}\theta_{j}^{\frac{5}{2}}(\tau)d\tau\lesssim\varepsilon^{2}M^2\theta_{i}^{\frac{7}{4}+s}\,.
\]
The definition of $V_{j,B_{1}}$ also gives
\begin{equation}
\label{cancellation-assumption--V-j-B1}
 \widehat{(\partial_t+\lambda_j\partial_x)} V_{j,B_{1}}
 =(\sigma_j(\eta)+i\lambda_j\eta)\widehat{V_{j,B_{1}}}+\chi_{\delta}(\eta)\langle e_j(\eta),E_j\rangle_\xi \,\mathcal{S}(\eta,t)\,,
\end{equation}
where
\begin{align*}
\mathcal{S}(\eta,t)&=\sum_{\substack{1\leq l,k\leq 3\\(l,k)\neq (j,j)}}\left(\left\langle \Gamma(E_{l},E_{k}),e_{j}(-\eta)\right\rangle_{\xi}\widehat{(\tilde\phi_{l}\tilde\phi_{k}})\right)
+\sum_{1\leq l,k\leq 3}\left(\left\langle \Gamma(E_{l},E_{k}),e_{j}(-\eta)\right\rangle_{\xi}\widehat{V_{l,B}V_{k,B}}\right)
\\
&\quad+2\sum_{1\leq l,k\leq 3}\left(\left\langle \Gamma(E_{l},E_{k}),e_{j}(-\eta)\right\rangle_{\xi}\widehat{\tilde\phi_{l}V_{k,B}}\right)+\left\langle \,\widehat{g}_{A},e_{j}(-\eta)\right\rangle_{\xi}\,.
\end{align*}

Together with the expression of $(\partial_t+\lambda_{j}\partial_{x})\tilde\phi_{j}$ \eqref{eq-tilde-j-par t+lamda j par x}, one yields
\begin{align*}
&\left\vert(\lambda_{j}-\lambda_{i})\partial_{x}^{s}J^{j}_{5\{[\sqrt t,t]\}}\Big\vert_{V_{j,B}=V_{j,B_{1}}}\right\vert
\\
 &=\left\vert-\big[\partial_{x}^{s}H_{ij}(t-\tau)*_x(2\tilde{\phi}_{j}V_{j,B_{1}})(\tau)\big]_{\sqrt t}^{t}
   -\int_{\sqrt t}^t(\partial_{t}+\lambda_{i}\partial_{x})\partial_{x}^{s}H_{ij}(t-\tau)*_x(2\tilde\phi_{j}V_{j,B_{1}})(\tau)\,d\tau\right.\\
 &\quad\left.+\int_{\sqrt t}^t\partial_{x}^{s}H_{ij}(t-\tau)*_x(\partial_\tau+\lambda_j\partial_x)(2\tilde\phi_jV_{j,B_{1}})(\tau)\,d\tau\right\vert
              \\
 &\lesssim\varepsilon^2M^2\Big[\theta_j^{\frac{5}{2}+s}
   +\theta_i^{1+s}(t-\sqrt t)*_x\theta_j^{\frac{5}{2}}(\sqrt t)+\int_{\sqrt t}^t\theta_i^{2+s}(t-\tau)*_x\theta_j^{\frac{7}{2}}(\tau)\,d\tau\\
 &+\int_{\sqrt t}^t\theta_i^{1+s}(t-\tau)*_x\theta_j^4(\tau)\,d\tau\Big]\\
 &\lesssim\varepsilon^2M^2
       (1+t)^{-\frac{s}{2}}\big(\psi_i^{3/2}+\tilde\psi_j^{3/2}\big)\,.
\end{align*}
We note that one factor $\eta$ is transferred from $\widehat{(\partial_{t}+\lambda_{i}\partial_{x})H_{ij}}$ to the source $\widehat{\tilde\phi_{j}V_{j,B_{1}}}$ in the second term. In the third term, we transfer one factor $\eta$ from $(\sigma_{j}(\eta)+i\lambda_{j}\eta)\widehat{V_{j,B_{1}}}$ to Green function $\widehat{H_{ij}}$ and the factor $\widehat{\phi_{j}}$ of source, since $(\sigma_{j}(\eta)+i\lambda_{j}\eta)=-A_{j}\eta^{2}+\eta^{3}R_{\sigma,j}(\eta)$.
For $J_{6}$, treating $\partial^{s}_{x}$ in the same way as for $J_{5}^{j}\vert_{V_{j,B}=V_{j,B_{2}}}$, we immediately obtain
\[
\vert\partial_{x}^{s}J_{6}\vert \lesssim \varepsilon^2M^2(1+t)^{-\frac{s}{2}}\Bigg(\psi_{i}^{\frac{3}{2}}+\sum_{k\neq i}\tilde\psi_{k}^{\frac{3}{2}}\Bigg)\,.
\]
For $J_{7}$, by Lemma \ref{Prop-linear-fluid}, we obtain
\[
\vert \partial_{x}^{s}J_{7}\vert \lesssim M^{2}\varepsilon^{2}(\theta_{i}^{2+s})\ast_{x,t}(1+t+\vert x\vert)^{-\frac{N_{0}}{1-\gamma}}\lesssim M^{2}\varepsilon^{2}\theta_{i}^{2+s}\,.
\]
Therefore, $V_{i,B_{1}}$ satisfies the ansatz \eqref{asatz-V-i-B1}.

We next consider the remaining components of $V_{i,B_{2}}$.
By Lemma \ref{Prop-linear-fluid} and Proposition \ref{PF-estimate}, we obtain
\[
\vert \partial_{x}^{s}L_{6}\vert\lesssim M^{2}\varepsilon^{2}\Bigg(\sum_{l\neq i}\theta_{l}^{3+s}+(1+t+\vert x\vert)^{-\frac{N_{0}}{1-\gamma}}\Bigg)\,,\qquad \text{for any }N_{0}>2N(1-\gamma)\,,
\]
For $L_{7}$, applying the same argument as for $L_{2}$ to $g_B$
in \eqref{eq-gB}, with the ansatz \eqref{asatz-V-i-B1}--\eqref{asatz-V-i-B2}, gives
\[
\vert \partial_{x}^{s}L_{7}\vert \lesssim\varepsilon^2M^2(1+t)^{-\frac{1+s}{2}}\sum_{m=1}^{3}(\psi_{m}^{\frac{3}{2}})\,,
\]
\begin{align*}
\vert L_{8}\vert &\lesssim M^2\varepsilon^2\sum_{l\neq i}\left[\int_{0}^{\frac{t}{2}}\theta_{l}^{3}(t-\tau)\ast_{x}\theta_{l}^{\frac{5}{2}}(\tau)d\tau+\int_{\frac{t}{2}}^{t}\theta_{l}^{2}(t-\tau)\ast_{x}\theta_{l}^{\frac{7}{2}}(\tau)d\tau\right]
\lesssim M^2\varepsilon^2\sum_{l\neq i}(\theta_{l}^{\frac{5}{2}})\,,
\end{align*}
\begin{align*}
\vert \partial_{x}L_{8}\vert &\lesssim M^2\varepsilon^2\sum_{l\neq i}\left[\int_{0}^{\frac{t}{2}}\theta_{l}^{4}(t-\tau)\ast_{x}\theta_{l}^{\frac{5}{2}}(\tau)d\tau+\int_{\frac{t}{2}}^{t}\theta_{l}^{3}(t-\tau)\ast_{x}\theta_{l}^{\frac{7}{2}}(\tau)d\tau\right]
\\
&\lesssim M^2\varepsilon^2\ln(1+t)\sum_{l\neq i}(\theta_{l}^{\frac{7}{2}})\,,
\end{align*}
and
\begin{align*}
\vert \partial_{x}^{s}L_{9}\vert &\lesssim M^2\varepsilon^2\sum_{l\neq i}\left[\int_{0}^{\frac{t}{2}}\theta_{l}^{3+s}(t-\tau)\ast_{x}\Bigg(\sum_{m=1}^{3}\psi_{m}^{3}\Bigg)(\tau)d\tau\right.
\\
&\quad\left.+\int_{\frac{t}{2}}^{t}\theta_{l}^{2+s}(t-\tau)\ast_{x}\Bigg(\sum_{m=1}^{3}(1+t)^{-\frac{1}{2}}\psi_{m}^{3}\Bigg)(\tau)d\tau\right]
\\
&\lesssim M^2\varepsilon^2(1+t)^{-\frac{1+s}{2}}\Bigg(\sum_{m=1}^{3}\psi_{m}^{\frac{3}{2}}\Bigg)\,.
\end{align*}
The estimates of $L_{8}$ and $L_{9}$ are obtained by transferring the extra factor $\eta$, arising from $\langle E_{i},e_{l}(-\eta)\rangle_{\xi}$, to the source on $[\frac{t}{2},t]$ in their Fourier representations of $\widehat{L}_{8}$ and $\widehat{L}_{9}$, respectively.

We then consider the term $L_{10}$. We denote each component of $L_{10}$ by
\[
L_{10}^{lk}=\int_{0}^{t}\left\langle E_{i},\mathbb{G}_{l}^{t-\tau}2\Gamma\left(\tilde\phi_{k}E_{k},V_{k,B}E_{k
}\right)(\tau)\right\rangle_{\xi}d\tau\,.
\]
For the contribution of $V_{k,B_{2}}$, we immediately obtain
\begin{align*}
\vert \partial_{x}^{s}L_{10}^{lk}\vert_{V_{k,B}=V_{k,B_{2}}}\vert&
 \lesssim\varepsilon^2M^2\left[\int_{0}^{\frac{t}{2}}\theta_{l}^{3+s}(t-\tau)\ast_{x}\theta_{k}^{3}(\tau)d\tau+\int_{\frac{t}{2}}^{t}\theta_{l}^{3}(t-\tau)\ast_{x}\theta_{k}^{3+s}(\tau)d\tau\right]
 \\[2mm]
 &\lesssim\varepsilon^2M^2(1+t)^{-\frac{1+s}{2}}\left(\psi_{l}^{\frac{3}{2}}+\psi_{k}^{\frac{3}{2}}\right)\,.
\end{align*}
For the contribution of $V_{k,B_{1}}$, we use the Cancellation Mechanism. For $t\geq 4$, split $L_{10}^{lk}\vert_{V_{k,B}=V_{k,B_{1}}}$ into $$L_{10}^{lk}\vert_{V_{k,B}=V_{k,B_{1}}}=\int_{0}^{\sqrt{t}}\cdots+\int_{\sqrt{t}}^{t}\cdots\equiv L^{lk}_{10\left\{[0,\sqrt{t}]\right\}}\big\vert_{V_{k,B}=V_{k,B_{1}}}+L^{lk}_{10\left\{[\sqrt{t},t]\right\}}\big\vert_{V_{k,B}=V_{k,B_{1}}}\,.$$ Its first part satisfies
\[
\left\vert\partial_{x}^{s}L^{lk}_{10\left\{[0,\sqrt{t}]\right\}}\big\vert_{V_{k,B}=V_{k,B_{1}}}\right\vert\lesssim \varepsilon^{2}M^2\int_{0}^{\sqrt{t}}\theta_{l}^{3+s}(t-\tau)\ast_{x}\theta_{k}^{\frac{5}{2}}(\tau)d\tau\lesssim\varepsilon^{2}M^2\theta_{l}^{\frac{11}{4}+s}\,.
\]
For the second part, for $l\neq i$ and $k\neq l$, we define the kernel $\widetilde{H}_{lk}$ by
\begin{equation}
\label{def-tilde-H-lk}
 \widetilde{H}_{lk}(t,x)=\frac1{2\pi}\int_{\mathbb R}e^{ix\eta+\sigma_l(\eta)t}
 \frac{\chi_{\delta}(\eta)\langle e_l(\eta),E_i\rangle_\xi
       \langle2\Gamma(E_k,E_k),e_l(-\eta)\rangle_\xi}{(i\eta)^{2}}\,d\eta\,,
\end{equation}
which satisfies \eqref{asumption-cancellation-mechanism-1}.
Hence, $L^{lk}_{10\left\{[\sqrt{t},t]\right\}}\big\vert_{V_{k,B}=V_{k,B_{1}}}$ can be written as
\[
L^{lk}_{10\{[\sqrt t,t]\}}\big\vert_{V_{k,B}=V_{k,B_{1}}}=\int_{\sqrt{t}}^{t}\partial^{2}_{x}\widetilde{H}_{lk}(t-\tau)\ast_{x}(\tilde\phi_{k}V_{k,B_{1}})(\tau)d\tau\,.
\]
Combining this with the expression for $(\partial_t+\lambda_{j}\partial_{x})\tilde\phi_{j}$ \eqref{eq-tilde-j-par t+lamda j par x} and $(\partial_{t}+\lambda_{k}\partial_{x})V_{k,B_{1}}$ \eqref{cancellation-assumption--V-j-B1} yields
\begin{align*}
&\left\vert(\lambda_{k}-\lambda_{l})\partial_{x}^{s}L^{lk}_{10\{[\sqrt t,t]\}}\big\vert_{V_{k,B}=V_{k,B_{1}}}\right\vert
\\
 &=\left\vert-\big[\partial_{x}^{s+1}\widetilde{H}_{lk}(t-\tau)*_x(\tilde{\phi}_{k}V_{k,B_{1}})(\tau)\big]_{\sqrt t}^{t}
   -\int_{\sqrt t}^t(\partial_{t}+\lambda_{l}\partial_{x})\partial_{x}^{s+1}\widetilde{H}_{lk}(t-\tau)*_x(\tilde\phi_{k}V_{k,B_{1}})(\tau)\,d\tau\right.\\
 &\quad\left.+\int_{\sqrt t}^t\partial_{x}^{s+1}\widetilde{H}_{lk}(t-\tau)*_x(\partial_\tau+\lambda_k\partial_x)(\tilde\phi_kV_{k,B_{1}})(\tau)\,d\tau\right\vert
              \\
 &\lesssim\varepsilon^2M^2\Big[\theta_k^{\frac{5}{2}+s}
   +\theta_l^{2+s}(t-\sqrt t)*_x\theta_k^{\frac{5}{2}}(\sqrt t)+\int_{\sqrt t}^t\theta_l^{3+s}(t-\tau)*_x\theta_k^{\frac{7}{2}}(\tau)\,d\tau\\
 &+\int_{\sqrt t}^t\theta_l^{2+s}(t-\tau)*_x\theta_k^4(\tau)\,d\tau\Big]\\
 &\lesssim\varepsilon^2 M^2
       (1+t)^{-\frac{s+1}{2}}\big(\psi_l^{3/2}+\psi_k^{3/2}\big)\,.
\end{align*}
For $L_{11}$, treating $\partial^{s}_{x}$ in the same way as for $L_{9}$, we immediately obtain
\[
\vert\partial_{x}^{s}L_{11}\vert \lesssim \varepsilon^2M^2(1+t)^{-\frac{s+1}{2}}\Bigg(\sum_{m=1}^{3}\psi_{m}^{\frac{3}{2}}\Bigg)\,.
\]
Thus, $V_{i,B_{2}}$ satisfies the ansatz \eqref{asatz-V-i-B2}.
Combining the component estimates, the known forcing has amplitude
$C_{0}\varepsilon$, while the terms containing $V_{i,B}$ have amplitude at most
$C_1\varepsilon^2M^2$, with the corresponding profiles in \eqref{asatz-V-iA}-\eqref{asatz-V-i-B2}.
Here $C_0\geq1$ and $C_1$ depend only on the
initial-data norm, not on $M$, $\varepsilon$, or the terminal time.
The output coefficient is therefore bounded by
\[
C_0\varepsilon+C_1\varepsilon^{2} M^2=\varepsilon\left(C_0+C_1\varepsilon M^2\right).
\]
Choose $M\geq 4 C_{0}$ sufficiently large and $0<\varepsilon\leq1$ sufficiently small that
\[
C_0+C_1\varepsilon M^2\leq\frac12 M.
\]
The output is then bounded by $\frac{1}{2}M\varepsilon$,
completing the proof.
\begin{remark} \label{remark-cancellation mechanism} The estimate of $J_{2}^{ij}$ reveals a cancellation mechanism that plays an important role throughout the paper. We refer to this argument as the ``Cancellation Mechanism''. It reveals some cancellation arising from the wave interaction when the kernel and the source propagate with distinct characteristic speeds \(\lambda_i\) and \(\lambda_j\). This argument is analogous to the cancellation mechanism for viscous conservation laws in Lemma 3.4 of \cite{[Zeng]}.


We now specify how the Cancellation Mechanism is applied to mixed sources. Consider a quantity $L$ satisfying, for any $n\in \mathbb{N}\cup\left\{0\right\}$,
\[
\partial_{x}^{n}L=\int_{0}^{t}\partial_{x}^{n+1}H(t-\tau)\ast_{x}(AB)(\tau) d\tau\,.
\]
We first split the time integral at $\sqrt{t}$, namely, $$\partial_{x}^{n}L=\int_{0}^{\sqrt{t}}\cdots+\int_{\sqrt{t}}^{t}\cdots=\partial_{x}^{n}L_{\left\{[0,\sqrt{t}]\right\}}+\partial_{x}^{n}L_{\left\{[\sqrt{t},t]\right\}}\,.$$ Here the kernel $H$ satisfies, for any $n\in\mathbb{N}\cup\left\{0\right\}$,
\begin{equation}
\label{asumption-cancellation-mechanism-1}
 |\partial_x^nH|\lesssim\theta_i^{1+n},\qquad
 (\partial_t+\lambda_i\partial_x)H=\partial_x H_1,
 \qquad |\partial_x^nH_1|\lesssim\theta_i^{2+n}.
\end{equation}
Here $H,H_1$ are used only to express the exact multipliers in this
calculation; no new wave is introduced. For the source $AB$, both $A$ and $B$ propagate in the $\lambda_j$ direction, with $j\ne i$, and satisfy
\begin{equation}
\widehat{(\partial_{t}+\lambda_{j}\partial_{x})X}=O(\eta^{2})\widehat{X}+\mathcal{S}_{X}(\eta,t)\,, \ \ \ X\in\left\{A,B\right\}\,,
\label{asumption-cancellation-mechanism-2}
\end{equation}
for some source term $\mathcal{S}_{X}(\eta,t)$.
The integration by parts already used for $(\lambda_{j}-\lambda_{i})\partial_{x}^{n}L_{\left\{[\sqrt{t},t]\right\}}$ gives
\begin{align}
 &(\lambda_j-\lambda_i)
       \int_{\sqrt t}^t\partial_x^{n+1}H(t-\tau)*_x(AB)(\tau)\,d\tau \notag\\
 &=-\big[\partial_x^nH(t-\tau)*_x(AB)(\tau)\big]_{\sqrt t}^{t}
   -\int_{\sqrt t}^t\partial_x^nH_1(t-\tau)*_x\partial_x(AB)(\tau)\,d\tau\notag\\
 &\quad+\int_{\sqrt t}^t\partial_x^nH(t-\tau)*_x
              (\partial_\tau+\lambda_j\partial_x)(AB)(\tau)\,d\tau\,,
              \label{eq-cancellation-mechanism}
\end{align}
which yields a better estimate than a direct convolution bound.
The interval $[0,\sqrt t]$ is estimated directly.
\end{remark}
\end{proof}
We now turn to the estimate of $V^{\perp}$.
  \begin{proposition}
  \label{prop-V-perp}
    Let $\beta >\frac{3}{2}$, $N_{0}>2N(1-\gamma)$. Assume that $f_{0}$ is compactly supported in the variable $x$ for all $\xi$, namely,
    \begin{equation*}
    f_{0}(x,\xi)\equiv 0 \ \text{  for  }\ \vert x\vert\geq 1, \ \xi\in\mathbb{R}^{3},
    \end{equation*}
    and $f_0\in L_x^\infty L^{\infty}_{\xi,\beta+N_0}$. Then there exists a sufficiently small $\varepsilon>0$ such that $V^{\perp}$ in \eqref{leading-nonfluid} satisfies
    \begin{equation}
    \left\vert V^{\perp}\right\vert_{L^{\infty}_{\xi,\beta}} \lesssim \varepsilon\left( (1+t+\vert x\vert)^{-\frac{N_{0}}{1-\gamma}}+
    \sum_{l=1}^{3}\theta^{2}_{l}\right)+\varepsilon^{2}\left(\sum_{l=1}^{3}(1+t)^{-\frac{1}{2}}\psi_{l}^{\frac{3}{2}}\right)
    \label{est-V-perp}
    \end{equation}
    for any $N_{0}>2N(1-\gamma)$.
    \end{proposition}
\begin{proof} The convolution estimates used in this proof are listed in Tables \ref{table-1}--\ref{table-2} in the Appendix. We decompose $V^{\perp}$ as follows:
\begin{equation}
\label{decom-V-perp}
V^{\perp}=V_{1}^{\perp}+\sum_{m=1}^{3}V_{2,m}^{\perp}\,,
\end{equation}
where
\begin{equation}
\label{eq-V-1-perp}
V_{1}^{\perp}=V_{1,A}^{\perp}+V_{1,B}^{\perp}\,,
\end{equation}
\begin{align}
V_{1,A}^{\perp}&=\varepsilon \mathbb{P}_{1}\mathbb{G}^{t}_{\ast}f_{0}
\notag\\
&\quad+\int_{0}^{t}\mathbb{P}_{1}\mathbb{G}^{t-\tau}_{\ast}\left[\sum_{1\leq j,k\leq 3}\Big(\Gamma\left(2\tilde{\phi}_{j}E_{j},V_{k,A}E_{k}\right)+\Gamma\left(\left(V_{j,A}+2V_{j,B}\right)E_{j},V_{k,A}E_{k}\right)\Big)\right](\tau)d\tau\,,
\label{eq-V-1-A-perp}
\end{align}
\begin{align}
V_{1,B}^{\perp}&=\varepsilon\mathbb{P}_{1}\mathcal{G}^{t}f_{0}
\notag\\
&\quad +\int_{0}^{t}\mathbb{P}_{1}\left(\mathbb{G}^{t-\tau}_{\ast}+\mathcal{G}^{t-\tau}\right)\Bigg[\sum_{1\leq j,k\leq 3}\Gamma\left(\left(2\tilde{\phi}_{j}+V_{j,B}\right)E_{j},V_{k,B}E_{k}\right)\Bigg](\tau)d\tau
\notag\\
&\quad+\int_{0}^{t}\mathbb{P}_{1}\mathbb{G}_{\sharp}^{t-\tau}\Bigg[\sum_{1\leq j,k\leq 3}\Big(\Gamma\left(2\tilde{\phi}_{j}E_{j},V_{k,A}E_{k}\right)+\Gamma\left(\left(V_{j,A}+2V_{j,B}\right)E_{j},V_{k,A}E_{k}\right)\Big)\Bigg](\tau)d\tau
\notag\\
&\quad+\sum_{m=1}^{3}\int_{0}^{t}\mathbb{P}_{1}\mathbb{G}^{t-\tau}_{m}\left[\sum_{\substack{1\leq j,k\leq 3\\ j\neq k}}\Gamma\left(\tilde{\phi}_{j}E_{j},\tilde{\phi}_{k}E_{k}\right)+\sum_{\substack{1\leq j\leq 3\\ j\neq m}}\Gamma\left(\tilde{\phi}_{j}E_{j},\tilde{\phi}_{j}E_{j}\right)\right](\tau)d\tau
\notag\\
&\quad+\int_{0}^{t}\mathbb{P}_{1}\left(\mathbb{G}^{t-\tau}_{\ast}+\mathcal{G}^{t-\tau}\right)\Bigg[\sum_{\substack{1\leq j,k\leq 3\\ j\neq k}}\Gamma\left(\tilde{\phi}_{j}E_{j},\tilde{\phi}_{k}E_{k}\right)\Bigg](\tau)d\tau
\notag\\
&\quad+\int_{0}^{t}\mathbb{P}_{1}\mathbb{G}^{t-\tau}_{L;0}\Bigg[\sum_{1\leq j,k\leq 3}\Big(\Gamma\left(\left(2\tilde{\phi}_{j}+V_{j,B}\right)E_{j},V_{k,B}E_{k}\right)\Big)\Bigg](\tau)d\tau\,,
\label{eq-V-1-B-perp}
\end{align}
and
\begin{align}
V_{2,m}^{\perp}&=\varepsilon \mathbb{P}_{1}\mathbb{G}^{t}_{m}f_{0}+\int_{0}^{t}\mathbb{P}_{1}\mathbb{G}^{t-\tau}_{m}\Gamma\left(\tilde{\phi}_{m}E_{m},\tilde{\phi}_{m}E_{m}\right)(\tau)d\tau
\notag\\
&\quad+\int_{0}^{t}\mathbb{P}_{1}\left(\mathbb{G}^{t-\tau}_{\ast}
+\mathcal{G}^{t-\tau}\right)\Gamma\left(\tilde{\phi}_{m}E_{m},\tilde{\phi}_{m}E_{m}\right)(\tau)d\tau\,.
\label{eq-V-2-m-perp}
\end{align}

Here, \(V_1^\perp\) corresponds to the off-diagonal component, whereas \(V_{2,m}^\perp\) corresponds to the \(m\)-th kinetic Burgers wave \(\tilde\phi_m\). Applying the arguments used for these two types of waves, we obtain the following estimates.

     For $V^{\perp}_{1,A}$ and any $N_{0}>2N(1-\gamma)$, we obtain
    \begin{equation}
    \label{est-V-perp-1A}
     \left\vert V^{\perp}_{1,A}\right\vert_{L^{\infty}_{\xi,\beta}} \lesssim \varepsilon(1+t+\vert x\vert )^{-\frac{N_{0}}{1-\gamma}} \left\Vert f_{0}\right\Vert_{L^{\infty}_{\xi,\beta+N_{0}}L^{\infty}_{x}}+\varepsilon^{2} (1+t+\vert x\vert )^{-\frac{N_{0}}{1-\gamma}}\,.
    \end{equation}

The estimates for $V_{1,B}^{\perp}$ are
    \begin{equation}
    \left\vert V^{\perp}_{1,B}\right\vert_{L^{\infty}_{\xi,\beta}} \lesssim \varepsilon \left(1+t+\vert x\vert\right)^{-\frac{N_{0}}{1-\gamma}}+\varepsilon^{2}\left[\sum_{m=1}^{3}(1+t)^{-\frac{1}{2}}\psi_{m}^{\frac{3}{2}}\right],
    \label{est-V-perp-1B}
    \end{equation}
    \begin{equation}
    \label{est-der-V-perp-1B}
    \left\vert \partial_{x}V^{\perp}_{1,B}\right\vert_{L^{\infty}_{\xi,\beta}} \lesssim \varepsilon \left(1+t+\vert x\vert\right)^{-\frac{N_{0}}{1-\gamma}}+\varepsilon^{2}\left[\sum_{m=1}^{3}\ln(1+t)\theta_{m}^{\frac{7}{2}}+(1+t)^{-1}\psi_{m}^{\frac{3}{2}}\right].
    \end{equation}

    Finally, for $V_{2,m}^{\perp}$, we obtain
    \begin{equation}
    \label{est-der-V-perp-2m}
    \left\vert \partial^{s}_{x}V_{2,m}^{\perp}\right\vert_{L^{\infty}_{\xi,\beta}}\lesssim \varepsilon \theta^{2+s}_{m}, \ \ \ \text{for any} \ s\in \mathbb{N}\cup\left\{0\right\}.
    \end{equation}
 \end{proof}
 To apply the Cancellation Mechanism to the mixed terms $\Gamma\big(\widetilde{\phi}_m E_m,V_{2,m}^{\perp}\big)$ in the proof of Proposition \ref{prop-remainder}, we also need the following characteristic derivative estimate.
 \begin{remark}
We also record the corresponding characteristic derivative estimate.
For each fixed $s\in\mathbb{N}\cup\left\{0\right\}$, we have
\begin{equation}\label{V2-characteristic-bound}
 \left|\partial_x^s
       (\partial_t+\lambda_m\partial_x)V_{2,m}^\perp
                                      \right|_{L^\infty_{\xi,\beta}}
 \lesssim\varepsilon\theta_m^{4+s}.
\end{equation}
First, we denote the sum of the first two terms in $V_{2,m}^{\perp}$ by $A_{m}$ and the last term by $B_{m}$. The definition of $A_{m}$ gives
\begin{align*}
\widehat{(\partial_{t}+\lambda_{m}\partial_{x})A_{m}}&=\left(\sigma_{m}(\eta)+i\lambda_{m}\eta\right)\widehat{A_{m}}
\\
&\quad+\chi_{\delta}(\eta)\mathbb{P}_{1}e_{m}(\eta)\left\langle \Gamma(E_{m},E_{m}),e_{m}(-\eta)\right\rangle_{\xi}\widehat{(\tilde\phi_{m})^{2}}\,.
\end{align*}
Hence \(A_m\) satisfies the same estimate as \(\tilde\phi_m\), namely.
\[
\left\vert \partial_{x}^{s}\left(\partial_{t}+\lambda_{m}\partial_{x}\right)A_{m}\right\vert \lesssim (\varepsilon+\varepsilon^{2})\theta_{m}^{4+s}\,, \ \ \
\text{for every }s\in \mathbb{N}\cup\left\{0\right\}\,.
\]
We next consider the last term $B_{m}$. After changing the integration variable $u=t-\tau$, we rewrite
\\
\(\partial_{x}^{s}(\partial_{t}+\lambda_{m}\partial_{x})B_m\) as
\begin{align*}
\partial_{x}^{s}(\partial_{t}+\lambda_{m}\partial_{x})B_{m}&=
\partial_{x}^{s}(\partial_{t}+\lambda_{m}\partial_{x})\int_{0}^{t}\mathbb{P}_{1}\left(\mathbb{G}_{\ast}^{u}+\mathcal{G}^{u}\right)\Gamma(\tilde\phi_{m}E_{m},\tilde\phi_{m}E_{m})(t-u)du
\\
 &=\partial_{x}^{s}\left[\mathbb P_1(\mathbb G_*^t+\mathcal G^t)
                       \Gamma(\tilde\phi_mE_m,\tilde\phi_mE_m)(0)\right]\\
 &\quad+\int_0^t\mathbb P_1(\mathbb G_*^{t-\tau}+\mathcal G^{t-\tau})
       \left[\partial_{x}^{s}(\partial_\tau+\lambda_m\partial_x)\right]
                       \Gamma(\tilde\phi_mE_m,\tilde\phi_mE_m)(\tau)\,d\tau\,,
\end{align*}
for every $s\in\mathbb{N}\cup\left\{0\right\}$.
The first term is estimated directly and decays rapidly. In the second term,
the characteristic derivative is bounded by
$C\varepsilon^2\theta_m^{4+s}$. This establishes
\eqref{V2-characteristic-bound}.
\end{remark}

\subsection{Remainder structure}
Set
\begin{equation}\label{eq-R-definition}
 R=f-\Phi-V^\perp.
\end{equation}
According to \eqref{leading fluid} and \eqref{leading-nonfluid},
\[
 \Phi+V^\perp=\varepsilon\mathbb G^tf_0
       +\int_0^t\mathbb G^{t-\tau}\Gamma(\Phi,\Phi)(\tau)\,d\tau.
\]
Subtracting this identity from the equation for $f$, and using the
symmetry of $\Gamma$, we obtain
\begin{align}\label{eq-R-exact}
 R(t)={}&\int_0^t\mathbb G^{t-\tau}\Big[
       \Gamma(R,R)+2\Gamma(R,\Phi+V^\perp)\notag\\
 &\hspace{34mm}+2\Gamma(\Phi,V^\perp)
                   +\Gamma(V^\perp,V^\perp)\Big](\tau)\,d\tau.
\end{align}

We decompose $R$ into fluid and non-fluid parts:
\begin{equation}\label{eq-R-p}
 R=R^F+R^\perp,\qquad
 R^F=\mathbb P_0R=\sum_{i=1}^3R_iE_i,\qquad
 R_i=\langle E_i,R\rangle_\xi,\qquad R^\perp=\mathbb P_1R.
\end{equation}
The exact equations for these components are therefore
\begin{align}\label{eq-R-i}
 R_i(t,x)={}&\int_0^t\Big\langle E_i,\mathbb G^{t-\tau}\Big[
       \Gamma(R,R)+2\Gamma(R,\Phi+V^\perp)\notag\\
 &\hspace{34mm}+2\Gamma(\Phi,V^\perp)
                   +\Gamma(V^\perp,V^\perp)\Big](\tau)\Big\rangle_\xi\,d\tau\,,
\end{align}
and
\begin{align}\label{eq-R-perp}
 R^\perp(t)={}&\int_0^t\mathbb P_1\mathbb G^{t-\tau}\Big[
       \Gamma(R,R)+2\Gamma(R,\Phi+V^\perp)\notag\\
 &\hspace{34mm}+2\Gamma(\Phi,V^\perp)
                   +\Gamma(V^\perp,V^\perp)\Big](\tau)\,d\tau.
\end{align}

\begin{proposition}[Remainder Structure]
\label{prop-remainder}
Let $\beta>\frac{3}{2}+\gamma$, $N_0>2N(1-\gamma)$, and
$0<\delta_0<\frac{1}{4}$. Fix an integer $k$ sufficiently large so that
\[
 2^{-k}\leq \min\left\{\frac{1-\gamma}{2N_{0}+1-\gamma}\,,\frac{1-\delta_{0}}{6-2\delta_{0}}\right\}\,,
\]
and set
\begin{equation}\label{eq-hat-beta}
 \hat\beta=\beta+2^k(N_0+\gamma).
\end{equation}
Assume that $f_0(x,\xi)=0$ for $|x|\geq1$ and that
$f_0\in L_x^\infty L^\infty_{\xi,\hat\beta}$.
For sufficiently small $\varepsilon>0$, the exact remainder components
admit the decompositions $R_i=R_{i,A}+R_{i,B}$ and
$R^\perp=R_A^\perp+R_B^\perp$, and satisfy
\begin{equation}\label{est-R-iA}
 |R_{i,A}|\lesssim\varepsilon^2(1+t+|x|)^{-N_0/(1-\gamma)},
\end{equation}
\begin{equation}\label{est-R-iB}
 |\partial_{x}^{s}R_{i,B}|\lesssim\varepsilon^2\left[
 \theta_i^{2-\delta_0+s}+(1+t)^{-\frac{1-\delta_{0}+2s}{4}}
 \left(\psi_i^{3/2}+\sum_{m\ne i}\tilde\psi_m^{3/2}\right)\right], \quad \text{for any } s\in\left\{0,1\right\}\,,
\end{equation}
\begin{equation}\label{est-R-perp-A}
 |R_A^\perp|_{L^\infty_{\xi,\beta}}
 \lesssim\varepsilon^2(1+t+|x|)^{-N_0/(1-\gamma)},
\end{equation}
and
\begin{equation}\label{est-R-perp-B}
 |R_B^\perp|_{L^\infty_{\xi,\beta}}
 \lesssim\varepsilon^2\sum_{m=1}^3\left[
 \theta_m^{3-\delta_0}+(1+t)^{-\frac{3-\delta_{0}}{4}}\psi_m^{3/2}\right].
\end{equation}
\end{proposition}
\begin{proof}
The convolution estimates used in this proof are listed in Tables \ref{table-1}--\ref{table-2} in the Appendix. We use the Cancellation Mechanism described in Remark~\ref{remark-cancellation mechanism}. For convenience, set
\begin{align*}
 V_A^\perp&=V_{1,A}^\perp,&
 V_B^\perp&=V_{1,B}^\perp+\sum_mV_{2,m}^\perp.
\end{align*}
Write $R=R_A+R_B$. Applying
\eqref{sec3-split-A}--\eqref{sec3-split-B} to each term in
\eqref{eq-R-exact}, we define
\begin{align}\label{eq-R-A}
 R_A(t)=\int_0^t\mathbb G_*^{t-\tau}\Big[
 &\Gamma(R_A,R_A)+2\Gamma(R_A,R_B)
      +2\Gamma(R_A,\Phi+V^\perp)\notag\\
 &+2\Gamma(R_B,\Phi_A+V_A^\perp)
      +2\Gamma(\Phi_A,V^\perp)+2\Gamma(\Phi_B,V_A^\perp)\notag\\
 &+\Gamma(V_A^\perp,V_A^\perp)
      +2\Gamma(V_A^\perp,V_B^\perp)\Big](\tau)\,d\tau
\end{align}
and
\begin{align}\label{eq-R-B}
 R_B(t)={}&\int_0^t\mathbb G_\sharp^{t-\tau}\Big[
       \Gamma(R,R)+2\Gamma(R,\Phi+V^\perp)\notag\\
 &\hspace{33mm}+2\Gamma(\Phi,V^\perp)
                   +\Gamma(V^\perp,V^\perp)\Big](\tau)\,d\tau\notag\\
 &+\int_0^t\mathbb G_*^{t-\tau}\Big[
       \Gamma(R_B,R_B)+2\Gamma(R_B,\Phi_B+V_B^\perp)\notag\\
 &\hspace{33mm}+2\Gamma(\Phi_B,V_B^\perp)
                   +\Gamma(V_B^\perp,V_B^\perp)\Big](\tau)\,d\tau.
\end{align}
Adding \eqref{eq-R-A} and \eqref{eq-R-B} recovers
\eqref{eq-R-exact}. We then set
\begin{equation}\label{decom-R-i}
 R_{i,A}=\langle E_i,R_A\rangle_\xi,\qquad
 R_{i,B}=\langle E_i,R_B\rangle_\xi,
\end{equation}
\begin{equation}\label{decom-R-perp}
 R_A^\perp=\mathbb P_1R_A,\qquad R_B^\perp=\mathbb P_1R_B.
\end{equation}
Accordingly, $R_A=\sum_iR_{i,A}E_i+R_A^\perp$ and
$R_B=\sum_iR_{i,B}E_i+R_B^\perp$. Let $h_{A}$ denote the full bracket in \eqref{eq-R-A},
and let $h_B$ denote the bracket in the second integral of \eqref{eq-R-B}.
For $X\in\left\{A,B\right\}$, write $h_{X}=h_{X}^{(0)}+h_{X}^{(1)}$, where $h_{X}\vert_{R=0}=h_{X}^{(0)}$.
Thus, $h_{X}^{(0)}$ depends only on the leading waves, whereas $h_{X}^{(1)}$
also depends on the remainder. We split the non-smooth components as
\[
 R_{i,A}=\int_0^t\left\langle E_{i},\mathbb G_*^{t-\tau}h_{A}^{(1)}(\tau)\right\rangle_{\xi}\,d\tau
 +\int_0^t\left\langle E_{i},\mathbb G_*^{t-\tau}h_{A}^{(0)}(\tau)\right\rangle_{\xi}\,d\tau=I_{1}+I_{2}\]
and
\[
 R_{A}^{\perp}=\int_0^t\mathbb{P}_{1}\mathbb G_*^{t-\tau}h_{A}^{(1)}(\tau)d\tau+\int_0^t\mathbb{P}_{1}\mathbb G_*^{t-\tau}h_{A}^{(0)}(\tau)\,d\tau=I_{1}^{\perp}+I_{2}^{\perp}\,.
\]

We further split the smooth fluid component as
\begin{equation}\label{sec32-split-B12}
 R_{i,B}=R_{i,B_{1}}+R_{i,B_{2}}.
\end{equation}

The two parts are defined by
\begin{align}\label{sec32-R-B1-definition}
 R_{i,B_{1}}
&=\int_0^t\langle E_i,\mathbb G_i^{t-\tau}h_{A}^{(1)}(\tau)\rangle_\xi\,d\tau+\int_0^t\langle E_i,\mathbb G_i^{t-\tau}h_{A}^{(0)}(\tau)\rangle_\xi\,d\tau
\notag\\[2mm]
&\quad+\int_0^t\langle E_i,\mathbb G_i^{t-\tau}\left[\Gamma(R_{B},R_{B})+2\Gamma(R_{B},V_{B}^{\perp})+2\Gamma(R_{B}^{\perp},\Phi_{B})\right](\tau)\rangle_\xi\,d\tau
\notag\\[2mm]
&\quad+\int_0^t\left\langle E_i,\mathbb G_i^{t-\tau}\Bigg[\sum_{j\neq k}2\Gamma\big(R_{j,B}E_{j},(\tilde\phi_{k}+V_{k,B})E_{k}\big)\Bigg](\tau)\right\rangle_\xi\,d\tau
\notag\\[2mm]
&\quad+\int_0^t\left\langle E_i,\mathbb G_i^{t-\tau}2\Gamma\big(R_{i,B}E_{i},(\tilde\phi_{i}+V_{i,B})E_{i}\big)(\tau)\right\rangle_\xi\,d\tau
\notag\\[2mm]
&\quad+\int_0^t\left\langle E_i,\mathbb G_i^{t-\tau}\Bigg[\sum_{j\neq i}2\Gamma\big(R_{j,B}E_{j},(\tilde\phi_{j}+V_{j,B})E_{j}\big)\Bigg](\tau)\right\rangle_\xi\,d\tau
\notag\\[2mm]
&\quad+\int_0^t\left\langle E_i,\mathbb G_i^{t-\tau}\Bigg[2\Gamma(\Phi_{B},V_{1,B}^{\perp})+2\Gamma\Bigg(\sum_{m=1}^{3}(V_{m,B}E_{m}),\sum_{l=1}^{3}(V_{2,l}^{\perp})\Bigg)+\Gamma\Bigg(V_{B}^{\perp},V_{B}^{\perp}\Bigg)\Bigg](\tau)\right\rangle_\xi\,d\tau
\notag\\[2mm]
&\quad+\int_0^t\left\langle E_i,\mathbb G_i^{t-\tau}\Bigg[\sum_{j\neq k}2\Gamma\Big(\tilde\phi_{j}E_{j},V_{2,k}^{\perp}\Big)\Bigg](\tau)\right\rangle_\xi\,d\tau
\notag\\[2mm]
&\quad+\int_0^t\left\langle E_i,\mathbb G_i^{t-\tau}2\Gamma\Big(\tilde\phi_{i}E_{i},V_{2,i}^{\perp}\Big)(\tau)\right\rangle_\xi\,d\tau
+\int_0^t\left\langle E_i,\mathbb G_i^{t-\tau}\Bigg[\sum_{j\neq i}2\Gamma\Big(\tilde\phi_{j}E_{j},V_{2,j}^{\perp}\Big)\Bigg](\tau)\right\rangle_\xi\,d\tau
\notag\\[2mm]
=&\,K_{1}+\cdots+K_{4}+J_{1}+J_{2}+O_{1}+O_{2}+L_{1}+L_{2}\,,
\end{align}
and
\begin{align}\label{sec32-R-B2-definition}
 R_{i,B_{2}}=&\int_0^t\left\langle E_i,
      \left( \mathbb{G}^{t-\tau}_{\sharp}-\mathbb{G}_{i}^{t-\tau}\right)h_{A}^{(1)}(\tau)\right\rangle_{\xi}d\tau+\int_0^t\left\langle E_i,
      \left( \mathbb{G}^{t-\tau}_{\sharp}-\mathbb{G}_{i}^{t-\tau}\right)h_{A}^{(0)}(\tau)\right\rangle_{\xi}d\tau
      \notag\\[2mm]
      &+\int_{0}^{t}\left\langle E_{i},\left(\mathbb{G}_{\ast}^{t-\tau}+\mathcal{G}^{t-\tau}\right)h_{B}^{(1)}(\tau)\right\rangle_{\xi}d\tau+\int_{0}^{t}\left\langle E_{i},\left(\mathbb{G}_{\ast}^{t-\tau}+\mathcal{G}^{t-\tau}\right)h_{B}^{(0)}(\tau)\right\rangle_{\xi}d\tau
      \notag\\[2mm]
      &+\int_0^t\left\langle E_i,
      \Bigg(\sum_{l\neq i}\mathbb{G}_{l}^{t-\tau}\Bigg)h_{B}^{(1)}(\tau)\right\rangle_{\xi}d\tau
      +\int_0^t\left\langle E_i,
      \Bigg(\sum_{l\neq i}\mathbb{G}_{l}^{t-\tau}\Bigg)h_{B}^{(0)}(\tau)\right\rangle_{\xi}d\tau
      \notag\\[2mm]
      =&\,N_{1}+N_{2}+M_{1}+M_{2}+Q_{1}+Q_{2}\,.
\end{align}
The component \(R_{i,B_{1}}\) collects all terms involving the diagonal spectral kernel $\mathbb G_i$. Finally, we decompose $R_{B}^{\perp}$ as
\begin{align}\label{sec32-R-B-perp-definition}
 R_{B}^{\perp}=&\int_0^t\mathbb{P}_{1}
       \mathbb{G}^{t-\tau}_{\sharp}h_{A}^{(1)}(\tau)d\tau+\int_0^t\mathbb{P}_{1}
       \mathbb{G}^{t-\tau}_{\sharp}h_{A}^{(0)}(\tau)d\tau
      \notag\\[2mm]
&+\int_{0}^{t}\mathbb{P}_{1}\left(\mathbb{G}_{\ast}^{t-\tau}+\mathcal{G}^{t-\tau}\right)h_{B}^{(1)}(\tau)d\tau+\int_{0}^{t}\mathbb{P}_{1}\left(\mathbb{G}_{\ast}^{t-\tau}+\mathcal{G}^{t-\tau}\right)h_{B}^{(0)}(\tau)d\tau
      \notag\\[2mm]
&+\int_0^t\mathbb{P}_{1}\mathbb G_{L;0}^{t-\tau}\left[\Gamma(R_{B},R_{B})+2\Gamma(R_{B},V_{B}^{\perp})+2\Gamma(R_{B}^{\perp},\Phi_{B})\right](\tau)\,d\tau
\notag\\[2mm]
&+\int_0^t\mathbb{P}_{1} G_{L;0}^{t-\tau}\Bigg[\sum_{j\neq k}2\Gamma\big(R_{j,B}E_{j},(\tilde\phi_{k}+V_{k,B})E_{k}\big)\Bigg](\tau)\,d\tau
\notag\\[2mm]
&+\int_0^t\mathbb{P}_{1}\Bigg[\sum_{m=1}^{3}\mathbb G_m^{t-\tau}2\Gamma\big(R_{m,B}E_{m},(\tilde\phi_{m}+V_{m,B})E_{m}\big)\Bigg](\tau)\,d\tau
\notag\\[2mm]
&+\left[\sum_{i=1}^{3}\int_0^t\mathbb{P}_{1}\Bigg[\sum_{j\neq i}\mathbb G_i^{t-\tau}2\Gamma\big(R_{j,B}E_{j},(\tilde\phi_{j}+V_{j,B})E_{j}\big)\Bigg](\tau)\,d\tau\right]
\notag\\[2mm]
&+\int_0^t\mathbb{P}_{1}\mathbb G_{L;0}^{t-\tau}\Bigg[2\Gamma(\Phi_{B},V_{1,B}^{\perp})+2\Gamma\Bigg(\sum_{m=1}^{3}(V_{m,B}E_{m}),\sum_{l=1}^{3}(V_{2,l}^{\perp})\Bigg)+\Gamma\bigg(V_{B}^{\perp},V_{B}^{\perp}\bigg)\Bigg](\tau)\,d\tau
\notag\\[2mm]
&+\int_0^t\mathbb{P}_{1}\mathbb G_{L;0}^{t-\tau}\Bigg[\sum_{j\neq k}2\Gamma\Big(\tilde\phi_{j}E_{j},V_{2,k}^{\perp}\Big)\Bigg](\tau)\,d\tau+\int_0^t\mathbb{P}_{1}\Bigg[\sum_{m=1}^{3}\mathbb G_m^{t-\tau}2\Gamma\Big(\tilde\phi_{m}E_{m},V_{2,m}^{\perp}\Big)\Bigg](\tau)\,d\tau
\notag\\[2mm]
&\quad+\int_0^t\mathbb{P}_{1}\Bigg[\sum_{i=1}^{3}\mathbb G_i^{t-\tau}\Bigg(\sum_{j\neq i}2\Gamma\Big(\tilde\phi_{j}E_{j},V_{2,j}^{\perp}\Big)\Bigg)\Bigg](\tau)\,d\tau
\notag\\[2mm]
=&K_{1}^{\perp}+K_{2}^{\perp}+M_{1}^{\perp}+M_{2}^{\perp}+K_{3}^{\perp}+K_{4}^{\perp}+J_{1}^{\perp}+J_{2}^{\perp}+O_1^{\perp}+O_2^{\perp}+L_1^{\perp}+L_{2}^{\perp}\,.
\end{align}
 Throughout this proof, \(s\) and \(n\) denote differentiation orders, with $s\in \left\{0,1\right\}$ and $n\in \mathbb{N}\cup \left\{0\right\}$. We first estimate the terms involving only the leading waves. We begin with the non-smooth components $R_{i,A}$ and $R_{A}^{\perp}$. In view of Lemma~\ref{Prop-linear-fluid}, Proposition~\ref{PF-estimate}, and \eqref{est-theta-1-j-times-theta-1-k}, one yields
\[
|I_{2}|+|I_2^{\perp}|_{L^\infty_{\xi,\beta}}\lesssim\varepsilon^2 (1+t+|x|)^{-N_0/(1-\gamma)}\,.
\]
We next turn to $R_{i,B}$. For the component $R_{i,B_{1}}$, we have
\[
\vert \partial_{x}^{s}K_{2}\vert\lesssim \varepsilon^{2}\theta_{i}^{2+s}\,,
\]
\[
\vert O_{1}\vert\lesssim \varepsilon^{2}(\theta_{i}^{2})\ast_{x,t}\Bigg(\sum_{m=1}^{3}\theta_{m}^{\frac{7}{2}}+(1+t)^{-\frac{1}{2}}\psi_{m}^{3}\Bigg)\lesssim \varepsilon^{2}(1+t)^{-\frac{1}{4}}\Bigg(\psi_{i}^{\frac{3}{2}}+\sum_{k\neq i}\tilde\psi_{k}^{\frac{3}{2}}\Bigg)\,,
\]
\[
\vert\partial_{x}O_{1}\vert\lesssim \varepsilon^{2}(\theta_{i}^{3})\ast_{x,t}\Bigg(\sum_{m=1}^{3}\theta_{m}^{\frac{7}{2}}+(1+t)^{-\frac{1}{2}}\psi_{m}^{3}\Bigg)\lesssim \varepsilon^{2}(1+t)^{-\frac{3}{4}}\ln(1+t)\Bigg(\psi_{i}^{\frac{3}{2}}+\sum_{k\neq i}\tilde\psi_{k}^{\frac{3}{2}}\Bigg)\,,
\]
\[
 \vert \partial_{x}^{s}O_{2}\vert\lesssim \varepsilon^{2}(\theta_{i}^{2+s})\ast_{x,t}\Bigg(\sum_{m=1}^{3}(1+t)^{-1}\psi_{m}^{2N}\Bigg)\lesssim \varepsilon^{2}(1+t)^{-(\frac{1}{4}+\frac{s}{2})}\Bigg(\psi_{i}^{\frac{3}{2}}+\sum_{k\neq i}\tilde\psi_{k}^{\frac{3}{2}}\Bigg)\,,
\]
and
\[
 \vert \partial_{x}^{s}L_{1}\vert\lesssim \varepsilon^{2}(1+t)^{-\frac{s}{2}}(\theta_{i}^{2})\ast_{x,t}(\theta_{i}^{3})\lesssim \varepsilon^{2}(1+t)^{-\frac{s}{2}}\theta_{i}^{2-\delta_{0}}\,.
\]
To estimate $L_{2}$, we use the Cancellation Mechanism. Define kernel $H$ by
\[
H(t,x)=\frac{1}{2\pi}\int_{\mathbb{R}}e^{ix\eta+\sigma_{i}(\eta)t}(i)^{-1}\chi_{\delta}(\eta)\left\langle e_{i}(\eta),E_{i}\right\rangle_{\xi}d\eta\,,
\]
and let \(H_1\) be the corresponding kernel in \eqref{asumption-cancellation-mechanism-1}.
Fix \(i\ne j\). For \(t\ge4\), each component of $L_{2}$ can be written as
\[
L_{2}^{ij}\equiv\left(\int_{0}^{\sqrt{t}}+\int_{\sqrt{t}}^{t}\right)\partial_{x}H(t-\tau)\ast_{x}\mathcal{Q}(\tilde\phi_{j}E_{j},V_{2,j}^{\perp})(\tau)d\tau=L^{ij}_{2\left\{[0,\sqrt{t}]\right\}}+L^{ij}_{2\left\{[\sqrt{t},t]\right\}}\,,
\]
where $\mathcal{Q}(\cdot,\cdot)$ is defined by
\[
\mathcal{Q}(g,h)=\mathcal{F}^{-1}\left(\rho(\eta)\frac{2}{\eta}\left\langle \widehat{\Gamma(g,h)},e_{i}(-\eta)\right\rangle_{\xi}\right)\,,
\]
where $\rho$ is a cutoff function satisfying $\rho=1$ on $ \text{supp} \  \chi_{\delta}$.
The first term is estimated directly:
\[
\vert \partial_{x}^{s}L^{ij}_{2\left\{[0,\sqrt{t}]\right\}}\vert \lesssim\varepsilon^{2}\int_{0}^{\sqrt{t}}\theta_{i}^{2+s}(t-\tau)\ast \theta_{j}^{3}(\tau)d\tau\lesssim\varepsilon^{2}\theta_{i}^{2-\frac{\delta_{0}}{2}+s}\,.
\]
Applying the Cancellation Mechanism together with \eqref{tilde-phi-j-cancell-1} and \eqref{V2-characteristic-bound}, we obtain
\begin{align*}
 &\left\vert(\lambda_j-\lambda_i)
     \int_{\sqrt t}^t\partial_x^{s+1}H(t-\tau)*_x\mathcal{Q}(\tilde\phi_{j}E_{j},V_{2,j}^{\perp})(\tau)\right\vert\,d\tau\\
 &\leq\left\vert-\Bigg[\partial_x^sH(t-\tau)*_x\mathcal{Q}(\tilde\phi_{j}E_{j},V_{2,j}^{\perp})\Bigg]_{\sqrt t}^{t}
   -\int_{\sqrt t}^t\partial_x^sH_1(t-\tau)*_x\partial_x\mathcal{Q}(\tilde\phi_{j}E_{j},V_{2,j}^{\perp})(\tau)\right\vert\,d\tau\\
 &\quad+\left\vert\int_{\sqrt t}^t\partial_x^sH(t-\tau)*_x
(\partial_\tau+\lambda_j\partial_x)\mathcal{Q}(\tilde\phi_{j}E_{j},V_{2,j}^{\perp})(\tau)\,d\tau\right\vert
\\
&\lesssim\varepsilon^2\Big[\theta_j^{3+s}
   +\theta_i^{1+s}(t-\sqrt t)*_x\theta_j^3(\sqrt t)+\int_{\sqrt t}^{\frac{t}{2}}\theta_i^{2+s}(t-\tau)*_x\theta_j^4(\tau)\,d\tau+\int_{\frac{t}{2}}^{t}\theta_i^{2}(t-\tau)*_x\theta_j^{4+s}(\tau)\,d\tau\\
 &+\int_{\sqrt t}^{\frac{t}{2}}\theta_i^{1+s}(t-\tau)*_x\theta_j^5(\tau)\,d\tau+\int_{\frac{t}{2}}^t\theta_i^1(t-\tau)*_x\theta_j^{5+s}(\tau)\,d\tau\Big]\\
 &\lesssim\varepsilon^2(1+t)^{-\frac{1}{4}-\frac{s}{2}}\big(\psi_i^{3/2}+\tilde\psi_j^{3/2}\big).
\end{align*}
Combining the estimates on the two time intervals gives
\[
\vert \partial_{x}^{s}L_{2}\vert\lesssim \varepsilon^{2}\left[(1+t)^{-\frac{1}{4}-\frac{s}{2}}\Bigg(\psi_i^{3/2}+\sum_{k\neq i}\tilde\psi_k^{3/2}\Bigg)+\theta_{i}^{2-\frac{\delta_{0}}{2}+s}\right]\,.
\]
We next estimate $R_{i,B_{2}}$. We obtain
\[
\vert \partial_{x}^{s}N_{2}\vert\lesssim \varepsilon^{2}\left(\sum_{k\neq i}\theta_{k}^{3+s}\right)\,, \ \ \ \vert \partial_{x}^{s}M_{2}\vert\lesssim\varepsilon^{2}\Bigg(\sum_{m=1}^{3}\theta_{m}^{3+s}+(1+t)^{-\frac{1}{2}}\psi_{m}^{3}\Bigg)\,,
\]
\[
\vert Q_{2}\vert\lesssim \varepsilon^{2}\Bigg(\sum_{l\neq i}\theta^{3}_{l}\Bigg)\ast_{x,t}\Bigg(\sum_{m=1}^{3}\theta_{m}^{3}+(1+t)^{-\frac{1}{2}}\psi_{m}^{3}\Bigg)\lesssim \varepsilon^{2}(1+t)^{-\frac{1}{2}}\ln(1+t)\Bigg(\sum_{m=1}^{3}\tilde\psi_{m}^{\frac{3}{2}}\Bigg)\,,
\]
\begin{align*}
\vert \partial_{x}Q_{2}\vert&\lesssim\varepsilon^{2}\left[\int_{0}^{\frac{t}{2}}\Bigg(\sum_{l\neq i}\theta^{4}_{l}\Bigg)(t-\tau)\ast_{x}\Bigg(\sum_{m=1}^{3}\theta_{m}^{3}+(1+t)^{-\frac{1}{2}}\psi_{m}^{3}\Bigg)(\tau)d\tau\right.
\\
&\left.\quad+\int_{\frac{t}{2}}^{t}\Bigg(\sum_{l\neq i}\theta^{3}_{l}\Bigg)(t-\tau)\ast_{x}\Bigg(\sum_{m=1}^{3}\theta_{m}^{4}+(1+t)^{-\frac{1}{2}}\psi_{m}^{3}\Bigg)(\tau)d\tau\right]
\\
&\lesssim \varepsilon^{2}(1+t)^{-1}\ln(1+t)\Bigg(\sum_{m=1}^{3}\tilde\psi_{m}^{\frac{3}{2}}\Bigg)\,.
\end{align*}

For the non-fluid component $R_{B}^{\perp}$, the same argument gives
\[
\vert \partial_{x}^{s}K_{2}^{\perp}\vert_{L^{\infty}_{\xi,\beta}}\lesssim\varepsilon^{2}\left(\sum_{m=1}^{3}\theta_{m}^{3+s}\right)\,, \ \ \ \vert \partial_{x}^{s}M_{2}^{\perp}\vert_{L^{\infty}_{\xi,\beta}}\lesssim\varepsilon^{2}\Bigg(\sum_{m=1}^{3}\theta_{m}^{3+s}+(1+t)^{-\frac{1}{2}}\psi_{m}^{3}\Bigg)\,,
\]
\begin{align*}
\vert O_{1}^{\perp}\vert_{L^{\infty}_{\xi,\beta}}
&\lesssim \varepsilon^{2}\int_{0}^{t}\Bigg(\sum_{i=1}^{3}\theta_{i}^{3}\Bigg)(t-\tau)\ast_{x}\Bigg(\sum_{m=1}^{3}\theta_{m}^{\frac{7}{2}}+(1+t)^{-\frac{1}{2}}\psi_{m}^{3}\Bigg)(\tau)d\tau
\\
&\lesssim \varepsilon^{2}\ln(1+t)(1+t)^{-\frac{3}{4}}\Bigg(\sum_{m=1}^{3}\psi_{m}^{\frac{3}{2}}\Bigg)\,,
\end{align*}
\begin{align*}
\vert \partial_{x}O_{1}^{\perp}\vert_{L^{\infty}_{\xi,\beta}}
&\lesssim \varepsilon^{2}\int_{0}^{t}\Bigg(\sum_{i=1}^{3}\theta_{i}^{4}\Bigg)(t-\tau)\ast_{x}\Bigg(\sum_{m=1}^{3}\theta_{m}^{\frac{7}{2}}+(1+t)^{-\frac{1}{2}}\psi_{m}^{3}\Bigg)(\tau)d\tau
\\
&\lesssim \varepsilon^{2}(1+t)^{-\frac{3}{4}}\Bigg(\sum_{m=1}^{3} \tilde\psi_{m}^{\frac{3}{2}}\Bigg)\,,
\end{align*}
\[
 \vert \partial_{x}^{s}O_{2}^{\perp}\vert_{L^{\infty}_{\xi,\beta}}\lesssim \varepsilon^{2}\Bigg(\sum_{l=1}^{3}\theta_{l}^{3+s}\Bigg)\ast_{x,t}\Bigg(\sum_{m=1}^{3}(1+t)^{-1}\psi_{m}^{2N}\Bigg)\lesssim \varepsilon^{2}(1+t)^{-(\frac{3}{4}+\frac{s}{2})}\Bigg(\sum_{m=1}^{3}\psi_{m}^{\frac{3}{2}}\Bigg)
\]
and
\[
 \vert \partial_{x}^{s}L_{1}^{\perp}\vert_{L^{\infty}_{\xi,\beta}}\lesssim \varepsilon^{2}\Bigg[\sum_{m=1}^{3}(1+t)^{-(\frac{s+1}{2})}(\theta_{m}^{2})\ast_{x,t}(\theta_{m}^{3})\Bigg]\lesssim \varepsilon^{2}(1+t)^{-(\frac{s+1}{2})}\Bigg(\sum_{m=1}^{3}\theta_{m}^{2-\delta_{0}}\Bigg)\,.
\]
Applying the Cancellation Mechanism to $L^{\perp}_{2}$ gives
\[
\vert \partial_{x}^{s}L_{2}^{\perp}\vert_{L^{\infty}_{\xi,\beta}}\lesssim \varepsilon^{2}\Bigg[\sum_{m=1}^{3}(1+t)^{-\frac{3}{4}-\frac{s}{2}}\psi_m^{3/2}+\theta_{m}^{3-\frac{\delta_{0}}{2}+s}\Bigg]\,.
\]
The argument is the same as for $L_{2}$, so we omit the details.
Motivated by these estimates, we assume that the remainder $R$ satisfies the following ansatz:

\begin{align}
 |R_{i,A}|+|R_A^\perp|_{L^\infty_{\xi,\beta}}
 &\leq M\varepsilon^2(1+t+|x|)^{-\frac{N_0}{1-\gamma}},
 \label{ansatz-R-A}\\
  |\partial_x^sR_{i,B1}|
 &\leq M\varepsilon^2\Big[\theta_i^{2+s-\delta_0}+(1+t)^{-\frac{1-\delta_{0}}{4}-\frac{s}{2}}
       \Big(\psi_i^{3/2}+\sum_{m\ne i}(\tilde\psi_m^{3/2})\Big)\Big]\,,
\label{ansatz-R-i-B1}\\
 |\partial_x^sR_{i,B2}|
 &\leq M\varepsilon^2(1+t)^{-\frac{1-\delta_{0}}{4}-\frac{s}{2}}\Bigg(\sum_{m=1}^{3}\tilde\psi_m^{3/2}\Bigg)\,,
\label{ansatz-R-i-B2}\\
 |R_B^\perp|_{L^\infty_{\xi,\beta}}
 &\leq M\varepsilon^2\sum_{m}^{3}\Big[
       \theta_m^{3-\delta_0}
       +(1+t)^{-\frac{3-\delta_{0}}{4}}\psi_m^{3/2}\Big]\,,
\label{ansatz-R-B-perp}\\
|\partial_xR_B^\perp|_{L^\infty_{\xi,\beta}}
 &\leq M\varepsilon^2\sum_{m=1}^{3}\Big[
       \theta_m^{4-2\delta_0}
       +(1+t)^{-\frac{3-\delta_{0}}{4}}\tilde\psi_m^{3/2}\Big]\,,
\label{ansatz-der-R-B-perp}
\end{align}
where the constant $M$ will be chosen later.
We now estimate the terms involving the remainder $R$.
Only $I_{1}^{\perp}$ in $R_{A}^{\perp}$ and $M_{1}^{\perp}$ in $R_{B}^{\perp}$ require separate treatment of the high-velocity weights.
We first estimate the remaining terms and postpone the treatment of $I_{1}^{\perp}$ and $M_{1}^{\perp}$.
For $R_{i,A}$, Proposition~\ref{PF-estimate} gives
\[
|I_{1}|\lesssim(\varepsilon^4M^2+\varepsilon^{3}M)(1+t+|x|)^{-\frac{N_0}{1-\gamma}}\,.
\]
For $R_{i,B_{1}}$, Lemma~\ref{Prop-linear-fluid}, Proposition~\ref{PF-estimate}, and \eqref{est-theta-1-j-times-theta-1-k} yield
\[
\vert \partial_{x}^{s}K_{1}\vert\lesssim (\varepsilon^4M^2+\varepsilon^{3}M)\theta_{i}^{2+s}\,,
\]
\[
\vert K_{3}\vert\lesssim\varepsilon^3M(\varepsilon M+1)(\theta_{i}^{2})\ast_{x,t}\Bigg(\sum_{k=1}^{3}(1+t)^{-\frac{1-\delta_{0}}{4}}\psi_{k}^{3}\Bigg)
\lesssim\varepsilon^3M(\varepsilon M+1)(1+t)^{-\frac{1-\delta_{0}}{4}}\Bigg(\psi_{i}^{\frac{3}{2}}+\sum_{k\neq i}\tilde\psi_{k}^{\frac{3}{2}}\Bigg)\,,
\]
\begin{align*}
\vert \partial_{x}K_{3}\vert&\lesssim\varepsilon^3M(\varepsilon M+1)\left[\int_{0}^{\frac{t}{2}}(\theta_{i}^{3})(t-\tau)\ast_{x}\Bigg(\sum_{k=1}^{3}(1+t)^{-\frac{1-\delta_{0}}{4}}\psi_{k}^{3}\Bigg)(\tau)d\tau\right.
\\
&\quad\left.+\int_{\frac{t}{2}}^{t}(\theta_{i}^{2})(t-\tau)\ast_{x}\Bigg(\sum_{k=1}^{3}(1+t)^{-\frac{3-\delta_{0}}{4}}\psi_{k}^{3}\Bigg)(\tau)d\tau\right]
\\
&\lesssim\varepsilon^3M(\varepsilon M+1)(1+t)^{-\frac{3-\delta_{0}}{4}}\Bigg(\psi_{i}^{\frac{3}{2}}+\sum_{k\neq i}\tilde\psi_{k}^{\frac{3}{2}}\Bigg)\,,
\end{align*}
\begin{align*}
 \vert \partial_{x}^{s}K_{4}\vert&\lesssim (\varepsilon^4M^2+\varepsilon^{3}M)(1+t)^{-\frac{s}{2}}(\theta_{i}^{2})\ast_{x,t}\Bigg(\sum_{m=1}^{3}(1+t)^{-\frac{1-\delta_{0}}{4}}\psi_{m}^{3}\Bigg)
 \\
 &\lesssim (\varepsilon^4M^2+\varepsilon^{3}M)(1+t)^{-(\frac{1-\delta_{0}}{4}+\frac{s}{2})}\Bigg(\psi_{i}^{\frac{3}{2}}+\sum_{k\neq i}\tilde\psi_{k}^{\frac{3}{2}}\Bigg)\,,
\end{align*}
and
\begin{align*}
 \vert \partial_{x}^{s}J_{1}\vert&\lesssim (\varepsilon^4M^2+\varepsilon^{3}M)(1+t)^{-\frac{s}{2}}(\theta_{i}^{2})\ast_{x,t}\Bigg[\theta_{i}^{3-\delta_{0}}+\sum_{m=1}^{3}\left[(1+t)^{-\frac{1-\delta_{0}}{4}}\psi_{m}^{3}\right]\Bigg]
 \\
 &\lesssim (\varepsilon^4M^2+\varepsilon^{3}M)(1+t)^{-\frac{s}{2}}\Bigg[\theta_{i}^{2-\delta_{0}}+(1+t)^{-\frac{1-\delta_{0}}{4}}\Bigg(\psi_{i}^{\frac{3}{2}}+\sum_{k\neq i}\tilde\psi_{k}^{\frac{3}{2}}\Bigg)\Bigg]\,.
\end{align*}
The terms \(J_2\) and \(J_2^\perp\) will be treated together below. For $R_{i,B_{2}}$, Lemma~\ref{Prop-linear-fluid}, Proposition~\ref{PF-estimate}, and \eqref{est-theta-1-j-times-theta-1-k} give
\[
\vert \partial_{x}^{s}N_{1}\vert\lesssim(\varepsilon^4M^2+\varepsilon^{3}M)\left(\sum_{k\neq i}\theta_{k}^{3+s}\right)\,,
\]
\[
  \vert \partial_{x}^{s}M_{1}\vert\lesssim(\varepsilon^4M^2+\varepsilon^{3}M)\Bigg(\sum_{m=1}^{3}\theta_{m}^{3-\delta_{0}+s(1-\delta_{0})}+(1+t)^{-\frac{3-\delta_{0}}{4}}\psi_{m}^{\frac{3+s}{2}}\Bigg)\,,
\]
\begin{align*}
\vert \partial_{x}^{s}Q_{1}\vert&\lesssim(\varepsilon^4M^2+\varepsilon^{3}M)(1+t)^{-\frac{s}{2}}\Bigg(\sum_{l\neq i}\theta^{3}_{l}\Bigg)\ast_{x,t}\Bigg(\sum_{m=1}^{3}\theta_{m}^{3-\delta_{0}}+(1+t)^{-\frac{1-\delta_{0}}{4}}\psi_{m}^{3}\Bigg)
\\
&\lesssim(\varepsilon^4M^2+\varepsilon^{3}M)(1+t)^{-\frac{1-\delta_{0}+2s}{4}}\Bigg(\sum_{m=1}^{3}\tilde\psi_{m}^{\frac{3}{2}}\Bigg)\,.
\end{align*}
For $R_{B}^{\perp}$, we have
\[
\vert \partial_{x}^{s}K_{1}^{\perp}\vert_{L^{\infty}_{\xi,\beta}}\lesssim(\varepsilon^4M^2+\varepsilon^{3}M)\left(\sum_{m=1}^{3}\theta_{m}^{3+s}\right)\,.
\]
For $K_{3}^{\perp}$, $K_{4}^{\perp}$, and $J_{1}^{\perp}$, we transfer one Fourier factor $\eta$ to the source on the time interval $[\frac{t}{2},t]$ to obtain
\begin{align*}
\vert \partial_{x}^{s} K_{3}^{\perp}\vert_{L^{\infty}_{\xi,\beta}}\,, \  \vert \partial_{x}^{s} K_{4}^{\perp}\vert_{L^{\infty}_{\xi,\beta}}&\lesssim\varepsilon^3M(\varepsilon M+1)\left[\int_{0}^{\frac{t}{2}}\Bigg(\sum_{m=1}^{3}\theta_{m}^{3+s}\Bigg)(t-\tau)\ast_{x}\Bigg(\sum_{k=1}^{3}(1+t)^{-\frac{1-\delta_{0}}{4}}\psi_{k}^{3}\Bigg)(\tau)d\tau\right.
\\
&\quad\left.+\int_{\frac{t}{2}}^{t}\Bigg(\sum_{m=1}^{3}\theta_{m}^{2+s}\Bigg)(t-\tau)\ast_{x}\Bigg(\sum_{k=1}^{3}(1+t)^{-\frac{3-\delta_{0}}{4}}\psi_{k}^{3}\Bigg)(\tau)d\tau\right]
\\
&\lesssim\varepsilon^3M(\varepsilon M+1)(1+t)^{-\frac{3-\delta_{0}}{4}-\frac{s}{2}}\Bigg(\sum_{m=1}^{3}\psi_{m}^{\frac{3}{2}}\Bigg)\,,
\end{align*}
\begin{align*}
\vert \partial_{x}^{s} J_{1}^{\perp}\vert_{L^{\infty}_{\xi,\beta}}&\lesssim\varepsilon^3M(\varepsilon M+1)\left[\int_{0}^{\frac{t}{2}}\Bigg[\sum_{m=1}^{3}\theta_{m}^{3+s}(t-\tau)\ast_{x}\Bigg(\theta_{m}^{3-\delta_{0}}+\sum_{k=1}^{3}\bigg((1+t)^{-\frac{1-\delta_{0}}{4}}\psi_{k}^{3}\bigg)\Bigg)(\tau)\Bigg]d\tau\right.
\\
&\quad\left.+\int_{\frac{t}{2}}^{t}\Bigg[\sum_{m=1}^{3}\theta_{m}^{2+s}(t-\tau)\ast_{x}\Bigg(\theta_{m}^{4-\delta_{0}}+\sum_{k=1}^{3}\bigg((1+t)^{-\frac{3-\delta_{0}}{4}}\psi_{k}^{3}\bigg)\Bigg)(\tau)\Bigg]d\tau\right]
\\
&\lesssim\varepsilon^3M(\varepsilon M+1)\left[\sum_{m=1}^{3}\theta_{m}^{3+s(1-\delta_{0})-\delta_{0}}+(1+t)^{-\frac{3-\delta_{0}}{4}-\frac{s}{2}}\psi_{m}^{\frac{3}{2}}\right]\,.
\end{align*}
We next apply the Cancellation Mechanism to $J_{2}$ and $J_{2}^{\perp}$.
For each component $\tilde J_{(i,j)}$ of $J_{2}$ and $\tilde J_{(i,j)}^{\perp}$ of $J_{2}^{\perp}$, we define
\[
\tilde J_{(i,j)}=\int_0^t\left\langle E_i,\mathbb G_i^{t-\tau}2\Gamma\big(R_{j,B}E_{j},(\tilde\phi_{j}+V_{j,B})E_{j}\big)(\tau)\right\rangle_\xi\,d\tau\,,
\]
and
\[
\tilde J_{(i,j)}^{\perp}=\int_0^t\mathbb{P}_{1}\mathbb G_i^{t-\tau}2\Gamma\big(R_{j,B}E_{j},(\tilde\phi_{j}+V_{j,B})E_{j}\big)(\tau)\,d\tau\,,
\]
for $1\leq j\leq 3$ and $j\neq i$. We divide their source $2\Gamma\big(R_{j,B}E_{j},(\tilde\phi_{j}+V_{j,B})E_{j}\big)$ into
\begin{equation}
\label{eq-mathcal-Aj}
\mathcal{A}_{j}=2\Gamma\big(R_{j,B_{2}}E_{j},(\tilde\phi_{j}+V_{j,B})E_{j}\big)+2\Gamma\big(R_{j,B_{1}}E_{j},V_{j,B_{2}}E_{j}\big)\,,
\end{equation}
and
\begin{equation}
\label{eq-mathcal-Bj}
\mathcal{B}_{j}=2\Gamma\big(R_{j,B_{1}}E_{j},(\tilde\phi_{j}+V_{j,B_{1}})E_{j}\big)\,.
\end{equation}
For convenience, set
\[
\tilde{J}_{(i,j)}\bigg\vert_{\mathcal{A}_{j}}=\int_0^t\left\langle E_i,\mathbb G_i^{t-\tau}\mathcal{A}_{j}(\tau)\right\rangle_\xi\,d\tau\,,\qquad
\tilde{J}_{(i,j)}\bigg\vert_{\mathcal{B}_{j}}=\int_0^t\left\langle E_i,\mathbb G_i^{t-\tau}\mathcal{B}_{j}(\tau)\right\rangle_\xi\,d\tau\,,
\]
and
\[
\tilde{J}_{(i,j)}^{\perp}\bigg\vert_{\mathcal{A}_{j}}=\int_0^t\mathbb{P}_{1}\mathbb G_i^{t-\tau}\mathcal{A}_{j}(\tau)\,d\tau\,,\qquad
\tilde{J}_{(i,j)}^{\perp}\bigg\vert_{\mathcal{B}_{j}}=\int_0^t\mathbb{P}_{1}\mathbb G_i^{t-\tau}\mathcal{B}_{j}(\tau)\,d\tau\,.
\]
For $\tilde{J}_{(i,j)}\big\vert_{\mathcal{A}_{j}}$,
we can estimate directly:
\begin{align*}
\left\vert\partial_{x}^{s}\tilde{J}_{(i,j)}\right\vert\Bigg\vert_{\mathcal{A}_{j}}&\lesssim\varepsilon^3M(\varepsilon M+1)(1+t)^{-\frac{s}{2}}(\theta_{i}^{2})\ast_{x,t}\left[(1+t)^{-\frac{1-\delta_{0}}{4}}\Bigg(\sum_{m=1}^{3}\psi_{m}^{3}\Bigg)\right]
\\
&\lesssim\varepsilon^3M(\varepsilon M+1)(1+t)^{-\frac{s}{2}}(1+t)^{-\frac{1-\delta_{0}}{4}}\Bigg(\psi_{i}^{\frac{3}{2}}+\sum_{k\neq i}\tilde\psi_{k}^{\frac{3}{2}}\Bigg)\,.
\end{align*}
For $\tilde{J}_{(i,j)}\big\vert_{\mathcal{B}_{j}}$, we apply the Cancellation Mechanism.
Using the definition in \eqref{def-H-ij}, we write $\tilde{J}_{(i,j)}\vert_{\mathcal{B}}$ as
\begin{align*}
\tilde{J}_{(i,j)}\big\vert_{\mathcal{B}_{j}}&=\left(\int_{0}^{\sqrt{t}}+\int_{\sqrt{t}}^{t}\right)\partial_{x}H_{ij}(t-\tau)\ast_{x}2\left[R_{j,B_{1}}\tilde\phi_{j}+R_{j,B_{1}}V_{j,B_{1}}\right](\tau)d\tau
\\[2mm]
&\equiv\tilde{J}_{(i,j)\left\{[0,\sqrt{t}]\right\}}\Big\vert_{\mathcal{B}_{j}}+\tilde{J}_{(i,j)\left\{[\sqrt{t},t]\right\}}\Big\vert_{\mathcal{B}_{j}}\,.
\end{align*}
The ansatz \eqref{ansatz-R-i-B1} for $R_{j,B_{1}}$ and the pointwise estimates of $\tilde\phi_{j}$ and $V_{j,B_{1}}$ give
\begin{align*}
\left\vert\partial_{x}^{s}\tilde{J}_{(i,j)\left\{[0,\sqrt{t}]\right\}}\Big\vert_{\mathcal{B}_{j}}\right\vert&\lesssim\varepsilon^3M(\varepsilon M+1)\int_{0}^{\sqrt{t}}\theta_{i}^{2+s}(t-\tau)\ast_{x}\left(\theta_{j}^{3-\delta_{0}}+(1+t)^{-\frac{1-\delta_{0}}{4}}\Bigg(\sum_{m=1}^{3}\psi_{m}^{3}\Bigg)\right)(\tau)d\tau
\\
&\lesssim\varepsilon^3M(\varepsilon M+1)(1+t)^{-\frac{s}{2}}\left[\theta_{i}^{2-\delta_{0}}+(1+t)^{-\frac{1-\delta_{0}}{4}}\Bigg(\psi_{i}^{\frac{3}{2}}+\sum_{k\neq i}\tilde\psi_{k}^{\frac{3}{2}}\Bigg)\right]
\end{align*}
The term $R_{j,B_{1}}$ satisfies the Cancellation Mechanism assumption \eqref{asumption-cancellation-mechanism-2} :
\begin{equation}
\widehat{\left(\partial_{t}+\lambda_{j}\partial_{x}\right)R_{j,B_{1}}}=(-A_{j}\eta^{2}+\eta^{3}R_{\sigma,j}(\eta))\widehat{R_{j,B_{1}}}+\mathcal{R}(\eta,t)\,,
\end{equation}
where
\begin{equation*}
\mathcal{R}(\eta,t)=\chi_{\delta}(\eta)\left\langle e_{j}(\eta),E_{j}\right\rangle_{\xi}\left\langle \widehat{(h_{A}+h_{B})}, e_{j}(-\eta)\right\rangle_{\xi}\,.
\end{equation*}
Together with the expression of $(\partial_{t}+\lambda_{j}\partial_{x})V_{j,B_{1}}$ \eqref{cancellation-assumption--V-j-B1}\,, one yields
\begin{align*}
 \left(\lambda_{j}-\lambda_{i}\right)\partial_{x}^{s}\tilde{J}_{(i,j)\left\{[\sqrt{t},t]\right\}}\bigg\vert_{\mathcal{B}_{j}}&= -\big[\partial_x^sH_{ij}(t-\tau)*_x2(R_{j,B_{1}})(\tilde\phi_{j}+V_{j,B_{1}})(\tau)\big]_{\sqrt t}^{t} \\
 &\quad-\int_{\sqrt t}^t\partial_x^sH_1(t-\tau)*_x\partial_x\left[2(R_{j,B_{1}})(\tilde\phi_{j}+V_{j,B_{1}})\right](\tau)\,d\tau\\
 &\quad+\int_{\sqrt t}^t\partial_x^sH_{ij}(t-\tau)*_x
              \left[(\partial_\tau+\lambda_j\partial_x)(2R_{j,B_{1}})\right](\tilde\phi_{j}+V_{j,B_{1}})(\tau)\,d\tau
              \\
&\quad+\int_{\sqrt t}^t\partial_x^sH_{ij}(t-\tau)*_x
(2R_{j,B_{1}})\left[(\partial_\tau+\lambda_j\partial_x)(\tilde\phi_{j}+V_{j,B_{1}})\right](\tau)\,d\tau\,.
\end{align*}
It follows that
\begin{align*}
\left\vert\partial_{x}^{s}\tilde{J}_{(i,j)\left\{[\sqrt{t},t]\right\}}\bigg\vert_{\mathcal{B}_{j}}\right\vert&\lesssim\varepsilon^3M(\varepsilon M+1)\left\{
\left[\theta_{j}^{3-\delta_{0}+s}+(1+t)^{-\frac{1-\delta_{0}+2s}{4}}\Bigg(\sum_{m=1}^{3}\psi_{m}^{3}\Bigg)\right]\right.
\\
&\quad+\theta_{i}^{1+s}(t-\sqrt{t})\ast_{x}\left[\theta_{j}^{3-\delta_{0}}+(1+t)^{-\frac{1-\delta_{0}}{4}}\Bigg(\sum_{m=1}^{3}\psi_{m}^{3}\Bigg)\right](\sqrt{t})
\\
&\quad+\int_{\sqrt{t}}^{t}\theta_{i}^{2+s}(t-\tau)\ast_{x}\left[\theta_{j}^{4-\delta_{0}}+(1+t)^{-\frac{3-\delta_{0}}{4}}\Bigg(\sum_{m=1}^{3}\psi_{m}^{3}\Bigg)\right](\tau)d\tau
\\
&\quad+\int_{\sqrt{t}}^{t}\theta_{i}^{1+s}(t-\tau)\ast_{x}\left[\theta_{j}^{5-\delta_{0}}+(1+t)^{-\frac{5-\delta_{0}}{4}}\Bigg(\sum_{m=1}^{3}\psi_{m}^{3}\Bigg)\right](\tau)d\tau
\\
&\quad\left.+\int_{\sqrt{t}}^{t}\theta_{i}^{2+s}(t-\tau)\ast_{x}\left[\theta_{j}^{4-\delta_{0}}+(1+t)^{-\frac{3-\delta_{0}}{4}}\Bigg(\sum_{m=1}^{3}\psi_{m}^{3}\Bigg)\right](\tau)d\tau\right\}
\\
&\lesssim \varepsilon^3M(\varepsilon M+1)(1+t)^{-\frac{1-\delta_{0}+2s}{4}}\Bigg(\psi_{i}^{\frac{3}{2}}+\sum_{k\neq i}\tilde\psi_{k}^{\frac{3}{2}}\Bigg)\,.
\end{align*}
For the first derivative of $\tilde{J}_{(i,j)}\big\vert_{\mathcal{B}_{j}}$, we further split $[\sqrt{t},t]$ into $[\sqrt{t},\frac{t}{2}]\cup [\frac{t}{2},t]$ and then calculate convolution directly.
This completes the estimate of $J_{2}$ and verifies the ansatz for $R_{i,B_{1}}$ in \eqref{ansatz-R-i-B1}.

We next consider the non-fluid component $\tilde{J}_{(i,j)}^{\perp}$.
For the contributions from $\mathcal{A}_{j}$ defined by \eqref{eq-mathcal-Aj}, direct estimation gives
\begin{align*}
\left\vert\partial_{x}^{s}\tilde{J}_{(i,j)}^{\perp}\right\vert\Bigg\vert_{\mathcal{A}_{j}}&\lesssim\varepsilon^3M(\varepsilon M+1)(1+t)^{-\frac{s}{2}}(\theta_{i}^{3})\ast_{x,t}\left[(1+t)^{-\frac{1-\delta_{0}}{4}}\Bigg(\sum_{m=1}^{3}\psi_{m}^{3}\Bigg)\right]
\\
&\lesssim\varepsilon^3M(\varepsilon M+1)(1+t)^{-\frac{s}{2}}(1+t)^{-\frac{3-\delta_{0}}{4}}\Bigg(\sum_{m=1}^{3}\psi_{m}^{\frac{3}{2}}\Bigg)\,.
\end{align*}
For the contributions from $\mathcal{B}_{j}$ defined by \eqref{eq-mathcal-Bj},
we use the Cancellation Mechanism. We write $\tilde{J}_{(i,j)}^{\perp}\big\vert_{\mathcal{B}_{j}}$ as
\begin{align*}
\tilde{J}_{(i,j)}^{\perp}\big\vert_{\mathcal{B}_{j}}&=\left(\int_{0}^{\sqrt{t}}+\int_{\sqrt{t}}^{t}\right)\partial_{x}^{2}H_{ij}^{\perp}(t-\tau)\ast_{x}\left[2(R_{j,B_{1}})(\tilde\phi_{j}+V_{j,B_{1}})\right](\tau)d\tau
\\[2mm]
&\equiv\tilde{J}^{\perp}_{(i,j)\left\{[0,\sqrt{t}]\right\}}\Big\vert_{\mathcal{B}}+\tilde{J}^{\perp}_{(i,j)\left\{[\sqrt{t},t]\right\}}\Big\vert_{\mathcal{B}}\,,
\end{align*}
where the kernel $H_{ij}^{\perp}$ is defined by
\begin{equation}
\label{def-H-ij-perp}
H_{ij}^{\perp}(t,x)=\frac{1}{2\pi}\int_{\mathbb R}e^{ix\eta+\sigma_{i}t}\chi_{\delta}(\eta)\frac{\left(\mathbb{P}_{1}e_{i}(\eta)\right)\left\langle \Gamma(E_{j},E_{j}),e_{i}(-\eta)\right\rangle_{\xi}}{(i\eta)^{2}}d\eta\,,
\end{equation}
The corresponding kernel $H_{1}^{\perp}$ plays the role of $H_{1}$ in \eqref{asumption-cancellation-mechanism-1} and satisfies \eqref{asumption-cancellation-mechanism-1}.
Repeating the argument for $\tilde{J}_{(i,j)}\vert_{\mathcal{B}}$, and using the ansatz \eqref{ansatz-R-i-B1} for $R_{j,B_{1}}$
together with the pointwise estimates \eqref{est-tilde-phi-i} and \eqref{est-V-i-B} for $\tilde\phi_{j}$ and $V_{j,B_{1}}$, we obtain
\[
\vert \tilde{J}_{(i,j)}^{\perp}\vert\bigg\vert_{\mathcal{B}_{j}}
\lesssim\varepsilon^3M(\varepsilon M+1)\left[\theta_{i}^{3-\delta_{0}}+(1+t)^{-\frac{3-\delta_{0}}{4}}\Bigg(\psi_{i}^{\frac{3}{2}}+\sum_{k\neq i}\tilde\psi_{k}^{\frac{3}{2}}\Bigg)\right]\,.
\]
For the first derivative, we transfer one derivative to the source on $[\frac{t}{2},t]$ to obtain
\begin{align*}
\left\vert \partial_{x}\tilde{J}_{(i,j)}^{\perp}\right\vert\bigg\vert_{\mathcal{B}_{j}}
&\lesssim\varepsilon^3M(\varepsilon M+1)\left\{\int_{0}^\frac{t}{2}\theta^{4}_{i}(t-\tau)\ast_{x}\left[\theta_{j}^{3-\delta_{0}}+(1+t)^{-\frac{1-\delta_{0}}{4}}\Bigg(\sum_{m=1}^{3}\psi_{m}^{3}\Bigg)\right](\tau)d\tau\right.
\\
&\quad\left.+\int_\frac{t}{2}^{t}\theta^{3}_{i}(t-\tau)\ast_{x}\left[\theta_{j}^{4-\delta_{0}}+(1+t)^{-\frac{3-\delta_{0}}{4}}\Bigg(\sum_{m=1}^{3}\psi_{m}^{3}\Bigg)\right](\tau)d\tau\right\}
\\
&\lesssim\varepsilon^3M(\varepsilon M+1)(1+t)^{-\frac{3-\delta_{0}}{4}}\Bigg(\sum_{m=1}^{3}\tilde\psi_{m}^{\frac{3}{2}}\Bigg)\,.
\end{align*}
Thus, $J_{2}$ and $J_{2}^{\perp}$ satisfy the ansatz for $R_{i,B_{1}}$ in \eqref{ansatz-R-i-B1}
and for $R_{B}^{\perp}$ in \eqref{ansatz-R-B-perp}, respectively.

It remains to estimate $I_{1}^{\perp}$ in $R_{A}^{\perp}$ and $M_{1}^{\perp}$ in $R_{B}^{\perp}$. To estimate them, their sources require estimates with higher velocity weights. Too deal with those sources estimates with higher velocity weights, we use weighted interpolation:
\\
For $0\leq\sigma\leq N_0$, weighted interpolation gives
\[
\begin{aligned}
|g(x,t)|_{L_{\xi,\beta+\sigma+\gamma}^\infty}
&\leq
 |g(x,t)|_{L_{\xi,\beta}^\infty}^{1-2^{-k}}
 |g(x,t)|_{L_{\xi,\beta+2^k(\sigma+\gamma)}^\infty}^{2^{-k}}
\leq
 |g(x,t)|_{L_{\xi,\beta}^\infty}^{1-2^{-k}}
 \|g(t)\|_{L_x^\infty
 L_{\xi,\beta+2^k(\sigma+\gamma)}^\infty}^{2^{-k}}\,.
\end{aligned}
\]
Combining this interpolation inequality with Lemma 2.3, we obtain
\begin{align}
|\Gamma(g,h)|_{L_{\xi,\beta+\sigma}^\infty}
&\lesssim
 |g|_{L_{\xi,\beta+\sigma+\gamma}^\infty}
 |h|_{L_{\xi,\beta+\sigma+\gamma}^\infty}
\notag\\
&\lesssim \left(\|g(t)\|_{L_x^\infty
 L_{\xi,\hat{\beta}}^\infty}\|h(t)\|_{L_x^\infty
 L_{\xi,\hat{\beta}}^\infty}\right)^{2^{-k}}
 \left(
 |g|_{L_{\xi,\beta}^\infty}
 |h|_{L_{\xi,\beta}^\infty}
 \right)^{1-2^{-k}}\,,  \ \  \text{for  }\ \hat{\beta}\geq\beta+2^k(\sigma+\gamma)\,.
 \label{eq-interpolation high velocity}
\end{align}
Not all source terms in $I_1^\perp$ and $M_1^\perp$
require velocity interpolation. For the leading-wave factors,
we use the corresponding estimates for each component
at the required velocity weights.
For the macroscopic factors, the Gaussian decay of the  \(E_l\) and \(E_m\) give
\[
|\Gamma(g(t,x)E_l,h(t,x)E_m)|_{L_{\xi,b}^\infty}\lesssim_{b}\vert g(t,x)\vert \vert h(t,x)\vert \ \ \ \text{for any  } 1\leq l, m \leq 3\,, \ b>0\,,
\]
and for the initial singular contribution, \eqref{est-G-t-ast} gives
\[
\begin{aligned}
\bigl|
\varepsilon\mathbb P_1\mathbb G_{\ast}^{t}f_0(x)
\bigr|_{L_{\xi,\beta+\sigma+\gamma}^{\infty}}
&\lesssim
\varepsilon e^{-ct}
\langle x\rangle^{-\frac{N_0}{1-\gamma}}
\|f_0\|_
 {L_x^\infty L_{\xi,\beta+\sigma+\gamma+N_0}^{\infty}}
\lesssim
\varepsilon
(1+t+|x|)^{-\frac{N_0}{1-\gamma}}
\|f_0\|_{L_x^\infty L_{\xi,\hat\beta}^{\infty}}\,.
\end{aligned}
\]
Because interpolation need the higher velocity weights estimates, we need to establish the estimates with higher velocity weights for $R_{A}$ and $R_{B}$. Write
\[
\begin{aligned}
R_A(t)&=\int_0^t\mathbb {G}_*^{t-\tau}h_A^{(1)}(\tau)\,d\tau+\int_0^t\mathbb G_*^{t-\tau}h_A^{(0)}(\tau)\,d\tau
\\
&=R_{A,1}+R_{A,2}\,,\\
R_B(t)&=\left[\int_0^t\mathbb G_\sharp^{t-\tau}
       (h_A^{(1)}+h_B^{(1)})(\tau)\,d\tau
   +\int_0^t\mathbb G_*^{t-\tau}h_B^{(1)}(\tau)\,d\tau\right]
   \\
   &\quad+\left[\int_0^t\mathbb G_\sharp^{t-\tau}
       (h_A^{(0)}+h_B^{(0)})(\tau)\,d\tau
   +\int_0^t\mathbb G_*^{t-\tau}h_B^{(0)}(\tau)\,d\tau\right]
   \\
   &=R_{B,1}+R_{B,2}
\end{aligned}
\]
By Proposition~\ref{PF-estimate}, the estimates for $\mathbb{G}_{\ast}^{t}$ also apply to any fixed Picard term $h^{(j)}$.
Lemma~1.2 in \cite{[Haitao]}, Lemma~\ref{basic-Gamma}, Lemma~\ref{Prop-linear-fluid}, Proposition~\ref{PF-estimate}, and the leading-wave estimates then give
\begin{align*}
 \vert R_{A,2}\vert_{L^{\infty}_{x}L^{\infty}_{\xi,\hat{\beta}}}&\lesssim\sup_{\xi}\int_{0}^{t}e^{-c_{0}\nu(\xi)(t-\tau)}\nu(\xi)\left\vert \nu^{-1}h_{A}^{(0)}\right\vert_{L^{\infty}_{x}L^{\infty}_{\xi,\hat{\beta}}}(\tau)d\tau
 \\
 &\lesssim\varepsilon^{2} \sup_{\xi}\int_{0}^{t}e^{-c_{0}\nu(\xi)(t-\tau)}\nu(\xi)(1+\tau)^{-\frac{N_{0}}{1-\gamma}-\frac{1}{2}}d\tau
 \\
 &\lesssim \varepsilon^{2} (1+t)^{-\frac{N_{0}}{1-\gamma}-\frac{1}{2}}\,,
\end{align*}
and
\begin{align*}
  &\vert \partial_{x}^{s}R_{B,2}\vert_{L^{\infty}_{x}L^{\infty}_{\xi,\hat{\beta}}}
  \\
  &\lesssim
  \int_{0}^{t}(1+t-\tau)^{-\frac{1+s}{2}}\left\vert\nu^{-1}(h_{A}^{(0)}+h_{B}^{(0)})\right\vert_{L^{\infty}_{x}L^{\infty}_{\xi,\hat{\beta}}}(\tau)d\tau
  +\sup_{\xi}\int_{0}^{t}e^{-c_{0}\nu(\xi)(t-\tau)}\nu(\xi)\left\vert \nu^{-1}\partial_{x}^{s}h_{B}^{(0)}\right\vert_{L^{\infty}_{x}L^{\infty}_{\xi,\hat{\beta}}}(\tau)d\tau
  \\
  &\lesssim \varepsilon^{2}\left[\int_{0}^{t}(1+t-\tau)^{-\frac{1+s}{2}}(1+\tau)^{-\frac{3}{2}}d\tau+\sup_{\xi}\int_{0}^{t}e^{-c_{0}\nu(\xi)(t-\tau)}\nu(\xi)(1+\tau)^{-\frac{3+s}{2}}d\tau\right]
  \\
  &\lesssim\varepsilon^{2}(1+t)^{-\frac{1+s}{2}}\,.
\end{align*}
We impose the following large-time ansatz on $R_{A}$ and $R_{B}$:

\begin{equation}
\label{asatz-large-time-R}
\|R_A(t)\|_{L_x^\infty L_{\xi,\hat\beta}^\infty}\leq \varepsilon^{2}M(1+t)^{-\frac{N_{0}}{1-\gamma}}\,, \ \ \ \|\partial_{x}^{s}R_B(t)\|_{L_x^\infty L_{\xi,\hat\beta}^\infty}\leq \varepsilon^{2}M(1+t)^{-\frac{1-\delta_{0}+s}{2}}\,,
\end{equation}
 where $M$ is the same constant as in the ansatz \eqref{ansatz-R-A}--\eqref{ansatz-der-R-B-perp} and will be chosen later.
We then have
\begin{align*}
 \vert R_{A,1}\vert_{L^{\infty}_{x}L^{\infty}_{\xi,\hat{\beta}}}&\lesssim\varepsilon^{3}M(\varepsilon M+1) \sup_{\xi}\int_{0}^{t}e^{-c_{0}\nu(\xi)(t-\tau)}\nu(\xi)(1+\tau)^{-\frac{N_{0}}{1-\gamma}-\frac{1-\delta_{0}}{2}}d\tau
 \\
 &\lesssim \varepsilon^{3}M(\varepsilon M+1) (1+t)^{-\frac{N_{0}}{1-\gamma}}\,,
\end{align*}
and
\begin{align*}
  &\vert \partial_{x}^{s}R_{B,1}\vert_{L^{\infty}_{x}L^{\infty}_{\xi,\hat{\beta}}}
  \\
   &\lesssim
  \int_{0}^{t}(1+t-\tau)^{-\frac{1+s}{2}}\left\vert\nu^{-1}(h_{A}^{(1)})\right\vert_{L^{\infty}_{x}L^{\infty}_{\xi,\hat{\beta}}}(\tau)d\tau+\int_{0}^{\frac{t}{2}}(1+t-\tau)^{-\frac{1+s}{2}}\left\vert\nu^{-1}(h_{B}^{(1)})\right\vert_{L^{\infty}_{x}L^{\infty}_{\xi,\hat{\beta}}}(\tau)d\tau
  \\
  &\quad+\int_{\frac{t}{2}}^{t}(1+t-\tau)^{-\frac{1}{2}}\left\vert\nu^{-1}\partial_{x}^{s}(h_{B}^{(1)})\right\vert_{L^{\infty}_{x}L^{\infty}_{\xi,\hat{\beta}}}(\tau)d\tau
  +\sup_{\xi}\int_{0}^{t}e^{-c_{0}\nu(\xi)(t-\tau)}\nu(\xi)\left\vert \nu^{-1}\partial_{x}^{s}h_{B}^{(1)}\right\vert_{L^{\infty}_{x}L^{\infty}_{\xi,\hat{\beta}}}(\tau)d\tau
  \\
  &\lesssim \varepsilon^{3}M(\varepsilon M+1) \left[\int_{0}^{t}(1+t-\tau)^{-\frac{1+s}{2}}(1+\tau)^{-\frac{N_{0}}{1-\gamma}-\frac{1-\delta_{0}}{2}}d\tau+\int_{0}^{\frac{t}{2}}(1+t-\tau)^{-\frac{1+s}{2}}(1+\tau)^{-1+\frac{\delta_{0}}{2}}d\tau\right.
  \\
  &\left.\quad+\int_{\frac{t}{2}}^{t}(1+t-\tau)^{-\frac{1}{2}}(1+\tau)^{-1+\frac{\delta_{0}}{2}-\frac{s}{2}}d\tau
  +\sup_{\xi}\int_{0}^{t}e^{-c_{0}\nu(\xi)(t-\tau)}\nu(\xi)(1+\tau)^{-1+\frac{\delta_{0}}{2}-\frac{s}{2}}d\tau\right]
  \\
  &\lesssim\varepsilon^{3}M(\varepsilon M+1) (1+t)^{-\frac{1+s-\delta_{0}}{2}}\,.
\end{align*}
These estimates show that $R_{A}$ and $R_{B}$ satisfy the ansatz \eqref{asatz-large-time-R}.
These time-decay rates will be used in the interpolation argument below.

Combining this inequality
\begin{equation}
(1+t+\vert x-\lambda_{i}t\vert^{2})^{-a}\leq (2+\vert\lambda_{i}\vert)^{a}(1+t+\vert x\vert)^{-a}\,, \ \ \ \text{for }a>0\,,
\end{equation}
with the ansatz for $R_{A}$ and $R_{B}$ in \eqref{ansatz-R-A}--\eqref{ansatz-der-R-B-perp},
their large-time bounds \eqref{asatz-large-time-R}, and the pointwise estimates for the leading waves, applying the velocity interpolation \eqref{eq-interpolation high velocity} gives
\begin{align*}
\left\vert h_{A}^{(1)}\right\vert_{L^{\infty}_{\xi,\beta+N_{0}}}&\lesssim\left(\Vert R_{A}\Vert_{L^{\infty}_{x}L^{\infty}_{\hat{\beta}}}\right)^{2^{-(k-1)}}\left(\vert R_{A}\vert_{L^{\infty}_{\xi,\beta}}\right)^{2-2^{-(k-1)}}
\\
&\quad+\left(\Vert R_{A}\Vert_{L^{\infty}_{x}L^{\infty}_{\hat{\beta}}}\Vert R_{B}\Vert_{L^{\infty}_{x}L^{\infty}_{\hat{\beta}}}\right)^{2^{-k}}\left(\vert R_{A}\vert_{L^{\infty}_{\xi,\beta}}\vert R_{B}\vert_{L^{\infty}_{\xi,\beta}}\right)^{1-2^{-k}}
\\
&\quad+\left(\Vert R_{A}\Vert_{L^{\infty}_{x}L^{\infty}_{\hat{\beta}}}\right)^{2^{-k}}\left(\vert
R_{A}\vert_{L^{\infty}_{\xi,\beta}}\right)^{1-2^{-k}}\left\vert \Phi+V^{\perp}\right\vert_{L^{\infty}_{\xi,\beta+N_{0}+\gamma}}
\\
&\quad+\left(\Vert R_{B}\Vert_{L^{\infty}_{x}L^{\infty}_{\hat{\beta}}}\right)^{2^{-k}}\left(\vert
R_{B}\vert_{L^{\infty}_{\xi,\beta}}\right)^{1-2^{-k}}\left\vert \Phi_{A}+V_{A}^{\perp}\right\vert_{L^{\infty}_{\xi,\beta+N_{0}+\gamma}}
\\
&\lesssim\left[(1+t+\vert x\vert)^{-\frac{N_{0}}{1-\gamma}-\frac{1}{2}}\right]^{1-2^{-k}}
\\[2mm]
&\lesssim (1+t+\vert x\vert )^{-\frac{N_{0}}{1-\gamma}}\,,
\end{align*}
provided that $1-\frac{\delta_{0}}{2}\geq \frac{1}{2}$ and $2^{-k}\leq \frac{1-\gamma}{1-\gamma+2N_{0}}$\,.
Consequently, $I_{1}^{\perp}$ satisfies
\begin{align*}
|I_1^\perp(x,t)|_{L_{\xi,\beta}^\infty}&\lesssim\int_{0}^{t}\int_{\mathbb{R}}e^{-c(t-\tau)}\langle x-y\rangle^{-\frac{N_{0}}{1-\gamma}}F_{L^{\infty}_{\xi,\beta+N_{0}}}^{\#}[h_{A}^{(1)}(\tau)](y)dyd\tau
\\
&\lesssim \varepsilon^{3}M(1+\varepsilon M)
\int_0^t\int_{\mathbb{R}}e^{-c(t-\tau)}
 \langle x-y\rangle^{-\frac{N_0}{1-\gamma}}
 (1+\tau+|y|)^{-\frac{N_0}{1-\gamma}}
\,dy\,d\tau
\\
&\lesssim
 \varepsilon^{3}M(1+\varepsilon M)
(1+t+|x|)^{-\frac{N_0}{1-\gamma}}\,,
\end{align*}
and hence, it satisfies the ansatz \eqref{ansatz-R-A} for $R_{A}^{\perp}$. For $M_1^\perp$, we estimate the source term directly:
\begin{align*}
\left\vert \partial_{x}^{s}h_{B}^{(1)}\right\vert_{L^{\infty}_{\xi,\beta+N_{0}}}&\lesssim (\varepsilon^{4}M^{2})^{2^{-k}}\left(\vert \partial_{x}^{s}R_{B}\vert_{L^{\infty}_{\xi,\beta}}\vert R_{B}\vert_{L^{\infty}_{\xi,\beta}}\right)^{1-2^{-k}}
+\vert\partial_{x}^{s}
\mathbb{P}_{0}R_{B}\vert_{L^{\infty}_{\xi,\hat{\beta}}}\left\vert \Phi_{B}+V^{\perp}_{B}\right\vert_{L^{\infty}_{\xi,\hat{\beta}}}
\\
&\quad+\vert
\mathbb{P}_{0}R_{B}\vert_{L^{\infty}_{\xi,\hat{\beta}}}\left\vert \partial_{x}^{s}(\Phi_{B}+V^{\perp}_{B})\right\vert_{L^{\infty}_{\xi,\hat{\beta}}}
\\
&\quad+(\varepsilon^{2}M)^{2^{-k}}\left(\vert\partial_{x}^{s}\mathbb{P}_{1}
R_{B}\vert_{L^{\infty}_{\xi,\beta}}\right)^{1-2^{-k}}\left\vert \Phi_{B}+V^{\perp}_{B}\right\vert_{L^{\infty}_{\xi,\hat{\beta}}}
\\
&\quad+(\varepsilon^{2}M)^{2^{-k}}\left(\vert\mathbb{P}_{1}
R_{B}\vert_{L^{\infty}_{\xi,\beta}}\right)^{1-2^{-k}}\left\vert\partial_{x}^{s}(\Phi_{B}+V^{\perp}_{B})\right\vert_{L^{\infty}_{\xi,\hat{\beta}}}
\\
&\lesssim M\varepsilon^{3}(1+\varepsilon M)\left\{\left[(1+t)^{-\frac{s}{2}}\Bigg(\sum_{m=1}^{3}\theta_{m}^{4-2\delta_{0}}+(1+t)^{\frac{-1+\delta_{0}}{2}}\psi_{m}^{3}\Bigg)\right]^{1-2^{-k}}\right.
\\
&\quad+(1+t)^{-\frac{s}{2}}\Bigg(\sum_{m=1}^{3}\theta_{m}^{3-\delta_{0}}+(1+t)^{\frac{-1+\delta_{0}}{4}}\psi_{m}^{3}\Bigg)
\\
&\quad\left.+\Bigg(\sum_{m=1}^{3}\theta_{m}^{3-\delta_{0}}+(1+t)^{\frac{-3+\delta_{0}}{4}}\psi_{m}^{\frac{3}{2}}\Bigg)^{1-2^{-k}}\Bigg(\sum_{m=1}^{3}\theta_{m}^{1}+\psi_{m}^{\frac{3}{2}}\Bigg)\right\}
\\
&\lesssim  M\varepsilon^{3}(1+\varepsilon M)\Bigg(\sum_{m=1}^{3}\theta_{m}^{3-\delta_{0}+s(1-\delta_{0})}+(1+t)^{\frac{-3+\delta_{0}}{4}}\psi_{m}^{\frac{3+s}{2}}\Bigg)\,.
\end{align*}
It follows that
\[
\begin{aligned}
|\partial_x^sM_1^\perp(x,t)|_{L_{\xi,\beta}^\infty}
&=
\left|
\int_0^t\mathbb P_1
(\mathbb G_*^{t-\tau}+\mathcal G^{t-\tau})
\partial_x^sh_B^{(1)}(\tau)\,d\tau
\right|_{L_{\xi,\beta}^\infty}
\\
&\lesssim
\varepsilon^{3}M(1+\varepsilon M)\int_0^t\int_{\mathbb R}e^{-c(t-\tau)}
\langle x-y\rangle^{-\frac{N_0}{1-\gamma}}
\\
&\quad\cdot\Bigg(\sum_{m=1}^{3}\theta_{m}^{3-\delta_{0}+s(1-\delta_{0})}+(1+t)^{\frac{-3+\delta_{0}}{4}}\psi_{m}^{\frac{3+s}{2}}\Bigg)(\tau,y)\,dyd\tau
\\
&\lesssim  \varepsilon^{3}M(1+\varepsilon M)\Bigg(\sum_{m=1}^{3}\theta_{m}^{3-\delta_{0}+s(1-\delta_{0})}+(1+t)^{\frac{-3+\delta_{0}}{4}}\psi_{m}^{\frac{3+s}{2}}\Bigg)\,,
\end{aligned}
\]
provided that $0<\delta_{0}<\frac{1}{4}$ and $2^{-k}\leq\frac{1-\delta_{0}}{6-2\delta_{0}}$.
Thus $M_{1}^{\perp}$ satisfies the ansatz \eqref{ansatz-R-B-perp}--\eqref{ansatz-der-R-B-perp}.

Combining the component estimates, the known forcing has amplitude
$C_{0}\varepsilon^2$, while the terms containing the remainder have amplitude at most
$C_1\varepsilon^3M(1+\varepsilon M)$, with the corresponding profiles in \eqref{ansatz-R-A}-\eqref{ansatz-der-R-B-perp}.
Here $C_0\geq1$ and $C_1$ depend only on the
initial-data norm, not on $M$, $\varepsilon$, or the terminal time.
The output coefficient is therefore bounded by
\[
C_0\varepsilon^{2}+C_1\varepsilon^{3} M(1+\varepsilon M)=\varepsilon^{2}\left(C_0+C_1\varepsilon M(1+\varepsilon M)\right).
\]
Choose $M\geq 4 C_{0}$ sufficiently large and $0<\varepsilon\leq1$ sufficiently small that
\[
C_0+C_1\varepsilon M(1+\varepsilon M)\leq\frac12 M.
\]
The output is then bounded by $\frac{1}{2}M\varepsilon^{2}$, completing the proof.

\end{proof}

\section{Fluid approximation to the Navier-Stokes equation}\label{fluid-approximation}

\subsection{Macro-Micro decomposition and formal derivation}

In this subsection, we use the macro–micro decomposition to derive the fluid approximation formally. The corresponding linear approximation was established in \cite{[LiuYu2]}; here we extend the argument to the nonlinear setting. For simplicity of notation, we write
\[
 f^0=\mathbb P_0f,\qquad f^1=\mathbb P_1f,\qquad f=f^0+f^1.
\]
Since $\mathbb P_0L=L\mathbb P_0=0$ and $\mathbb P_0\Gamma=0$, the exact projected equations are
\[
 \left\{\begin{aligned}
 (f^0)_t+\partial_x\mathbb P_0\xi^1(f^0+f^1)&=0,\\
 (f^1)_t+\partial_x\mathbb P_1\xi^1(f^0+f^1)
 &=Lf^1+\Gamma(f^0+f^1,f^0+f^1).
 \end{aligned}\right.
\]
In the spirit of the Chapman-Enskog expansion, we assume that the macro part $f^{0}$ dominates the micro part $f^{1}$ and that the derivatives $\nabla_{x,t}f$ is dominated by the quantity $f$ itself. Accordingly, we simplify the micro part of the equation as
\begin{align*}
\pa_{x}\mathbb{P}_{1}(\xi^{1} f^{0})=Lf^{1}+\Gamma(f^{0},f^{0})
\end{align*}
or
\begin{align}\label{ce}
f^{1}=L^{-1}\big[\pa_{x}\mathbb{P}_{1}(\xi^{1} f^{0})-\Gamma(f^{0},f^{0})\big]\,.
\end{align}
For the nonlinear term in \eqref{ce},
\eqref{eq:fixed-microscopic-vectors}, coercivity, the kernel bounds,
and Lemma~\ref{basic-K} give, for every fixed $b\geq0$,
\[
 |L^{-1}\Gamma(E_l,E_m)|_{L^\infty_{\xi,b}}<\infty.
\]
The formula (\ref{ce}) is the so-called Chapman-Enskog expansion. Substituting this relation into the macro equation gives
\begin{align}\label{NS-K}
(f^{0})_{t}+\pa_{x}\mathbb{P}_{0}(\xi^{1} f^{0})=\pa_{x}\mathbb{P}_{0}\xi^{1}(-L)^{-1}\big[\pa_{x}\mathbb{P}_{1}(\xi^{1}f^{0})-\Gamma(f^{0},f^{0})\big]\,.
\end{align}
This is the kinetic form of the quadratic Navier--Stokes equation with frozen physical viscosity.

To obtain the gas-dynamic form, we express
\begin{align*}
f^{0}(t,x,\xi)=\sum_{i=1}^{3}\omega_{i}(t,x)E_{i}(\xi)\,, \quad \omega_{i}(t,x)=\big<f, E_{i}\big>_{\xi}\,,
\end{align*}
in terms of the Euler characteristic directions $E_{j}$, $j=1,2,3$ in (\ref{Eigen function}).

The gas-dynamic form is obtained by projecting (\ref{NS-K}) onto the Euler characteristic directions $E_{j}$, $j=1,2,3$. For the initial condition, note that the fluid part of the Boltzmann equation is generated by the long-wave component of the solution. We therefore define
$$
\omega_{i0}=\varepsilon\left\langle \check{\chi}*_{x}f_{0}, E_{i}\right\rangle_{\xi}\,.
$$
We have the following proposition.
\begin{proposition}\label{LNS}
Equation \eqref{NS-K}, in the characteristic coordinates of
\eqref{Eigen function}, is
\begin{equation}\label{NS-burger}
 \left\{\begin{aligned}
 \partial_t\omega_i+\lambda_i\partial_x\omega_i
 &=\sum_{j=1}^3A_{ij}\partial_x^2\omega_j
       +\sum_{k,l=1}^3B_i^{kl}\partial_x(\omega_k\omega_l),\\
 A_{ij}&=\langle\mathbb P_1\xi^1E_i,
                       (-L)^{-1}\mathbb P_1\xi^1E_j\rangle_\xi,\\
 B_i^{kl}&=\langle L^{-1}\mathbb P_1\xi^1E_i,
                                     \Gamma(E_k,E_l)\rangle_\xi,\\
 \omega_i(0)&=\varepsilon\langle\check\chi*_xf_0,E_i\rangle_\xi.
 \end{aligned}\right.
\end{equation}
With the normalizations
\[
 \mu=\langle\xi^1\xi^2\sqrt{\mathcal M},
                 (-L)^{-1}(\xi^1\xi^2\sqrt{\mathcal M})\rangle_\xi,
 \qquad
 \kappa=\frac16\langle\xi^1(|\xi|^2-5)\sqrt{\mathcal M},
                 (-L)^{-1}\xi^1(|\xi|^2-5)\sqrt{\mathcal M}\rangle_\xi,
\]
the viscosity matrix is
\[
 (A_{ij})=
 \begin{pmatrix}
 \frac23\mu+\frac15\kappa&\frac{\sqrt3}{5}\kappa&-\frac23\mu+\frac15\kappa\\
 \frac{\sqrt3}{5}\kappa&\frac35\kappa&\frac{\sqrt3}{5}\kappa\\
 -\frac23\mu+\frac15\kappa&\frac{\sqrt3}{5}\kappa&\frac23\mu+\frac15\kappa
 \end{pmatrix}.
\]
\end{proposition}

\subsection{Rigorous estimates}

We first compare the leading kinetic Burgers waves with the artificial-viscosity system
\begin{equation}\label{artificial}
 \left\{\begin{aligned}
 \partial_tu_i+\lambda_i\partial_xu_i
 &=A_i\partial_x^2u_i+\sum_{k,l=1}^3B_i^{kl}\partial_x(u_ku_l),\\
 A_i&=A_{ii}>0,\qquad
 u_i(0)=\varepsilon\langle\check\chi*_xf_0,E_i\rangle_\xi.
 \end{aligned}\right.
\end{equation}
Here $A_i$ is the same diffusion coefficient as that appearing in the expansion of
$\sigma_i$. We define $\tilde u_i$ and $v_i$ by
\begin{equation}\label{burger-artificial}
 \partial_t\tilde u_i+\lambda_i\partial_x\tilde u_i
 =A_i\partial_x^2\tilde u_i+B_i^{ii}\partial_x(\tilde u_i^2),
 \qquad \tilde u_i(0)=u_i(0),
\end{equation}
and set $v_i=u_i-\tilde u_i$. Subtracting the two equations gives
\begin{align}\label{off-artificial}
 \partial_tv_i+\lambda_i\partial_xv_i
 ={}&A_i\partial_x^2v_i
     +\sum_{j\ne i}B_i^{jj}\partial_x(\tilde u_j^2)
     +\sum_{k\ne l}B_i^{kl}\partial_x(\tilde u_k\tilde u_l)\notag\\
 &+2\sum_{k,l}B_i^{kl}\partial_x(v_k\tilde u_l)
     +\sum_{k,l}B_i^{kl}\partial_x(v_kv_l),
 \qquad v_i(0)=0.
\end{align}
For $t>0$, we use the convected heat kernel
\[
 \mathbb H_i^t(x)=\frac1{\sqrt{4\pi A_it}}
             \exp\left(-\frac{|x-\lambda_it|^2}{4A_it}\right), \qquad \widehat{\mathbb{H}}_{i}^{t}(\eta)=e^{\alpha_{i}(\eta)t}\,,
 \qquad \alpha_i(\eta)=-i\lambda_i\eta-A_i\eta^2.
\]
At $t=0$, $\mathbb H_i^0$ is the identity operator. Thus the exact mild equations are
\begin{align}\label{def-tilde-u-i}
 \tilde u_i(t)
 ={}&\mathbb H_i^t*_xu_i(0)
       +B_i^{ii}\int_0^t\partial_x\mathbb H_i^{t-\tau}*_x\tilde u_i^2(\tau)\,d\tau,
\end{align}
and
\begin{align}\label{def-v-i}
 v_i(t)={}&\sum_{\substack{1\leq l,m\leq 3\\(l,m)\ne(i,i)}}B_i^{lm}
      \int_0^t\partial_x\mathbb H_i^{t-\tau}*_x(\tilde u_l\tilde u_m)(\tau)\,d\tau\notag\\
 &+2\sum_{1\leq l,m\leq 3}B_i^{lm}
      \int_0^t\partial_x\mathbb H_i^{t-\tau}*_x(v_l\tilde u_m)(\tau)\,d\tau\notag\\
 &+\sum_{1\leq l,m\leq 3}B_i^{lm}
      \int_0^t\partial_x\mathbb H_i^{t-\tau}*_x(v_lv_m)(\tau)\,d\tau.
\end{align}

The constant viscosity matrix in \eqref{artificial} is positive definite and commutes with the diagonal transport matrix. Hence, the results of \cite[Theorem~2.9 and Section~4]{[Zeng]} apply to this system, with the identity as the equilibrium symmetrizer. The required assumptions on the initial data follow from the estimates below.
\\
For every fixed
$m,n\geq0$,
\begin{equation}\label{sec42:data-decay}
 \|u(0)\|_{H^m}+\sup_x\langle x\rangle^n
       \big(|u(x,0)|+|\partial_x^m u(x,0)|\big)
       \lesssim_{m,n}\varepsilon.
\end{equation}
Moreover,
\begin{equation}\label{sec42:data-primitives}
 \left\|\int_{-\infty}^x u(y,0)\,dy\right\|_{L^1(-\infty,0)}
 +\left\|\int_x^\infty u(y,0)\,dy\right\|_{L^1(0,\infty)}
 \leq\int_{\mathbb R}|y|\,|u(y,0)|\,dy\lesssim\varepsilon.
\end{equation}
The pointwise primitive bounds follow from \eqref{sec42:data-decay}. We shall also use the source estimates in \cite[Lemma~3.4, (4.25)--(4.35), and (8.43)--(8.47)]{[Zeng]}. The initial condition in \eqref{burger-artificial} give $v_i(0)=0$; the derivative estimate below follows from the differentiated source argument, rather than solely from the $L_x^\infty$ conclusion of \cite[Theorem~2.9]{[Zeng]}.

\begin{proposition}\label{sec4:artificial-waves}
For each fixed finite integer $s\geq0$ and sufficiently small
$\varepsilon$, the Burgers waves satisfy
\begin{equation}
\label{est-tilde-ui}
 |\partial_x^s\tilde u_i(t,x)|\lesssim_s\varepsilon\theta_i^{1+s}.
\end{equation}
For the off-diagonal waves with $v_i(0)=0$, one has
\begin{equation}\label{sec5-artificial-bounds}
 |\partial_x^s v_i(t,x)|\lesssim\varepsilon^2(1+t)^{-s/2}
       \left(\psi_i^{3/2}+\sum_{j\ne i}\tilde\psi_j^{3/2}\right),
       \qquad s=0,1.
\end{equation}
\end{proposition}
The proof is similar to, and in fact simpler than, that for the leading fluid wave of the 1D Boltzmann equation. We therefore omit the details.

Set
\begin{equation}\label{def-mathcal-K}
 \mathcal K_i=\tilde\phi_i-\tilde u_i
\end{equation}
and
\begin{equation}\label{def-mathcal-S}
 \mathcal S_i=V_{i,B_{1}}-v_i.
\end{equation}

\begin{lemma}
\label{lemm-mathcal-K}
Under the hypotheses of Section~\ref{wave}, for any fixed $0<\delta_0<1/4$ and finite integer $s\geq0$, sufficiently small $\varepsilon$ gives
\begin{equation}\label{est-mathcal-K-i}
 |\partial_x^s\mathcal K_i(t,x)|
                  \lesssim\varepsilon\theta_i^{2+s-\delta_0}.
\end{equation}
\end{lemma}
\begin{proof}
Subtracting \eqref{def-tilde-u-i} from
\eqref{expression tilde-phi-i}, we obtain
\begin{align}\label{eq-mathcal-K}
 \mathcal K_i={}&\varepsilon\langle\mathbb G_i^tf_0,E_i\rangle_\xi
                         -\mathbb H_i^t*_xu_i(0)\notag\\
 &+B_i^{ii}\int_0^t\partial_x\mathbb H_i^{t-\tau}*_x
              [\mathcal K_i(\tilde\phi_i+\tilde u_i)](\tau)d\tau\notag\\
 &+\frac1{2\pi}\int_0^t\int_{\mathbb R}e^{ix\eta}\chi_\delta(\eta)
       \Big[e^{\sigma_i(\eta)(t-\tau)}
            \langle e_i(\eta),E_i\rangle_\xi
            \langle\Gamma(E_i,E_i),e_i(-\eta)\rangle_\xi\notag\\
 &\hspace{42mm}-i\eta B_i^{ii}e^{\alpha_i(\eta)(t-\tau)}\Big]
              \widehat{\tilde\phi_i^2}(\tau,\eta)d\eta d\tau\notag\\
 &-\frac{B_i^{ii}}{2\pi}\int_0^t\int_{\mathbb R}
       e^{ix\eta+\alpha_i(\eta)(t-\tau)}i\eta(1-\chi_\delta(\eta))
              \widehat{\tilde\phi_i^2}(\tau,\eta)d\eta d\tau\notag\\
 ={}&J_1+J_2+J_3+J_4.
\end{align}
Fix a finite integer \(s\in\mathbb N\cup\{0\}\). Throughout the proof, \(q\) denotes the order of differentiation, with \(0\le q\le s\). For $J_{1}$, the expansions of $\widehat{\mathbb{G}}_{i}^{t}$ and the heat kernel $\widehat{\mathbb{H}_{i}^{t}}$ give
\begin{align*}
J_{1}&=\frac{\varepsilon}{2\pi}\left[\int_{\mathbb{R}}e^{ix\eta}\chi_{\delta}(\eta)e^{\sigma_{i}(\eta)t}\left(\langle E_{i},e_{i}(-\eta)\rangle_{\xi}\langle \hat{f}_{0}(\eta),e_{i}(-\eta)\rangle_{\xi}-\langle \hat{f}_{0}(\eta),E_{i}\rangle_{\xi}\right)d\eta\right.
\\
&\quad\left.+\int_{\mathbb{R}}e^{ix\eta}\chi_{\delta}(\eta)\left(e^{\sigma_{i}(\eta)t}-e^{\alpha_{i}(\eta)t}\right)\langle \hat{f}_{0}(\eta),E_{i}\rangle_{\xi}d\eta
+\int_{\mathbb{R}}e^{ix\eta}\left(\chi_{\delta}(\eta)-\chi(\eta)\right)e^{\alpha_{i}(\eta)t}\langle \hat{f}_{0}(\eta),E_{i}\rangle_{\xi}d\eta\right]
\\
&=J_{11}+J_{12}+J_{13}
\end{align*}
For $J_{11}$,
\begin{align*}
&\langle E_{i},e_{i}(-\eta)\rangle_{\xi}\langle \hat{f}_{0}(\eta),e_{i}(-\eta)\rangle_{\xi}-\langle \hat{f}_{0}(\eta),E_{i}\rangle_{\xi}
\\
&=\left(\langle E_{i},e_{i}(-\eta)\rangle_{\xi}-1\right)\langle \hat{f}_{0}(\eta),e_{i}(-\eta)\rangle_{\xi}+\langle\hat{f}_{0}(\eta),e_{i}(-\eta)-E{i}\rangle_{\xi}
\\
&=\langle E_{i},e_{i}(-\eta)-E_{i}\rangle_{\xi}\langle \hat{f}_{0}(\eta),e_{i}(-\eta)\rangle_{\xi}+\langle\hat{f}_{0}(\eta),e_{i}(-\eta)-E{i}\rangle_{\xi}
\\
&=(i\eta)\left[\frac{1}{i\eta}\left(\langle E_{i},e_{i}(-\eta)-E_{i}\rangle_{\xi}\langle \hat{f}_{0}(\eta),e_{i}(-\eta)\rangle_{\xi}+\langle\hat{f}_{0}(\eta),e_{i}(-\eta)-E_{i}\rangle_{\xi}\right)\right]\,,
\end{align*}
where the factor in brackets is a smooth function. Hence,
\[
\vert \partial_{x}^{q}J_{11}\vert \lesssim\varepsilon\theta_{i}^{2+q}\,.
\]
For $J_{12}$, we write $\chi_{\delta}(\eta)\left( e^{\sigma_i t}-e^{\alpha_i t}\right)$ as
\begin{equation}
\label{eq-artificial-boltzmann}
\chi_{\delta}(\eta)t(\sigma_i-\alpha_i)
       \int_0^1e^{t[(1-r)\alpha_i+r\sigma_i]}dr =(i\eta)\left(\chi_{\delta}(\eta)\frac{t \left(\sigma_{i}-\alpha_{i}\right)e^{\alpha_{i}(\eta)t}\int_{0}^{1}e^{rt(\sigma_{i}-\alpha_{i})}dr}{i\eta}\right)\equiv (i\eta)\widehat{\mathcal{D}}_{i}(\eta,t)\,,
\end{equation}
Using the expansion $\sigma_i-\alpha_i=\eta^{3}R_{\sigma,i}(\eta)$ and the identity
\begin{equation}
\label{derivative-change-time-space}
\left(1+\frac{\vert x-\lambda_{i}t\vert^{2}}{t}\right)^{-1}\left(1-\partial_{\zeta}^{2}\right)\left(e^{i\left(\frac{x-\lambda_{i}t}{\sqrt{t}}\zeta\right)}\right)=e^{i\left(\frac{x-\lambda_{i}t}{\sqrt{t}}\zeta\right)}\,,
\end{equation}
for $t\geq 1$ and $n\in\mathbb{N}\cup\left\{0\right\}$, we change variables by $\eta=\frac{\zeta} {\sqrt{t}}$ and integrate by parts to obtain
\begin{align*}
\partial_{x}^{n}\mathcal{D}_{i}(t,x)&=\frac{1}{2\pi}(t)^{-\frac{1+n}{2}}\left(1+\frac{\vert x-\lambda_{i}t\vert^{2}}{t}\right)^{-L}\int_{\mathbb{R}}\left[e^{i\left(\frac{x-\lambda_{i}t}{\sqrt{t}}\zeta\right)}\right]
 \\
 &\quad\cdot\left(1-\partial_{\zeta}^{2}\right)^{L} \left[(i\zeta)^{n}e^{-A_{i}\zeta^{2}}\left(\frac{t^{3/2}(\sigma_{i}-\alpha_{i})(\zeta/\sqrt{t})}{i\zeta}\right)
 \left(\int_{0}^{1}e^{rt(\sigma_{i}-\alpha_{i})(\zeta/\sqrt{t})}dr\right)\chi_{\delta}\left(\frac{\zeta}{\sqrt{t}}\right)\right]d\zeta\,,
\end{align*}
for any fixed $L\geq N$. All remaining multiplicative factors can be controlled by the Gaussian factor $e^{-c\zeta^{2}}$ for some $0<c<A_{i}$. Hence,
\begin{equation}
\label{est-mathcal-Di}
\left\vert \partial_{x}^{n}\mathcal{D}_{i}\right\vert\lesssim\theta_{i}^{1+n}(t,x)\,, \qquad n\in \mathbb{N}\cup\left\{0\right\}\,,
\end{equation}
Moreover, it satisfies \eqref{asumption-cancellation-mechanism-1} with characteristic speed $\lambda_{i}$. Therefore,
\begin{align*}
 \left\vert\partial_{x}^{q}J_{12}\right\vert&=\frac{\varepsilon}{2\pi}\left\vert\int_{\mathbb{R}}e^{ix\eta}(i\eta)^{q+1}\widehat{\mathcal{D}}_{i}(\eta,t)\langle \hat{f}_{0}(\eta),E_{i}\rangle_{\xi}d\eta\right\vert\lesssim\varepsilon\theta_{i}^{q+2}\,.
\end{align*}

For $J_{13}$, there exists a fixed $\eta_{\ast}>0$ such that
$\text{supp}\left(\chi_{\delta}(\eta)-\chi(\eta)\right)\subset\left\{\vert \eta\vert\geq \eta^{\ast}\right\}$. It follows that
\[
\vert \partial_{x}^{q}J_{13}\vert\lesssim \varepsilon e^{-ct}\left(1+\vert x-\lambda_{i}t\vert\right)^{-L}\lesssim \varepsilon\theta_{i}^{2+q}\,, \qquad \text{for any fixed }L\geq 2N.
\]
Next, we decompose $J_{3}$ as follows:
\begin{align*}
J_{3}&=\frac1{2\pi}\left[\int_0^t\int_{\mathbb R}e^{ix\eta}\chi_\delta(\eta)
       e^{\sigma_i(\eta)(t-\tau)}\left(\langle e_i(\eta),E_i\rangle_\xi
            \langle e_i(\eta),\Gamma(E_i,E_i)\rangle_\xi-i\eta B_i^{ii}\right) \widehat{\tilde\phi_i^2}(\tau,\eta)d\eta d\tau\right.\notag\\
 &\quad+\int_0^t\int_{\mathbb R}e^{ix\eta}\chi_\delta(\eta)
       \left(e^{\sigma_i(\eta)(t-\tau)}-e^{\alpha_{i}(\eta)(t-\tau)}\right)\left(i\eta B_i^{ii}\right) \widehat{\tilde\phi_i^2}(\tau,\eta)d\eta d\tau
       \\
       &=J_{31}+J_{32}
\end{align*}
From the expansion of $e_{i}(\eta)$, we have
\begin{align*}
 \langle e_i(\eta),&E_i\rangle_\xi
       \langle e_i(\eta),\Gamma(E_i,E_i)\rangle_\xi-i\eta B_i^{ii}
       \\
       &\quad=\langle \eta^{2}R_{e,i}(\eta),\Gamma(E_{i},E_{i})\rangle_{\xi}+\langle e_{i}(\eta)-E_{i},E_{i}\rangle_{\xi}\langle e_{i}(\eta),\Gamma(E_{i},E_{i})\rangle_{\xi}\,,
\end{align*}
and hence
\[
\vert \partial_{x}^{q}J_{31}\vert\lesssim \varepsilon^{2}\left[\int_{0}^{\frac{t}{2}}\theta_{i}^{3+q}(t-\tau)\ast_{x}\theta_{i}^{2}(\tau)d\tau+\int_{\frac{t}{2}}^{t}\theta_{i}^{2}(t-\tau)\ast_{x}\theta_{i}^{3+q}(\tau)d\tau\right]\lesssim \varepsilon^{2}\theta_{i}^{2+q}\,.
\]
For $J_{32}$, the estimate of $\mathcal{D}_{i}(t,x)$ in \eqref{est-mathcal-Di} gives
\begin{align*}
\vert \partial_{x}^{q}J_{32}\vert &\leq\left\vert \int_{0}^{t}\int_{\mathbb{R}}e^{ix\eta}B_{i}^{ii}(i\eta)^{q+2}\widehat{\mathcal{D}}_{i}(\eta,t-\tau)\widehat{\tilde\phi^{2}_{i}}(\tau,\eta)d\eta d\tau\right\vert
\\
&\lesssim \varepsilon^{2}\left[\int_{0}^{\frac{t}{2}}\theta_{i}^{3+q}(t-\tau)\ast_{x}\theta_{i}^{2}(\tau)d\tau+\int_{\frac{t}{2}}^{t}\theta_{i}^{2}(t-\tau)\ast_{x}\theta_{i}^{3+q}(\tau)d\tau\right]
\\
&\lesssim\varepsilon^{2}\theta_{i}^{2+q}\,,
\end{align*}
where $(i\eta)^{2+q}$ is placed on $\widehat{D_{i}}$ over $[0,\frac{t}{2}]$, while $(i\eta)^{1+q}$ is placed on $\widehat{\tilde\phi_{i}^{2}}$ over $[\frac{t}{2},t]$.
For $J_{4}$,
$\text{supp}\left(1-\chi_{\delta}(\eta)\right)\subset\left\{\vert \eta\vert\geq \frac{\delta}{2}\right\}$. Placing $(i\eta)$ and the external derivative factor $(i\eta)^{q}$ on $\widehat{\tilde\phi_{i}^{2}}$, we obtain
\[
\vert \partial_{x}^{q}J_{4}\vert\lesssim \varepsilon^2 \left[e^{-ct}\left(1+\vert x-\lambda_{i}t\vert\right)^{-L}\right]\ast_{x,t}(\theta_{i}^{3+q})\lesssim \varepsilon\theta_{i}^{3+q}\,, \qquad \text{for any fixed }L\geq 2N.
\]

In view of the estimates for $J_{1}, \ J_{3}, \ J_{4}$, we impose the following ansatz for $\mathcal{K}_{i}$:
\begin{equation}
\label{asatz-mathcal-i}
\left\vert \partial_{x}^{q}\mathcal{K}_{i}\right\vert\leq M\varepsilon\theta_{i}^{2+q-\delta_{0}}\,,
\end{equation}
where $M$ will be chosen later. By integration by parts,
\begin{align*}
\vert \partial_{x}^{q}J_{2}\vert&\lesssim \left\vert \int_{0}^{\frac{t}{2}}\partial_{x}^{q+1}\mathbb H_i^{t-\tau}*_x
[\mathcal K_i(\tilde\phi_i+\tilde u_i)](\tau)d\tau\right\vert+\left\vert\int_{\frac{t}{2}}^{t}\partial_{x}\mathbb H_i^{t-\tau}*_x\partial_{x}^{q}
[\mathcal K_i(\tilde\phi_i+\tilde u_i)](\tau)d\tau\right\vert
              \\
&\lesssim\varepsilon^{2}M_{s}\left[\int_{0}^{\frac{t}{2}}\theta_{i}^{q+2}(t-\tau)\ast_{x}\theta_{i}^{3-\delta_{0}}(\tau)d\tau+\int_{\frac{t}{2}}^{t}\theta_{i}^{2}(t-\tau)\ast_{x}\theta_{i}^{3-\delta_{0}+q}(\tau)d\tau\right]
\\
&\lesssim\varepsilon^{2}M\theta^{2-\delta_{0}+q}_{i}\,.
\end{align*}
Combining the above component estimates, the known forcing has amplitude
$C_{0}\varepsilon+C_{1}\varepsilon^{2}$, whereas the terms containing $\mathcal{K}_{i}$ have amplitude at most
$C_{2}\varepsilon^2M$, with the corresponding profiles in \eqref{asatz-mathcal-i}.
Here $C_{0}\,, \ C_{1}\geq1$ and $C_{2}$ depend only on the fixed parameters $s$ and the
initial-data norm, and are independent of $M$, $\varepsilon$, and the terminal time.
Thus the output coefficient is bounded by
\[
C_{0}\varepsilon+C_{1}\varepsilon^{2}+C_{2}\varepsilon^2M=\varepsilon\left(C_{0}+C_{1}\varepsilon+C_{2}\varepsilon M\right)
\]
Choose $M\geq 4 C_{0}$ sufficiently large and then $0<\varepsilon\leq1$ sufficiently small so that
\[
C_{1}\varepsilon+C_{2}\varepsilon M\leq\frac14 M.
\]
The output is then bounded by $\frac{1}{2}M\varepsilon$, which closes the ansatz and completes the proof.

\end{proof}

\begin{lemma}[Off-diagonal comparison]\label{sec4:refined-comparison}
Under the hypotheses of Section~\ref{wave}, for any fixed $0<\delta_0<\frac{1}{4}$ and
sufficiently small $\varepsilon$, one has
\begin{align}\label{est-mathcal-S}
 &\left|\partial_{x}^{s}\mathcal S_i(t,x)\right|\lesssim\varepsilon^2\left[
       \theta_i^{2-\delta_0+s}+(1+t)^{-\frac{1-\delta_{0}+2s}{4}}\left(\psi_i^{3/2}+\sum_{j\ne i}\tilde\psi_j^{3/2}\right)\right]\,, \quad \text{for any }s\in\left\{0,1\right\}\,.
\end{align}
\end{lemma}
\begin{proof}
By the decomposition $V_{i,B}=V_{i,B_{1}}+V_{i,B_{2}}$ established in the proof of Proposition \ref{prop-V-i}, we have
\[
V_{i,B_{1}}=\int_{0}^{t}\left\langle E_{i},\mathbb{G}_{i}^{t-\tau}\left[\Gamma\left(\Phi_{B},\Phi_{B}\right)-\Gamma(\tilde\phi_{i}E_{i},\tilde\phi_{i}E_{i})\right](\tau)\right\rangle_{\xi}d\tau+\int_{0}^{t}\left\langle E_{i},\mathbb{G}^{t-\tau}_{i} g_{A}(\tau)\right\rangle_{\xi}\,,
\]
We further decompose $V_{i,B_{2}}$ as
\[
V_{i,B_{2}}=\sum_{l=1}^{3}(V_{i,B_{2}}^{l})+r_{i}\,,
\]
where
\begin{align*}
V_{i,B_{2}}^{l}&=\varepsilon\langle E_{i},\mathbb{G}_{l}^{t}
f_{0}\rangle_{\xi}+\int_{0}^{t}\left\langle E_{i},\left(\mathbb{G}_{\ast}^{t-\tau}+\mathcal{G}^{t-\tau}\right)\Gamma\left(\tilde\phi_{l}E_{l},\tilde\phi_{l}E_{l}\right)(\tau)\right\rangle_{\xi}d\tau
\\
&\quad+\int_{0}^{t}\left\langle E_{i},\mathbb{G}_{l}^{t-\tau}\Gamma\left(\tilde\phi_{l}E_{l},\tilde\phi_{l}E_{l}\right)(\tau)\right\rangle_{\xi}d\tau\,,
\end{align*}
for $l\in\left\{1,2,3\right\} \setminus  \left\{i\right\}$, while $V_{i,B_{2}}^{i}$ is defined by
\begin{align*}
V_{i,B_{2}}^{i}&=\int_{0}^{t}\left\langle E_{i},\left(\mathbb{G}_{\ast}^{t-\tau}+\mathcal{G}^{t-\tau}\right)\Gamma\left(\tilde\phi_{i}E_{i},\tilde\phi_{i}E_{i}\right)(\tau)\right\rangle_{\xi}d\tau\,,
\end{align*}
and
\begin{align*}
r_{i}&=\varepsilon \left\langle E_{i},\mathcal{G}^{t}f_{0}\right\rangle_{\xi}+\sum_{\substack{1\leq j,k\leq 3\\ j\neq k}}\int_{0}^{t}\left\langle E_{i},\left(\mathbb{G}_{\ast}^{t-\tau}+\mathcal{G}^{t-\tau}\right)\Gamma\left(\tilde\phi_{j}E_{j},\tilde\phi_{k}E_{k}\right)(\tau)\right\rangle_{\xi}d\tau
\\
&\quad+\sum_{\substack{1\leq l\leq 3 \\l\neq i}}\int_{0}^{t}\left\langle E_{i},\mathbb{G}_{l}^{t-\tau}\Bigg(\sum_{\substack{1\leq j,k\leq 3\\ j\neq k}}\Gamma\left(\tilde\phi_{j}E_{j},\tilde\phi_{k}E_{k}\right)\Bigg)(\tau)\right\rangle_{\xi}d\tau+\int_{0}^{t}\langle E_{i},\mathcal{G}^{t-\tau}g_{A}(\tau)\rangle_{\xi}d\tau\notag
\\
&\quad+\sum_{\substack{1\leq l\leq 3 \\l\neq i}}\int_{0}^{t}\left\langle E_{i},\mathbb{G}_{l}^{t-\tau}\Bigg(\sum_{\substack{1\leq k\leq 3\\k\neq l}}\Gamma\left(\tilde\phi_{k}E_{k},\tilde\phi_{k}E_{k}\right)(\tau)\Bigg)\right\rangle_{\xi}d\tau
\\
&\quad+\int_{0}^{t}\left\langle E_{i},\left(\mathbb{G}_{\ast}^{t-\tau}+\mathcal{G}^{t-\tau}\right)g_{B}(\tau)\right\rangle_{\xi}d\tau+\sum_{\substack{1\leq l\leq 3 \\l\neq i}}\int_{0}^{t}\left\langle E_{i},\mathbb{G}_{l}^{t-\tau}2\Gamma\left(\tilde\phi_{l}E_{l},V_{l,B}E_{l}\right)(\tau)\right\rangle_{\xi}d\tau\notag
\\
&\quad+\sum_{\substack{1\leq l\leq 3 \\l\neq i}}\int_{0}^{t}\left\langle E_{i},\mathbb{G}_{l}^{t-\tau}\Bigg(\sum_{\substack{1\leq k\leq 3\\k\neq l}}2\Gamma\left(\tilde\phi_{k}E_{k},V_{k,B}E_{k}\right)(\tau)\Bigg)\right\rangle_{\xi}d\tau\notag
\\
&\quad+\sum_{\substack{1\leq l\leq 3 \\l\neq i}}\int_{0}^{t}\left\langle E_{i},\mathbb{G}_{l}^{t-\tau}\Bigg(\sum_{1\leq j,k\leq 3}\Gamma\left(V_{j,B}E_{j},V_{k,B}E_{k}\right)(\tau)\Bigg)\right\rangle_{\xi}d\tau
\\
&\quad+\sum_{\substack{1\leq l\leq 3 \\l\neq i}}\int_{0}^{t}\left\langle E_{i},\mathbb{G}_{l}^{t-\tau}\Bigg(\sum_{\substack{1\leq j,k\leq 3\\ j\neq k}}2\Gamma\left(\tilde\phi_{j}E_{j},V_{k,B}E_{k}\right)\Bigg)(\tau)\right\rangle_{\xi}d\tau
\\
&\quad+\sum_{\substack{1\leq l\leq 3\\l\neq i}}\int_{0}^{t}\left\langle E_{i},\mathbb{G}_{l}^{t-\tau}g_{A}(\tau)\right\rangle_{\xi}d\tau\notag\,.
\end{align*}
From the argument of $V_{2,m}^{\perp}$ in proof of Proposition \ref{prop-V-perp}, the components $V_{i,B_{2}}^{m}\,$, $m=1,2,3$, satisfy the same estimates \eqref{est-der-V-perp-2m}-\eqref{V2-characteristic-bound} of $V_{2,m}^{\perp}$:
\begin{equation}
\label{est-V-iB2-m}
\vert \partial_{x}^{n}V_{i,B_{2}}^{m}\vert\lesssim\varepsilon\theta_{m}^{2+n}\,, \quad\vert \partial_{x}^{n}(\partial_{t}+\lambda_{m}\partial_{x})V_{i,B_{2}}^{m}\vert\lesssim\varepsilon \theta_{m}^{4+n}\quad \text{for any }n\in \mathbb{N}\cup\left\{0\right\}, \ 1\leq i,m\leq 3\,,
\end{equation}
whereas the remaining component $r_{i}$ decays faster in time $t$:
\begin{equation}
\label{est-D-i}
\vert r_{i}\vert\lesssim \varepsilon(1+t+\vert x\vert)^{-\frac{N_{0}}{1-\gamma}}+\varepsilon^{2}(1+t)^{-\frac{1}{2}}\Bigg(\sum_{m=1}^{3}\psi_{m}^{\frac{3}{2}}\Bigg)\,,
\end{equation}
\begin{equation}
\label{est-der-D-i}
 \vert \partial_{x}r_{i}\vert\lesssim \varepsilon(1+t+\vert x\vert)^{-\frac{N_{0}}{1-\gamma}}+\varepsilon^{2}\left[\sum_{l\neq i}(\theta_{l}^{\frac{7}{2}-\frac{\delta_{0}}{2}})+(1+t)^{-\frac{1}{2}}\Bigg(\sum_{m=1}^{3}\psi_{m}^{\frac{3}{2}}\Bigg)\right]\,.
\end{equation}
Combining this decomposition with the definition of the off-diagonal waves $v_{i}$ of the artificial-viscosity system in \eqref{def-v-i}, we write
\begin{align}
\mathcal{S}_{i}(t,x)
={}&
\sum_{\substack{1\le l,m\le3\\(l,m)\ne(i,i)}}
B_i^{lm}\int_0^t
\partial_x\mathbb{H}_i^{t-\tau}*_x
\bigl[
\mathcal{K}_l\tilde u_m+\tilde\phi_l\mathcal{K}_m
\bigr](\tau,x)
\,\mathrm d\tau
\notag\\
&+
\sum_{1\leq l,m,k\leq 3}B_i^{lm}\int_0^t
\partial_x\mathbb{H}_i^{t-\tau}*_x
\Bigl[
2V_{l,B_{2}}^{k}\tilde u_m
\Bigr](\tau,x)
\,\mathrm d\tau
\notag\\
&+
\sum_{1\leq l,m\leq 3}B_i^{lm}\int_0^t
\partial_x\mathbb{H}_i^{t-\tau}*_x
\Bigl[
\mathcal{S}_l(2\tilde u_m+v_m)
+V_{l,B_{1}}(2\mathcal{K}_m+\mathcal{S}_m)
\Bigr](\tau,x)
\,\mathrm d\tau
\notag\\
&+
\sum_{1\leq l,m\leq 3}B_i^{lm}\int_0^t
\partial_x\mathbb{H}_i^{t-\tau}*_x
\Bigl[
2r_l\tilde u_m+V_{l,B_{2}}v_m
\Bigr](\tau,x)
\,\mathrm d\tau
\notag\\
&+
\sum_{1\leq l,m\leq 3}B_i^{lm}\int_0^t
\partial_x\mathbb{H}_i^{t-\tau}*_x
\Bigl[
V_{l,B_{2}}(2\mathcal{K}_m+\mathcal{S}_{m})+V_{l,B}V_{m,B_{2}}
\Bigr](\tau,x)
\,\mathrm d\tau
\notag\\
&+
\frac{1}{2\pi}\sum_{\substack{1\leq l,m\leq 3\\ (l,m)\neq (i,i)}}
\int_0^t\int_{\mathbb R}
e^{\mathrm i x\eta}\chi_\delta(\eta)
e^{\sigma_i(\eta)(t-\tau)}\widehat{\tilde\phi_{l}\tilde\phi_{m}}(\tau,\eta)
\\
&\hspace{20mm}\cdot\left[\langle e_i(\eta),E_i\rangle_\xi\langle\Gamma(E_l,E_m),e_i(-\eta)\rangle_\xi-\mathrm i\eta B_i^{lm}\right]d\eta d\tau
\notag\\
&+
\frac{1}{2\pi}\sum_{\substack{1\leq l,m\leq 3\\ (l,m)\neq (i,i)}}
\int_0^t\int_{\mathbb R}
e^{\mathrm i x\eta}\chi_\delta(\eta)\mathrm (i\eta B_i^{lm})
\left(e^{\sigma_i(\eta)(t-\tau)}-
e^{\alpha_i(\eta)(t-\tau)}\right)
\widehat{\tilde\phi_{l}\tilde\phi_{m}}(\tau,\eta)
\,\mathrm d\eta\,\mathrm d\tau
\notag\\
&+
\frac{1}{2\pi}\sum_{1\leq l,m\leq 3}
\int_0^t\int_{\mathbb R}
e^{\mathrm i x\eta}\chi_\delta(\eta)
e^{\sigma_i(\eta)(t-\tau)}\widehat{(2\tilde\phi_{l}V_{m,B}+V_{l,B}V_{m,B})}(\tau,\eta)
\\
&\hspace{20mm}\cdot\left[\langle e_i(\eta),E_i\rangle_\xi\langle\Gamma(E_l,E_m),e_i(-\eta)\rangle_\xi-\mathrm i\eta B_i^{lm}\right]d\eta d\tau
\notag\\
&+
\frac{1}{2\pi}\sum_{1\leq l,m\leq 3}
\int_0^t\int_{\mathbb R}
e^{\mathrm i x\eta}\chi_\delta(\eta)\mathrm (i\eta B_i^{lm})
\left(e^{\sigma_i(\eta)(t-\tau)}-
e^{\alpha_i(\eta)(t-\tau)}\right)
\notag\\
&\hspace{20mm}\cdot\widehat{(2\tilde\phi_{l}V_{m,B}+V_{l,B}V_{m,B})}(\tau,\eta)
\,\mathrm d\eta\,\mathrm d\tau
\notag\\
&-
\frac{1}{2\pi}\sum_{\substack{1\leq l,m\leq 3\\ (l,m)\neq (i,i)}}
\int_0^t\int_{\mathbb R}
e^{\mathrm i x\eta+\alpha_i(\eta)(t-\tau)}
(i\eta B_{i}^{lm})\bigl(1-\chi_\delta(\eta)\bigr)
\widehat{\tilde\phi_{l}\tilde\phi_{m}}(\tau,\eta)
\,\mathrm d\eta\,\mathrm d\tau
\notag\\
&-
\frac{1}{2\pi}\sum_{1\leq l,m\leq 3}
\int_0^t\int_{\mathbb R}
e^{\mathrm i x\eta+\alpha_i(\eta)(t-\tau)}
 (i\eta B_{i}^{lm})\bigl(1-\chi_\delta(\eta)\bigr)
\widehat{(2\tilde\phi_{l}V_{m,B}+V_{l,B}V_{m,B})}(\tau,\eta)
\,\mathrm d\eta\,\mathrm d\tau
\notag\\
&+\int_{0}^{t}\left\langle E_{i},\mathbb{G}_{i}^{t-\tau} g_{A}(\tau)\right\rangle_{\xi}d\tau
\notag\\
\equiv{}&
Q_{i,1}+\cdots +Q_{i,12},
\label{eq:VBone-mathcal-S-decomposition}
\end{align}
Throughout the proof, the parameter $s$ denotes the order of differentiation with $s\in\left\{0,1\right\}$. For each component of $Q_{i,1}$, we write
\[
Q_{i,1}^{l,m}=B_i^{lm}\int_0^t
\partial_x\mathbb{H}_i^{t-\tau}*_x
\bigl[
\mathcal{K}_l\tilde u_m+\tilde\phi_l\mathcal{K}_m
\bigr](\tau,x)
\,\mathrm d\tau\,.
\]
When $l\neq m$, combining \eqref{est-theta-1-j-times-theta-1-k} with the estimates of $\mathcal{K}_{i}$ and $\tilde u_{i}$ gives
\[
\vert \partial_{x}^{s} Q_{i,1}^{l,m}\vert \lesssim \varepsilon^{2} (1+t)^{-\frac{1+2s}{4}}\Bigg(\psi_{i}^{\frac{3}{2}}+\sum_{k\neq i}\tilde\psi_{k}^{\frac{3}{2}}\Bigg)\,.
\]
When $l=m=j$ and $j\neq i$, we apply the argument of cancellation mechanism used to estimate $J_{2}^{ij}$ in \eqref{eq-J-2-ij}. This gives
\[
\vert \partial_{x}^{s}Q_{i,1}^{jj}\vert\lesssim\varepsilon^{2}(1+t)^{-\frac{1-\delta_{0}+2s}{4}}\left(\psi_{i}^{\frac{3}{2}}+\tilde\psi_{j}^{\frac{3}{2}}\right)\,, \qquad j\neq i\,.
\]
For $Q_{i,2}$, combining \eqref{est-V-iB2-m}, \eqref{est-tilde-ui}, and applying the arguments used to estimates $O_{2}$, $L_{1}$, and $L_{2}$ in the proof of Proposition \ref{prop-remainder} to the corresponding terms, we obtain
\[
\vert \partial_{x}^{s} Q_{i,2}\vert\lesssim\varepsilon^{2}\left[\theta_{i}^{2-\delta_{0}+s}+(1+t)^{-\frac{1+2s}{4}}\Bigg(\psi_{i}^{\frac{3}{2}}+\sum_{k\neq i}\tilde\psi_{k}^{\frac{3}{2}}\Bigg)\right]\,.
\]
For $Q_{i,4}$, applying the argument used to estimates $O_{1}$ in the proof of Proposition \ref{prop-remainder}, we obtain
\[
\vert \partial_{x}^{s}Q_{i,4}\vert\lesssim\varepsilon^{2}(1+t)^{-\frac{1-\delta_{0}+2s}{4}}\Bigg(\psi_{i}^{\frac{3}{2}}+\sum_{k\neq i}\tilde\psi_{k}^{\frac{3}{2}}\Bigg)\,.
\]
For $Q_{i,6},\cdots,Q_{i,9}$, we first note that
\begin{align*}
&\chi_{\delta}(\eta)e^{\sigma_{i}(\eta)t}\Big(\langle e_i(\eta),E_i\rangle_\xi\langle\Gamma(E_l,E_m),e_i(-\eta)\rangle_\xi-\mathrm i\eta B_i^{lm}\Big)
\\
&=(i\eta)^{2}\left[\frac{\chi_{\delta}(\eta)e^{\sigma_{i}(\eta)t}}{(i\eta)^{2}}\left( \langle \eta^{2}R_{e,i}(\eta),\Gamma(E_{l},E_{m})\rangle_{\xi}+\langle e_{i}(\eta)-E_{i},E_{i}\rangle_{\xi}\langle e_{i}(\eta),\Gamma(E_{l},E_{m})\rangle_{\xi}\right)\right]
\\
&=(i\eta)^{2}\widehat{\mathcal{D}_{i}^{lm}}(\eta,t)\,,
\end{align*}
the kernel $\mathcal{D}_{i}^{lm}$ satisfies \eqref{asumption-cancellation-mechanism-1} with characteristic velocity $\lambda_{i}$.
Combining this with the kernel $\mathcal{D}_{i}$ defined by \eqref{eq-artificial-boltzmann}, we can write $Q_{i,6}+\cdots +Q_{i,9}$ as
\begin{align*}
\sum_{\substack{1\leq l,m\leq 3\\(l,m)\neq (i,i)}}&\int_{0}^{t} \partial_{x}^{2}\left(\mathcal{D}_{i}^{lm}+B_{i}^{lm}\mathcal{D}_{i}\right)(t-\tau)\ast_{x}(\tilde\phi_{l}\tilde\phi_{m})(\tau)d\tau
\\
&\quad+\sum_{1\leq l,m\leq 3}\int_{0}^{t}\partial_{x}^{2}\left(\mathcal{D}_{i}^{lm}+B_{i}^{lm}\mathcal{D}_{i}\right)(t-\tau)\ast_{x}(2\tilde\phi_{l}V_{m,B}+V_{l,B}V_{m,B})(\tau)d\tau\,.
\end{align*}
Applying the arguments used to estimates $L_{3}$, $L_{5}$, $L_{8}$, $L_{9}$, $L_{10}$, and $L_{11}$ in the proof of Proposition \ref{prop-V-i} to the corresponding terms, we obtain
\[
\vert \partial_{x}^{s}Q_{i,k}\vert\lesssim \varepsilon^{2}(1+t)^{-\frac{1+s}{2}}\Bigg(\sum_{m=1}^{3}\psi_{m}^{\frac{3}{2}}\Bigg)\,, \qquad k=6,7\,,
\]
and
\[
\vert Q_{i,k}\vert\lesssim \varepsilon^{2}(1+t)^{-\frac{1}{2}}\Bigg(\sum_{m=1}^{3}\psi_{m}^{\frac{3}{2}}\Bigg)\,,  \quad \vert \partial_{x}Q_{i,k}\vert\lesssim \varepsilon^{2}\left[\ln(1+t)\theta_{i}^{\frac{7}{2}}+(1+t)^{-1}\Bigg(\sum_{m=1}^{3}\psi_{m}^{\frac{3}{2}}\Bigg)\right]\,,
\]
for $k=8,9$.
For $Q_{i,10}$ and $Q_{i,11}$, applying the argument used to estimate $J_{4}$ in the proof of Lemma \ref{lemm-mathcal-K}, we yields
\[
\vert \partial_{x}^{s}Q_{i,10}\vert\lesssim \varepsilon^2 \left[e^{-ct}\left(1+\vert x-\lambda_{i}t\vert\right)^{-L}\right]\ast_{x,t}\Bigg(\sum_{m=1}^{3}\theta_{m}^{3+s}\Bigg)\lesssim \varepsilon^{2}\Bigg(\sum_{m=1}^{3}\theta_{m}^{3+s}\Bigg)\,, \quad \text{for any fixed }L\geq 2N.
\]
and
\[
\vert \partial_{x}^{s}Q_{i,11}\vert\lesssim \varepsilon^2 \left[e^{-ct}\left(1+\vert x-\lambda_{i}t\vert\right)^{-L}\right]\ast_{x,t}\Bigg(\sum_{m=1}^{3}\theta_{m}^{\frac{5}{2}+s}+(1+t)^{-\frac{s}{2}}\psi_{m}^{3}\Bigg)\lesssim \varepsilon^{2}(1+t)^{-\frac{1+s}{2}}\Bigg(\sum_{m=1}^{3}\psi_{m}^{\frac{3}{2}}\Bigg)\,,
\]
for any fixed $L\geq 3$.
Note that we place $(i\eta)$ and the external derivative factor $(i\eta)^{s}$ on $\widehat{\tilde\phi_{l}\tilde\phi_{m}}$ in $Q_{i,10}$; for $Q_{i,11}$, we place $(i\eta)$ on $e^{\alpha_{i}(\eta)(t-\tau)}(1-\chi_{\delta}(\eta))$ and the external derivative factor $(i\eta)^{s}$ on the source $\widehat{(2\tilde\phi_{l}V_{m,B}+V_{l,B}V_{m,B})}$.
For $Q_{i,12}$, we directly obtain
\[
\vert \partial_{x}^{s}Q_{i,12}\vert\lesssim \varepsilon^{2}\theta_{i}^{2+s}\,.
\]
In view of the estimates above, we impose the following ansatz for $\mathcal{S}_{i}$:
\begin{equation}
\label{asatz-mathcal-S-i}
\left\vert \partial_{x}^{s}\mathcal{S}_{i}\right\vert \leq M\varepsilon^{2}\left[\theta_{i}^{2-\delta_{0}+s}+(1+t)^{-\frac{1-\delta_{0}+2s}{4}}\Bigg(\psi_{i}^{\frac{3}{2}}+\sum_{k\neq i}\tilde\psi_{k}^{\frac{3}{2}}\Bigg)\right]\,,
\end{equation}
where the constant $M$ will be chosen later.
For $Q_{i,3}$, applying the argument used to estimates $K_{4}$, $J_{1}$, and $J_{2}$ in the proof of Proposition \ref{prop-remainder} to the corresponding terms, we obtain
\[
\vert \partial_{x}^{s}Q_{i,3}\vert \lesssim(\varepsilon^{2}+M\varepsilon^{3})\left[\theta_{i}^{2-\delta_{0}+s}+(1+t)^{-\frac{1-\delta_{0}+2s}{4}}\Bigg(\psi_{i}^{\frac{3}{2}}+\sum_{k\neq i}\tilde\psi_{k}^{\frac{3}{2}}\Bigg)\right]\,.
\]
For $Q_{i,5}$, applying the argument used to estimates $K_{3}$ in the proof of Proposition \ref{prop-remainder}, one yields
\[
\vert \partial_{x}^{s}Q_{i,5}\vert \lesssim(\varepsilon^{2}+\varepsilon^{3}M)\left[\theta_{i}^{2-\delta_{0}+s}+(1+t)^{-\frac{1-\delta_{0}+2s}{4}}\Bigg(\psi_{i}^{\frac{3}{2}}+\sum_{k\neq i}\tilde\psi_{k}^{\frac{3}{2}}\Bigg)\right]\,.
\]
Combining the component estimates, the known forcing has amplitude
$C_{0}\varepsilon^{2}$, whereas the terms containing $\mathcal{S}_{i}$ have amplitude at most
$C_{1}\varepsilon^3M$, with the corresponding profiles in \eqref{asatz-mathcal-S-i}.
Here $C_{0}\geq1$ and $C_{1}$ depend only on the
initial-data norm, and are independent of $M$, $\varepsilon$, and the terminal time.
Thus the output coefficient is bounded by
\[
C_{0}\varepsilon^{2}+C_{1}\varepsilon^{3}M=\varepsilon^{2}\left(C_{0}+C_{1}\varepsilon M\right)
\]
Choose $M \geq 4 C_{0}$ sufficiently large and $0<\varepsilon\leq1$ sufficiently small that
\[
C_{1}\varepsilon M\leq\frac14 M.
\]
The output is then bounded by $\frac{1}{2}M\varepsilon^{2}$, which closes the ansatz and completes the proof.

\end{proof}
The exact decomposition established in Section~\ref{wave} gives
\[
 \langle\mathbb P_0f,E_i\rangle_\xi-u_i
       =\mathcal K_i+\mathcal{S}_{i}+V_{i,B_{2}}+V_{i,A}+R_i.
\]
Combining \eqref{est-mathcal-K-i}, the refined estimates \eqref{est-V-i-A}, \eqref{est-R-iA}--\eqref{est-R-iB} and
\eqref{est-mathcal-S}-\eqref{est-D-i}, we obtain
\begin{align}
 \label{sec4:kinetic-artificial-pointwise}
 &\left|\langle\mathbb P_0f,E_i\rangle_\xi-u_i\right|\\
 &\quad\lesssim_{\delta_0}\varepsilon\left[
       (1+t+|x|)^{-N_0/(1-\gamma)}+\theta_i^{2-\delta_0}
                                      +\sum_{j\ne i}\theta_j^2\right]\notag+\varepsilon^2(1+t)^{-1/4+\delta_0/4}
                  \Bigg(\psi_i^{3/2}+\sum_{j\ne i}\tilde\psi_j^{3/2}\Bigg).
                  \notag
\end{align}
Integrating \eqref{sec4:kinetic-artificial-pointwise} over $x\in\mathbb R$ yields
\begin{align}
\label{boltzmann-artificial-L-1}
 \|\langle\mathbb P_0f,E_i\rangle_\xi-u_i\|_{L_x^1}
       &\lesssim_{\delta_0}\varepsilon(1+t)^{-\frac{1}{2}+\frac{\delta_{0}}{2}}.
\end{align}

We next compare the quadratic physical-viscosity system
\eqref{NS-burger} with the artificial-viscosity system \eqref{artificial}.
They have the same quadratic flux and the same initial data, while their
viscosity matrices are $(A_{ij})$ and
$\operatorname{diag}(A_{11},A_{22},A_{33})$, respectively. Since the two systems have the same quadratic flux and the same initial data, and differ only in their viscosity matrices, \cite[Remark~8.9, (8.48)--(8.52),
pp.~109--110]{[Zeng]}, applies directly.

Applying \cite[Remark~8.9, (8.52)]{[Zeng]} with the arbitrarily small
loss chosen to be $\frac{\delta_{0}}{2}$, and substituting the profiles from
\cite[(2.19) and (2.25)]{[Zeng]}, we obtain
\begin{align}\label{sec42:physical-artificial-pointwise}
 |\omega_i(x,t)-u_i(x,t)|
 &\lesssim_{\delta_0}\varepsilon (1+t)^{-\frac{1-2\delta_{0}}{4}}
       \sum_{j=1}^3\left[
       \frac{1}{(1+t+|x-\lambda_j(1+t)|^2)^{\frac{3}{4}}}
\right.
\notag\\
       &\quad \left.+\frac{1}{(1+t+|x-\lambda_j(1+t)|)^{\frac{1}{2}}
                       (1+t+|x-\lambda_j(1+t)|^2)^{\frac{3}{8}}}\right]\notag\\[2mm]
 &\quad+\varepsilon (1+t)^{-\frac{1}{2}}|x-\lambda_i(1+t)|^{-1+\frac{\delta_{0}}{2}}
       \mathbf1_{\left\{\frac{2\sqrt{1+t}}{\delta_{0}}\leq|x-\lambda_i(1+t)|\leq\frac{\delta_{0}(1+t)}{2}\right\}}\,.
\end{align}
The restricted term is understood to vanish outside the indicated set, including
when that set is empty. Consequently,
\begin{equation}\label{sec42:physical-artificial-norms}
  \|\omega(t)-u(t)\|_{L_x^1}
    \lesssim_{\delta_0}\varepsilon(1+t)^{-\frac{1}{2}+\frac{\delta_{0}}{2}}.
\end{equation}
Finally,
\[
 \left\langle\mathbb P_0f,E_i\right\rangle_\xi-\omega_i
 =\left(\left\langle\mathbb P_0f,E_i\right\rangle_\xi-u_i\right)
       +(u_i-\omega_i).
\]
Combining \eqref{sec4:kinetic-artificial-pointwise} with
\eqref{sec42:physical-artificial-pointwise} gives the pointwise estimate
in Part~(I) of the main theorem. The corresponding $L_x^1$ estimate follows from \eqref{boltzmann-artificial-L-1} and
\eqref{sec42:physical-artificial-norms}\,.

\subsection{Comparison with the full Navier--Stokes system}
\label{sec43:full-comparison}

We now compare $u_i$ in \eqref{artificial} with the full-flux
artificial solution $u_i^f$, and then compare $u_i^f$ with the full NS
solution $\om_i^f$. The three systems are written in the same Eulerian spatial coordinate and have the same initial data. The coefficients $\lambda_i$,
$A_{ij}$, $A_i=A_{ii}$, and $B_i^{kl}$ are those in
\eqref{NS-burger} and \eqref{artificial}. We write
\[
 \mathcal Q_i(z,\widetilde z)
 =\sum_{k,l=1}^3B_i^{kl}z_k\widetilde z_l,
 \qquad i=1,2,3.
\]
Since $B_i^{kl}=B_i^{lk}$, $\mathcal Q$ is a symmetric bilinear form.

Let $\mathcal F$ and $\mathcal A$ denote the full Euler flux and
physical viscosity matrix of the compressible NS system, expressed
in the equilibrium characteristic coordinates corresponding to $E_i$.
The full physical equations are discussed in
\cite[Section~9]{[Zeng]}. Here we use the abstract conservation-law form
\cite[(2.1), (2.8)--(2.10)]{[Zeng]}, rather than the component formulas.
After subtracting the equilibrium flux, the normalization is
\begin{equation}\label{sec43:coefficient-matching}
 \begin{aligned}
 \mathcal F(0)&=0,\qquad
 D_z\mathcal F(0)=\operatorname{diag}(\lambda_1,\lambda_2,\lambda_3),\\
 \frac12D_z^2\mathcal F(0)[z,z]&=-\mathcal Q(z,z),\qquad
 \mathcal A(0)=(A_{ij}).
 \end{aligned}
\end{equation}

Define the exact nonlinear remainders by
\begin{equation}\label{sec43:cubic-flux}
 \begin{aligned}
 \mathcal R_{3,i}(z)
   &=\mathcal F_i(z)-\lambda_i z_i+\mathcal Q_i(z,z),\\
 \mathcal V_i(z,q)
   &=\sum_{j=1}^3\big(\mathcal A_{ij}(z)-A_{ij}\big)q_j.
 \end{aligned}
\end{equation}
Thus $\mathcal R_3$ contains all higher-order Euler terms, rather than
only the cubic Taylor polynomial, while $\mathcal V$ is linear in
the gradient variable $q$. The full-flux artificial system is
\begin{equation}\label{sec43:full-artificial}
 \left\{\begin{aligned}
 \partial_tu_i^f+\lambda_i\partial_xu_i^f
 &=A_i\partial_x^2u_i^f+\partial_x\mathcal Q_i(u^f,u^f)
                    -\partial_x\mathcal R_{3,i}(u^f),\\
 u_i^f(0)&=u_i(0).
 \end{aligned}\right.
\end{equation}
The full NS system takes the form
\begin{equation}\label{sec43:full-ns-characteristic}
 \left\{\begin{aligned}
 \partial_t\om_i^f+\lambda_i\partial_x\om_i^f
 &=\sum_{j=1}^3A_{ij}\partial_x^2\om_j^f
            +\partial_x\mathcal Q_i(\om^f,\om^f)\\
 &\quad-\partial_x\mathcal R_{3,i}(\om^f)
            +\partial_x\mathcal V_i(\om^f,\partial_x\om^f),\\
 \om_i^f(0)&=u_i(0),\qquad i=1,2,3.
 \end{aligned}\right.
\end{equation}
The two equations have the same full Euler flux. In the second equation,
the off-diagonal entries of $A_{ij}$ remain in the linear viscous
operator.

All constitutive maps are smooth in a fixed neighborhood of
zero, and for sufficiently small $\varepsilon$ the solutions considered below remain in this neighborhood. Taylor's formula and bilinearity give
\begin{equation}\label{sec43:structural-bounds}
 \begin{aligned}
 |\mathcal Q(z,\widetilde z)|&\lesssim |z|\,|\widetilde z|,\\
 |D_z^j\mathcal R_3(z)|&\lesssim |z|^{3-j},\quad j=0,1,2,
       \qquad |D_z^3\mathcal R_3(z)|\lesssim1,\\
 |\mathcal V(z,q)|&\lesssim |z|\,|q|.
 \end{aligned}
\end{equation}
Here $D_z$ denotes differentiation in the state variable. We will also use
\begin{equation}\label{sec43:difference-bounds}
 \begin{aligned}
 \mathcal Q(z,z)-\mathcal Q(\widetilde z,\widetilde z)
   &=\mathcal Q(z+\widetilde z,z-\widetilde z),\\
 |\mathcal R_3(z)-\mathcal R_3(\widetilde z)|
   &\lesssim (|z|^2+|\widetilde z|^2)|z-\widetilde z|.
 \end{aligned}
\end{equation}
The common initial data are
\begin{equation}\label{sec43:common-data}
 u_i^f(0)=\om_i^f(0)=u_i(0)
     =\varepsilon\langle\check\chi*_xf_0,E_i\rangle_\xi.
\end{equation}
The estimates \eqref{sec42:data-decay}--\eqref{sec42:data-primitives}, together with the one-sided pointwise decay established in the preceding subsection, apply to these common data. Hence the initial-data hypotheses of \cite[Theorems~2.3 and~2.9]{[Zeng]} hold with size $O(\varepsilon)$. \cite[Theorem~2.9]{[Zeng]} applies to $u$ and $u^f$, since their transport linearization is diagonal and their common viscosity is $\operatorname{diag}(A_1,A_2,A_3)>0$; the equilibrium symmetrizer can be taken to be the identity. The small physical solution is covered by \cite[Theorems~7.1 and~2.3]{[Zeng]} under the structural hypotheses above.

\begin{proposition}\label{sec43:prop-flux-comparison}
For $0<\delta_{0}<\frac{1}{4}$ and sufficiently small $\varepsilon$, the solutions
$u_i$ and $u_i^f$ with the common initial data satisfy, for $s=0,1$,
\begin{align}\label{sec43:flux-pointwise}
 |\partial_x^s(u_i^f-u_i)(x,t)|
 &\lesssim_{\delta_{0}}\varepsilon^3\left[
       \theta_i^{2+s-\delta_{0}}
       +(1+t)^{-\frac{1-2\delta_{0}}{4}-\frac{s}{2}}
          \Bigg(\psi_i^{3/2}+\sum_{j\ne i}\tilde\psi_j^{3/2}\Bigg)\right].
\end{align}
In particular,
\begin{equation}\label{sec43:flux-norms}
 \begin{aligned}
  \|\partial_x^s(u^f-u)(t)\|_{L_x^1}
       &\lesssim_{\delta_{0}}\varepsilon^3(1+t)^{-\frac{1+s}{2}+\frac{\delta_{0}}{2}},
       \qquad s=0,1.
 \end{aligned}
\end{equation}
\end{proposition}
\begin{proof}
Set $v_i^f=u_i^f-\tilde u_i$ and
$d_i=u_i^f-u_i=v_i^f-v_i$, where $v_i=u_i-\tilde u_i$
is defined in \eqref{off-artificial}.
Since $d_i(0)=0$, subtracting
\eqref{artificial} from \eqref{sec43:full-artificial} gives
\begin{equation}\label{sec43:difference-duhamel}
 u_i^f(t)-u_i(t)=\int_0^t\partial_x\mathbb H_i^{t-\tau}*_x
 [\mathcal Q_i(u^f+u,u^f-u)-\mathcal R_{3,i}(u^f)](\tau)\,d\tau.
\end{equation}

Since $u^f+u=2\tilde u+v^f+v$, the difference admits the exact decomposition
\begin{align}
 d_i&=\sum_{m=1}^{3}\int_0^t\partial_x\mathbb H_i^{t-\tau}*_xF_{im}(\tau)\,d\tau\notag\\
 &\quad+2\sum_{\substack{1\leq m,l\leq 3\\m\ne l}}B_i^{ml}
       \int_0^t\partial_x\mathbb H_i^{t-\tau}*_x
                      (\tilde u_md_l)(\tau)\,d\tau\notag\\
 &\quad+\sum_{1\leq m,l\leq 3}B_i^{ml}
       \int_0^t\partial_x\mathbb H_i^{t-\tau}*_x
                      ((v_m^f+v_m)d_l)(\tau)\,d\tau\notag\\
 &\quad-\int_0^t\partial_x\mathbb H_i^{t-\tau}*_x
   \left[\mathcal R_{3,i}(u^f)
                -\sum_{m=1}^{3}\mathcal R_{3,i}(\tilde u_m\mathbf e_m)\right](\tau)\,d\tau
\notag\\
&=J_{1}+J_{2}+J_{3}+J_{4}\,,
\label{sec43:difference-four-terms}
\end{align}
where
\begin{equation}\label{sec43:signed-family-source}
 F_{im}=2B_i^{mm}\tilde u_md_m
                   -\mathcal R_{3,i}(\tilde u_m\mathbf e_m).
\end{equation}
Throughout the proof, the parameter $s$ denotes the order of differentiation with $s\in\left\{0,1\right\}$. To estimate these terms, we have
\begin{equation}\label{sec43:individual-artificial}
 |\partial_x^s( v_{i}^{f})|+|\partial_x^s(v_{i})|
 \lesssim\varepsilon^2(1+t)^{-s/2}
       \left(\psi_i^{3/2}+\sum_{j\ne i}\tilde\psi_j^{3/2}\right),
 \qquad s=0,1\,,
\end{equation}
by a modification of the argument in \cite[Chapter 4]{[Zeng]}.
We first consider $J_{4}$, since it depends only on the known components $\tilde{u}$ and $v^{f}$.
For $J_{4}$, set
\[
\mathcal{R}_{3,i}(u^{f})-\Bigg[\sum_{m=1}^{3}\mathcal{R}_{3,i}(\tilde{u}_{m}e_{m})\Bigg]=\left[\mathcal{R}_{3,i}(u^{f})-\mathcal{R}_{3,i}(\tilde{u})\right]+\left[\mathcal{R}_{3,i}(\tilde{u})-\sum_{m=1}^{3}\left(\mathcal{R}_{3,i}(\tilde{u}_{m}e_{m})\right)\right]=B_{1}+B_{2}\,.
\]
 By the fundamental theorem of calculus, \eqref{sec43:structural-bounds}, and \eqref{sec43:individual-artificial}, the first term $B_{1}$ satisfies
\begin{align}
 \left\vert B_{1}\right\vert&\leq\left\vert\sum_{k=1}^3v_k^f\int_0^1
       \partial_{z_k}\mathcal R_{3,i}(\tilde u+\sigma v^f)\,d\sigma\right\vert
       \lesssim \left\vert v^{f}\right\vert\left( \vert \tilde{u}\vert +\left\vert v^{f}\right\vert \right)^{2}
       \lesssim \varepsilon^{4}\sum_{m=1}^{3}\left[\theta_{m}^{\frac{7}{2}}+(1+t)^{-\frac{3}{4}}\psi_{m}^{3}\right]
\label{sec43:full-flux-remainder-identity}
\end{align}
Differentiating $B_{1}$, we have
\begin{align}
 \left\vert\partial_x B_{1}\right\vert
 &\leq \left\vert\int_0^1\left\{
 \left(D_z^2\mathcal R_{3,i}(\tilde u+\sigma v^f)\right)
                       \left(\partial_x\tilde u+\sigma\partial_xv^f\right)(v^f)
 +\left(D_z\mathcal R_{3,i}(\tilde u+\sigma v^f)\right)\left(\partial_xv^f\right)
                         \right\}\,d\sigma\right\vert
\notag\\
&\lesssim \left(\vert \tilde{u}\vert +\left\vert v^{f}\right\vert\right)\left(\left\vert \partial_{x}\tilde{u}\right\vert+\left\vert \partial_{x}v^{f}\right\vert\right)\left\vert v^{f}\right\vert+\left(\vert \tilde{u}\vert +\left\vert v^{f}\right\vert\right)^{2}\left\vert \partial_{x}v^{f}\right\vert
\notag\\
&\lesssim\varepsilon^{4}\sum_{m=1}^{3}\left[\theta_{m}^{\frac{9}{2}}+(1+t)^{-\frac{5}{4}}\psi_{m}^{3}\right]\,.
\label{sec43:full-flux-remainder-derivative}
\end{align}
For the second term, integrating the mixed derivatives over the corresponding rectangles gives
\begin{align}
 B_{2}&=\tilde u_1\tilde u_2\int_0^1\!\!\int_0^1
       \partial_{z_1z_2}\mathcal R_{3,i}
          (\rho\tilde u_1,\sigma\tilde u_2,0)\,d\rho\,d\sigma+\tilde u_1\tilde u_3\int_0^1\!\!\int_0^1
       \partial_{z_1z_3}\mathcal R_{3,i}
          (\rho\tilde u_1,\rho\tilde u_2,\sigma\tilde u_3)\,d\rho\,d\sigma\notag\\
 &\quad+\tilde u_2\tilde u_3\int_0^1\!\!\int_0^1
       \partial_{z_2z_3}\mathcal R_{3,i}
          (\rho\tilde u_1,\rho\tilde u_2,\sigma\tilde u_3)\,d\rho\,d\sigma.
\label{sec43:full-flux-mixed-family-identity}
\end{align}
Together with \eqref{est-tilde-ui}, \eqref{sec43:structural-bounds}, and \eqref{est-theta-1-j-times-theta-1-k}, these exact identities yield

\begin{align}
 &\left|\partial_x^s\left(
       \mathcal R_{3,i}(\tilde u)
          -\sum_m\mathcal R_{3,i}(\tilde u_m\mathbf e_m)\right)\right|\lesssim\varepsilon^3(1+t)^{-\frac{3+s}{2}}\sum_{m=1}^{3}\left(\psi_m^{2N}\right)\,.
\label{sec43:full-flux-mixed-family-bound}
\end{align}
Therefore,
\begin{align}
\left\vert \partial_{x}^{s}J_{4}\right\vert&\lesssim\varepsilon^3\left[
 \int_0^{t/2}\theta_{i}^{2+s}(t-\tau)\ast_{x}
\Bigg(\sum_{m=1}^{3}\theta_{m}^{\frac{7}{2}}+(1+\tau)^{-\frac{3}{4}}\psi_{m}^{3}\Bigg)(\tau)\,d\tau\right.\notag\\
 &\quad\left.+\int_{t/2}^t\theta_{i}^{2}(t-\tau)*_x\Bigg(\sum_{m=1}^{3}\theta_{m}^{\frac{7}{2}+s}+(1+\tau)^{-\frac{3+2s}{4}}\psi_{m}^{3}\Bigg)(\tau)\,d\tau
\right]\notag\\
 &\quad\lesssim\varepsilon^3\Bigg[\theta_{i}^{2+s}+(1+t)^{-\frac{1+2s}{4}}\Bigg(\psi_i^{\frac{3}{2}}+\sum_{j\ne i}\tilde\psi_j^{\frac{3}{2}}\Bigg)\Bigg]\,.
\label{sec43:full-flux-remainder-integral}
\end{align}
Motivated by the estimate of $J_{4}$, we impose the following ansatz for $d_{i}$:
\begin{equation}\label{asatz-d-i}
 \left\vert\partial_x^sd_i(t,x)\right\vert\leq M\varepsilon^{3}
\left[\theta_i^{2+s-\delta_0}
 +(1+t)^{-\frac{1-2\delta_{0}+2s}{4}}
  \Bigg(\psi_i^{3/2}+\sum_{j\ne i}\tilde\psi_j^{3/2}\Bigg)\right]\,,
\end{equation}
where the constant $M\geq 1$ will be chosen later. We next consider the term $J_{1}$. For convenience, we denote the components $J_{1}^{m}$ of $J_{1}$ by
\[
J_{1}^{m}=\int_{0}^{t}\partial_{x}\mathbb{H}_{i}^{t-\tau}\ast_{x}F_{im}(\tau)d\tau\,.
\]
When $m=i$, we directly obtain
\begin{align*}
\vert \partial_{x}^{s}J_{1}^{i}\vert &\lesssim(M\varepsilon^{4}+\varepsilon^{3})\left[\int_{0}^{\frac{t}{2}}\theta_{i}^{2+s}(t-\tau)\ast_{x}\theta_{i}^{3-\delta_{0}}(\tau)d\tau+\int_{\frac{t}{2}}^{t}\theta_{i}^{2}(t-\tau)\ast_{x}\theta_{i}^{3-\delta_{0}+s}(\tau)d\tau\right]
\\
&\lesssim
(M\varepsilon^{4}+\varepsilon^{3})\theta_{i}^{2+s-\delta_{0}}\,.
\end{align*}
When $m\neq i$ and $t\geq4$, we split the time integral at $\sqrt t$ and $t-\sqrt{t}$, namely, $$J_1^{m}=\int_{0}^{\sqrt{t}}\cdots+\int_{\sqrt{t}}^{t-\sqrt{t}}\cdots+\int_{t-\sqrt{t}}^{t}\cdots\equiv J^{m}_{1\left\{[0,\sqrt{t}]\right\}}+J^{m}_{1\left\{[\sqrt{t},t-\sqrt{t}]\right\}}+J^{m}_{1\left\{[t-\sqrt{t},t]\right\}}\,.$$ The early-time part satisfies
\[
 \left\vert\partial_{x}^{s}J^{m}_{1\left\{[0,\sqrt{t}]\right\}}\right\vert\lesssim(M\varepsilon^{4}+\varepsilon^{3})
 \int_0^{\sqrt t}\theta_i^{2+s}(t-\tau)*_x\theta_m^{3-\delta_{0}}(\tau)\,d\tau
 \lesssim(M\varepsilon^{4}+\varepsilon^{3})\theta_i^{2+s-\frac{\delta_{0}}{2}}\,.
\]
On the last time interval, setting $r=t-\tau$ gives
\begin{align*}
\left\vert \partial_{x}^{s}J^{m}_{1\left\{[t-\sqrt{t},t]\right\}}\right\vert
&=\left\vert \int_{0}^{\sqrt{t}}\partial_{x}^{s}F_{im}(t-r)\ast_{x}\partial_{x}\mathbb{H}_{i}^{r}dr\right\vert
\\
 &\quad\lesssim(M\varepsilon^{4}+\varepsilon^{3})
 (1+t)^{-(3+s-\delta_0)/2}
 \left(1+\frac{|x-\lambda_mt|^2}{1+t}\right)^{-N}
                            \int_0^{\sqrt t}r^{-1/2}\,dr\\
 &\quad\lesssim(M\varepsilon^{4}+\varepsilon^{3})
                  \theta_m^{\frac{5}{2}+s-\delta_0}
\end{align*}
We now turn to the middle-time part. From the full artificial system \eqref{sec43:full-artificial}, the artificial system \eqref{artificial}, and the kinetic Burgers equation \eqref{burger-artificial}, we obtain the following characteristic equations:
\begin{equation}\label{sec43:difference-characteristic-flux}
 (\partial_t+\lambda_m\partial_x)d_m=\partial_x g_{m},
 \qquad
  g_m=A_m\partial_xd_m+
 \mathcal Q_m(u^f+u,d)-\mathcal R_{3,m}(u^f)\,,
\end{equation}
\begin{equation}
\label{sec43:kinetic-burger-flux}
(\partial_{t}+\lambda_m\partial_x)\tilde{u}_{m}=\partial_{x}q_{m}\,, \qquad q_{m}=A_{m}\partial_{x}\tilde{u}_{m}+B_{m}^{mm}(\tilde{u}_{m})^{2}\,.
\end{equation}
These characteristic equations imply
\begin{align*}
(\partial_t+\lambda_m\partial_x)F_{im}&=2B_{i}^{mm}\left[\left(\partial_{x}q_{m}\right)d_{m}+\partial_{x}(\tilde{u}_{m}g_{m})-(\partial_{x}\tilde{u}_{m})g_{m}\right]-\partial_{z_{m}}\mathcal{R}_{3,i}(\tilde{u}_{m}\mathbf{e}_{m})\left[\partial_{x}q_{m}\right]
\end{align*}
Recalling the known bound
 $$\vert u(t)\vert_{L^{\infty}_{x}}+\vert u^f(t)\vert_{L^{\infty}_{x}}
\lesssim\varepsilon(1+t)^{-1/2}\,,
$$
and combining it with the ansatz for $d$ in \eqref{asatz-d-i}, we obtain
\[
\left\vert g_{m}(t)\right\vert_{L^{\infty}_{x}}\lesssim \varepsilon^{3}M (1+t)^{-\frac{3}{2}+\frac{\delta_{0}}{2}}\,.
\]
The cancellation can now be seen directly in physical space. Using $(\partial_{t}+\lambda_{i}\partial_{x})\mathbb{H}_{i}^{t-\tau}=A_{i}\partial_{x}^{2}\mathbb{H}_{i}^{t-\tau}$ and transferring the spatial derivatives by integration by parts, we obtain
\begin{align*}
 &\left\vert(\lambda_{m}-\lambda_{i})\partial_{x}^{s}J_{1\left\{[\sqrt{t}, \,t-\sqrt{t}]\right\}}^{m}\right\vert
 \\
 &\leq \Bigg\vert\big[\partial_{x}^{s}\mathbb{H}_{i}^{t-\tau}\ast_{x}F_{im}(\tau)\big]_{\sqrt t}^{t-\sqrt{t}}\Bigg\vert
   +\left\vert\int_{\sqrt t}^{t-\sqrt{t}}\partial_{x}^{s}(\partial_{t}+\lambda_{i}\partial_{x})\mathbb{H}_{i}^{t-\tau}*_x F_{im}(\tau)\,d\tau\right\vert\\
 &\quad+\left\vert\int_{\sqrt t}^{t-\sqrt{t}} \partial_{x}^{s}\mathbb{H}_{i}^{t-\tau}*_x(\partial_\tau+\lambda_m\partial_x)F_{im}(\tau)\,d\tau\right\vert
              \\
 &\lesssim(M\varepsilon^{4}+\varepsilon^{3})\Big[\theta_{i}^{1+s}(\sqrt{t})\ast_{x}\theta_{m}^{3-\delta_{0}}(t-\sqrt{t})+\theta_{i}^{1+s}(t-\sqrt{t})\ast_{x}\theta_{m}^{3-\delta_{0}}(\sqrt{t})
  \\
 & +\int_{\sqrt t}^{t-\sqrt{t}}\theta_i^{2+s}(t-\tau)*_x\theta_m^{4-\delta_{0}}(\tau)\,d\tau+\int_{\sqrt t}^{t-\sqrt{t}}\theta_i^{1+s}(t-\tau)*_x\theta_m^{5-\delta_{0}}(\tau)\,d\tau\Big]\\
 &\lesssim(M\varepsilon^{4}+\varepsilon^{3})(1+t)^{-\frac{1+2s-2\delta_{0}}{4}}\left(\psi_{i}^{\frac{3}{2}}+\tilde\psi_{m}^{\frac{3}{2}}\right)
\end{align*}
Hence, $J_{1}$ satisfies the ansatz \eqref{asatz-d-i}. We finally estimate $J_{2}$ and $J_{3}$. By \eqref{sec43:individual-artificial}, the ansatz \eqref{asatz-d-i} for $d_{i}$, and \eqref{eq-theta-psi-j-neq-k}, we obtain
\begin{align*}
\vert \partial_{x}^{s}J_{2}\vert &\lesssim\varepsilon^{4}M\left[\int_{0}^{\frac{t}{2}}\theta_{i}^{2+s}(t-\tau)\ast_{x}\Bigg(\sum_{m=1}^{3}\theta_{m}^{\frac{7}{2}-\delta_{0}}+(1+\tau)^{-\frac{3-2\delta_{0}}{4}}\psi_{m}^{3}\Bigg)(\tau)d\tau\right.
\\
&\quad \left.+\int_{\frac{t}{2}}^{t}\theta_{i}^{2}(t-\tau)\ast_{x}\Bigg(\sum_{m=1}^{3}\theta_{m}^{\frac{7}{2}-\delta_{0}+s}+(1+\tau)^{-\frac{3-2\delta_{0}+2s}{4}}\psi_{m}^{3}\Bigg)(\tau)d\tau\right]
\\
&\lesssim\varepsilon^{4}M(1+t)^{-\frac{1-2\delta_{0}+2s}{4}}\Bigg(\psi_{i}^{\frac{3}{2}}+\sum_{j\neq i}\tilde\psi_{j}^{\frac{3}{2}}\Bigg)\,.
\end{align*}
For $J_{3}$, \eqref{sec43:individual-artificial} and the ansatz of $d_{i}$ \eqref{asatz-d-i} give
\begin{align*}
\vert \partial_{x}^{s}J_{3}\vert &\lesssim\varepsilon^{4}M\left[\int_{0}^{\frac{t}{2}}\theta_{i}^{2+s}(t-\tau)\ast_{x}\Bigg(\sum_{m=1}^{3}(1+\tau)^{-\frac{1-2\delta_{0}}{4}}\psi_{m}^{3}\Bigg)(\tau)d\tau\right.
\\
&\quad \left.+\int_{\frac{t}{2}}^{t}\theta_{i}^{2}(t-\tau)\ast_{x}\Bigg(\sum_{m=1}^{3}(1+\tau)^{-\frac{1-2\delta_{0}+2s}{4}}\psi_{m}^{3}\Bigg)(\tau)d\tau\right]
\\
&\lesssim\varepsilon^{4}M(1+t)^{-\frac{1-2\delta_{0}+2s}{4}}\Bigg(\psi_{i}^{\frac{3}{2}}+\sum_{j\neq i}\tilde\psi_{j}^{\frac{3}{2}}\Bigg)\,.
\end{align*}
Combining the component estimates, the known forcing has amplitude
$C_{0}\varepsilon^{3}$, whereas the terms containing $d_{i}$ have amplitude at most
$C_{1}\varepsilon^{4}M$, with the corresponding profiles in \eqref{asatz-d-i}.
Here $C_{0}\geq1$ and $C_{1}$ depend only on the
initial-data norm, and are independent of $M$, $\varepsilon$, and the terminal time.
Thus the output coefficient is bounded by
\[
C_{0}\varepsilon^{3}+C_{1}\varepsilon^{4}M=\varepsilon^{3}\left(C_{0}+C_{1}\varepsilon M\right)
\]
Choose $M \geq 4 C_{0}$ sufficiently large and $0<\varepsilon\leq1$ sufficiently small that
\[
C_{1}\varepsilon M\leq\frac14 M.
\]
The output is then bounded by $\frac{1}{2}M\varepsilon^{3}$,
which closes the ansatz and proves \eqref{sec43:flux-pointwise}.
Finally, integrating \eqref{sec43:flux-pointwise} over $x\in\mathbb R$ yields \eqref{sec43:flux-norms}. This completes the proof.
\end{proof}

We next compare $u_i^f$ with $\om_i^f$. Their complete Euler fluxes
coincide, whereas their viscosity matrices are
$\operatorname{diag}(A_1,A_2,A_3)$ and $\mathcal A(\om^f)$.
Since $\mathcal A(0)=(A_{ij})$ and the equilibrium characteristic
speeds are simple, the artificial matrix prescribed in
\cite[(8.49)]{[Zeng]} is exactly
$\operatorname{diag}(A_{11},A_{22},A_{33})$ in the present coordinates.
Therefore, \eqref{sec43:full-artificial} is precisely the full-flux
artificial system \cite[(8.48)]{[Zeng]}.

\begin{proposition}\label{sec43:prop-viscosity-comparison}
Under the structural hypotheses and common initial data above,
for sufficiently small $\varepsilon$ the solutions $u_i^f$ and
$\om_i^f$ satisfy, for $i=1,2,3$,
\begin{align}\label{sec43:viscosity-pointwise}
 |\om_i^f(x,t)-u_i^f(x,t)|
 &\lesssim_{\delta_0}\varepsilon (1+t)^{-\frac{1-2\delta_{0}}{4}}
       \sum_{j=1}^3\left[
       \frac{1}{(1+t+|x-\lambda_j(1+t)|^2)^{\frac{3}{4}}}
\right.
\notag\\
       &\quad \left.+\frac{1}{(1+t+|x-\lambda_j(1+t)|)^{\frac{1}{2}}
                       (1+t+|x-\lambda_j(1+t)|^2)^{\frac{3}{8}}}\right]\notag\\[2mm]
 &\quad+\varepsilon (1+t)^{-\frac{1}{2}}|x-\lambda_i(1+t)|^{-1+\frac{\delta_{0}}{2}}
       \mathbf1_{\left\{\frac{2\sqrt{1+t}}{\delta_{0}}\leq|x-\lambda_i(1+t)|\leq\frac{\delta_{0}(1+t)}{2}\right\}}\,.
\end{align}
The last term is understood to vanish outside the indicated set,
including when the set is empty. The right-hand side is integrable
in $x$ for every $t\geq0$. Moreover,
\begin{equation}\label{sec43:viscosity-norms}
 \begin{aligned}
 \|\om^f(t)-u^f(t)\|_{L_x^1}
   &\lesssim_{\delta
_{0}}\varepsilon(1+t)^{-\frac{1}{2}+\frac{\delta_{0}}{2}}.
 \end{aligned}
\end{equation}
\end{proposition}
\begin{proof}
The structural and initial-data hypotheses allow us to apply
\cite[Remark~8.9, (8.48)--(8.52), pp.~109--110]{[Zeng]} to
\eqref{sec43:full-ns-characteristic} and
\eqref{sec43:full-artificial}. By \eqref{sec42:data-decay}, the initial-data size is bounded by
$C\varepsilon$, and the characteristic difference is $\om_i^f-u_i^f$. Substituting the profiles in
\cite[(2.19) and (2.25)]{[Zeng]} into (8.52) gives precisely
\eqref{sec43:viscosity-pointwise}. We retain the original
centers $\lambda_j(t+1)$ and the indicated region in that estimate.
\end{proof}

Finally, under the structural hypotheses of this subsection, we write
\[
 \begin{aligned}
 \langle\mathbb P_0f,E_i\rangle_\xi-\om_i^f
 ={}&\big(\langle\mathbb P_0f,E_i\rangle_\xi-u_i\big)
       +(u_i-u_i^f)+(u_i^f-\om_i^f).
 \end{aligned}
\]
Combining the kinetic to artificial comparison obtained in Subsection 4.2 with Propositions~\ref{sec43:prop-flux-comparison} and \ref{sec43:prop-viscosity-comparison} yields the desired space–time comparison with the full Navier–Stokes solution.

\section{Appendix}
\subsection{Some convolution estimate}

We list some useful convolution formula in this section, which
are of elementary nature and were studied in viscous conservation laws. We need these convolution estimates for the interaction of waves in section \ref{wave} and \ref{fluid-approximation}. Note that our linear wave $\theta_{i}^{\alpha}$ (see \eqref{fluid-def}) is polynomial type with arbitrary order, while the linear wave of 1D CNS is exponential type, therefore those formula can be found in \cite{[Liu], [Zeng]} by replacing $\theta_{i}^{\alpha}$ as \eqref{fluid-def}.

\begin{table}[H]
     \renewcommand\arraystretch{2.5}
        \centering
        \begin{tabular}{|c|c|c|}
         \hline
        \textbf{Expression}     &  \textbf{Estimate} & \textbf{Reference}
        \\
        \hline
         $\displaystyle\int_{\sqrt{t}}^{t}\theta_{i}^{1}(t-\tau)\ast_{x}\theta_{j}^{4}(\tau)d\tau$ &
        $\psi_{i}^{\frac{3}{2}}$ &(3.30)-(3.32) in \cite{[Zeng]}
        \\
        \hline
         $\displaystyle\int_{\sqrt{t}}^{t}\theta_{i}^{1}(t-\tau)\ast_{x}(1+\tau)^{-\frac{1}{2}}\psi_{j}^{3}(\tau)d\tau$ &
        $\psi_{i}^{\frac{3}{2}}$ & Lemma \ref{lem:late-time-weighted-convolution}
        \\
        \hline
        $\displaystyle\int_{0}^{t}\theta_{i}^{2}(t-\tau)\ast_{x}\psi_{i}^{3}(\tau)d\tau$&$\theta_{i}^{2}\ln(1+t)+\psi_{i}^{2}$&(5.16) in \cite{[Liu]}
        \\
        \hline
         $\displaystyle\int_{0}^{t}\theta_{i}^{2}(t-\tau)\ast_{x}\psi_{j}^{3}(\tau)d\tau$&
        $\psi^{\frac{3}{2}}_{i}+\tilde{\psi}^{\frac{3}{2}}_{j}$&
        (5.18) in \cite{[Liu]}
        \\
        \hline
        $\displaystyle\int_{0}^{t}\theta_{i}^{3}(t-\tau)\ast_{x}\psi_{i}^{3}(\tau)d\tau$&
        $(1+t)^{-\frac{1}{2}}\ln(1+t)\psi_{i}^{2}$ &
        Similar to Lemma \ref{lemm-convalution-A}
        \\
        \hline
        $\displaystyle\int_{0}^{t}\theta_{i}^{3}(t-\tau)\ast_{x}\psi_{j}^{3}(\tau)d\tau$&
        $(1+t)^{-\frac{1}{2}}\left(\psi_{i}^{\frac{3}{2}}+\psi_{j}^{\frac{3}{2}}\right)$ &
         Similar to Lemma \ref{lemm-convalution-A}
        \\
        \hline
        $\displaystyle\int_{0}^{t}\theta_{i}^{3}(t-\tau)\ast_{x}\psi_{j}^{3}(\tau)d\tau$&
        $(1+t)^{-\frac{1}{2}}\ln(1+t)\left(\tilde{\psi_{i}}^{\frac{3}{2}}+\tilde{\psi}_{j}^{\frac{3}{2}}\right)$ & Similar to Lemma \ref{lemm-convalution-A}
        \\
        \hline
         $\displaystyle\theta_{i}^{1}(t-\sqrt{t})\ast_{x}\theta_{j}^{1}(\sqrt{t})$&
        $\theta_{i}^{1}$&(3.28) in \cite{[Zeng]}
        \\
        \hline
         $\displaystyle\theta_{i}^{1}(t-\sqrt{t})\ast_{x}\psi_{j}^{3}(\sqrt{t})$&
        $\psi_{i}^{2}$& Lemma \ref{lem:sqrt-time-spatial-convolution}
        \\
        \hline
        $\displaystyle \int_{0}^{\sqrt{t}}\theta_{i}^{2}(t-\tau)\ast_{x}\theta_{j}^{\alpha}(\tau)d\tau$,  $1\leq \alpha<3$&
        $\theta_{i}^{\frac{\alpha+1}{2}}$&(3.20)-(3.23) in \cite{[Zeng]}
        \\
        \hline
        $\displaystyle\int_{0}^{t}\theta_{i}^{2}(t-\tau)\ast_{x}\theta_{i}^{2}(\tau)d\tau$&$\theta_{i}^{1}$&
        (5.7) in \cite{[Liu]}
        \\
        \hline
        $\displaystyle\int_{0}^{\frac{t}{2}}\theta_{i}^{3}(t-\tau)\ast_{x}\theta_{i}^{2}(\tau)d\tau$ &$\theta_{i}^{2}$& (5.10) in \cite{[Liu]}
        \\
        \hline
        $\displaystyle\int_{0}^{t}\theta_{i}^{4}(t-\tau)\ast_{x}\theta_{j}^{3}(\tau)d\tau$& $(1+t)^{-\frac{3}{4}}\left(\psi_{j}^{\frac{3}{2}}+\tilde{\psi}_{i}^{\frac{3}{2}}\right)$
        &(5.8) in \cite{[Liu]}
        \\
        \hline
        $\displaystyle\int_{0}^{t}\theta_{i}^{2}(t-\tau)\ast_{x}\theta_{i}^{\frac{5}{2}}(\tau)d\tau$& $\theta_{i}^{\frac{3}{2}}$
        &(5.7) in \cite{[Liu]}
        \\
        \hline
        \end{tabular}
\caption{}
\label{table-1}
    \end{table}

    \begin{table}[H]
    \renewcommand\arraystretch{2.5}
        \centering

        \begin{tabular}{|c|c|c|}
        \hline
        \textbf{Expression}     &  \textbf{Estimate} & \textbf{Reference}
         \\
        \hline
        $\displaystyle\int_{0}^{t}\theta_{i}^{2}(t-\tau)\ast_{x}\theta_{i}^{3-\delta_{0}}(\tau)d\tau$&$\theta_{i}^{2-\delta_{0}}$ &(5.7) in \cite{[Liu]}
        \\
        \hline
        $\displaystyle\int_{0}^{t}\theta_{i}^{2}(t-\tau)\ast_{x}\theta_{i}^{3}(\tau)d\tau$&$\theta_{i}^{2}\ln(1+t)$&(5.7) in \cite{[Liu]}
        \\
        \hline
       $\displaystyle \int_{0}^{t}\theta_{i}^{2}(t-\tau)\ast_{x}\theta_{i}^{\alpha}(\tau)d\tau$, $\alpha>3$&$\theta_{i}^{2}$
        &(5.7) in \cite{[Liu]}
        \\
        \hline
        $\displaystyle\int_{0}^{t}\theta_{i}^{3}(t-\tau)\ast_{x}\theta_{j}^{\frac{7}{2}-\delta_{0}}(\tau)d\tau$&$(1+t)^{-\frac{3}{4}+\frac{\delta_{0}}{2}}
        \ln(1+t)\left(\psi_{i}^{\frac{3}{2}}+\tilde{\psi}_{j}^{\frac{3}{2}}\right)$& (5.8) in \cite{[Liu]}
        \\
        \hline
        $\displaystyle\int_{0}^{t}\theta_{i}^{2}(t-\tau)\ast_{x}(1+\tau)^{-\alpha}\psi_{j}^{3}(\tau)d\tau$,  $0<\alpha\leq \frac{1}{2}$&
        $\psi_{i}^{2}+(1+t)^{-\alpha}\left(\psi_{i}^{\frac{3}{2}}+\tilde{\psi}_{j}^{\frac{3}{2}}\right)$& Lemma \ref{lemm-convalution-A}
        \\
        \hline
        $\displaystyle\int_{0}^{t}\theta_{i}^{2}(t-\tau)\ast_{x}(1+\tau)^{-\alpha}\psi_{i}^{3}(\tau)d\tau$,  $\alpha>0$&
        $\psi_{i}^{2}$&Similar to Lemma \ref{lemm-convalution-A}
        \\
        \hline
        $\displaystyle\int_{0}^{t}\theta_{i}^{3-\delta_{0}}(t-\tau)\ast_{x}\theta_{j}^{3}(\tau)d\tau$&$(1+t)^{-\frac{1}{2}+\frac{\delta_{0}}{2}}\left(\psi_{i}^{\frac{3}{2}}+\psi_{j}^{\frac{3}{2}}\right)$&(5.8) in \cite{[Liu]}
        \\
        \hline
        $\displaystyle\int_{0}^{t}\theta_{i}^{3-\delta_{0}}(t-\tau)\ast_{x}\theta_{j}^{3}(\tau)d\tau$&$(1+t)^{-\frac{1}{2}+\frac{\delta_{0}}{2}}\ln(1+t)\left(\tilde\psi_{i}^{\frac{3}{2}}+\tilde\psi_{j}^{\frac{3}{2}}\right)$&(5.8) in \cite{[Liu]}
        \\
        \hline
        $\displaystyle\int_{0}^{t}\theta_{i}^{3}(t-\tau)\ast_{x}\theta_{i}^{\alpha}(\tau)d\tau$, $\alpha>3$& $\theta_{i}^{3}$ &(5.7) in \cite{[Liu]}
        \\
        \hline
        $\displaystyle\int_{0}^{t}\theta_{i}^{4}(t-\tau)\ast_{x}\psi_{i}^{3}(\tau)d\tau$&$(1+t)^{-\frac{1}{2}}\psi_{i}^{2}$&Similar to Lemma \ref{lemm-convalution-A}
        \\
        \hline
        $\displaystyle\int_{0}^{t}\theta_{i}^{4}(t-\tau)\ast_{x}\psi_{j}^{3}(\tau)d\tau$&$(1+t)^{-\frac{1}{2}}\left(\psi_{i}^{2}+\psi_{j}^{2}\right)$&Similar to Lemma \ref{lemm-convalution-A}
        \\
        \hline
        \end{tabular}
\caption{}
\label{table-2}
    \end{table}

\subsection{Proof of the convolution estimate}
In the estimates of the remainder, some convolution estimates may not appear in the references \cite{[Liu]} or \cite{[Zeng]}, and hence we will prove it in this subsection. To begin with, by the completing square, we have
\begin{align}
    &\quad\frac{\vert x-y-\lambda(t-\tau)\vert^{2}}{1+t-\tau}+\frac{\vert y-\lambda^{'}\tau\vert^{2}}{1+\tau}
    \notag\\
    &=\frac{2+t}{(1+t-\tau)(1+\tau)}\left[y+\lambda^{'}-\frac{(1+\tau)(x+2\lambda^{'}-(\lambda-\lambda^{'})(t-\tau))}{2+t}\right]^{2}
    \label{eq-complete-square}\\
    &\quad+\frac{\vert x-\lambda(t-\tau)-\lambda^{'}\tau\vert^{2}}{2+t}.\notag
\end{align}
\begin{lemma}
\label{lemm-convalution-A} For $i\neq j$, $0<\alpha\leq 2$,
\begin{equation*}
A:=\theta_{i}^{2}\ast_{x,t}(1+t)^{-\alpha}\psi_{j}^{3}\lesssim \psi_{i}^{2}+(1+t)^{-\alpha}\left(\psi_{i}^{\frac{3}{2}}+\tilde{\psi}_{j}^{\frac{3}{2}}\right).
\end{equation*}
\end{lemma}
\begin{proof}
We assume that $\lambda_j>\lambda_i$; the other case follows
by reflection. It suffices to consider $t\geq8$, since bounded times
are estimated directly.
Using this change of variables and \eqref{eq-complete-square}, we have
\begin{align*}
A&\lesssim\int_{0}^{t}\int_{\mathbb{R}}(1+t-\tau)^{-1}\left[1+\frac{\left\vert \bar{x}-y\right\vert^{2}}{1+t-\tau}\right]^{-N}(1+\tau)^{-\alpha-\frac{3}{2}}\left[1+\frac{\left\vert y-\tau\right\vert^{2}}{1+\tau}\right]^{-\frac{3}{2}}dyd\tau
\\
&\lesssim (1+t)^{-\frac{1}{2}}\int_{0}^{t}(1+t-\tau)^{-\frac{1}{2}}(1+\tau)^{-\alpha-1}\left[1+\frac{\left\vert \bar{x}-\tau\right\vert^{2}}{1+t}\right]^{-1}d\tau,
\end{align*}
where $\bar{x}=\frac{x-\lambda_{i}t}{\lambda_{j}-\lambda_{i}}$.
\\
For $\bar{x}\leq 0$ or $0\leq \bar{x}\leq \sqrt{1+t}$,
\begin{equation*}
A\lesssim (1+t)^{-\frac{1}{2}}\left[1+\frac{\bar{x}^{2}}{1+t}\right]^{-1}\int_{0}^{t}(1+t-\tau)^{-\frac{1}{2}}(1+\tau)^{-\alpha-1}d\tau\lesssim
\psi_{i}^{2}.
\end{equation*}
For $\sqrt{1+t}<\bar{x}<t-\sqrt{1+t}$, by choosing the constant $a$ to be large, we have
\begin{align*}
A&\lesssim (1+t)^{-\frac{1}{2}}\left[\int_{0}^{\frac{\bar{x}}{a}}(1+t)^{-\frac{1}{2}}(1+\tau)^{-\alpha-1}\left[1+\frac{\bar{x}^{2}}{1+t}\right]^{-1}d\tau\right.
\\
&\left.\quad+\int_{\frac{\bar{x}}{a}}^{\frac{\bar{x}}{a}+(1-\frac{1}{a})t}(1+t-\bar{x})^{-\frac{1}{2}}(1+\bar{x})^{-\alpha-1}\left[1+\frac{\left\vert\bar{x}-\tau\right\vert^{2}}{1+t}\right]^{-1}d\tau\right.
\\
&\left.\quad +\int_{\frac{\bar{x}}{a}+(1-\frac{1}{a})t}^{t}(1+t-\tau)^{-\frac{1}{2}}(1+t)^{-\alpha-1}\left[1+\frac{\left\vert \bar{x}-t\right\vert^{2}}{1+t}\right]^{-1}d\tau\right]
\\
&\lesssim \psi_{i}^{2}+\left\vert x-\lambda_{j}t\right\vert^{-\frac{1}{2}}\left\vert x-\lambda_{i}t\right\vert^{-\alpha-1}+(1+t)^{-\alpha}\psi_{j}^{2}.
\end{align*}
In particularly, as $\left\vert x-\lambda_{j}t\right\vert<\frac{(\lambda_{j}-\lambda_{i})t}{2}$,
\begin{equation*}
\left\vert x-\lambda_{i}t\right\vert =\left\vert \left(\lambda_{j}-\lambda_{i}\right)t-\left(\lambda_{j}t-x\right)\right\vert>\frac{\left(\lambda_{j}-\lambda_{i}\right)t}{2}>\left\vert x-\lambda_{j}t\right\vert,
\end{equation*}
we have
\begin{align*}
\left\vert x-\lambda_{j}t\right\vert^{-\frac{1}{2}}\left\vert x-\lambda_{i}t\right\vert^{-\alpha-1}&\lesssim \left(\left\vert x-\lambda_{j}t\right\vert+\sqrt{1+t}\right)^{-\frac{1}{2}}\left(\left\vert x-\lambda_{j}t\right\vert+1+t\right)^{-1}(1+t)^{-\alpha}
\\
&\lesssim (1+t)^{-\alpha}\tilde{\psi}_{j}^{\frac{3}{2}}.
\end{align*}
As $\left|x-\lambda_jt\right|\geq\frac{(\lambda_j-\lambda_i)t}{2}$,
we have $\left|x-\lambda_jt\right|\asymp1+t$ and
$\sqrt{1+t}\lesssim\left|x-\lambda_it\right|\lesssim1+t$.
\\
For $0<\alpha\leq\frac{1}{2}$,
\begin{align*}
 &\left|x-\lambda_jt\right|^{-\frac{1}{2}}
                     \left|x-\lambda_it\right|^{-\alpha-1}\lesssim(1+t)^{-\alpha}
                     \left|x-\lambda_it\right|^{-3/2}
 \lesssim(1+t)^{-\alpha}\psi_i^{3/2}.
\end{align*}
For $\frac{1}{2}\leq\alpha\leq2$,
\begin{align*}
 \left|x-\lambda_jt\right|^{-\frac{1}{2}}
\left|x-\lambda_it\right|^{-\alpha-1}\lesssim\left|x-\lambda_it\right|^{-2}\lesssim\psi_i^2,
\end{align*}
since $(1+t)^{-\frac{1}{2}}\left|x-\lambda_it\right|^{1-\alpha}$ is
uniformly bounded on this region.
\\
For $t-\sqrt{1+t}<\bar{x}<t+\sqrt{1+t}$,
\begin{align*}
A&\lesssim (1+t)^{-\frac{1}{2}}\left[\int_{0}^{\frac{t}{2}}(1+t)^{-\frac{1}{2}}(1+\tau)^{-\alpha-1}(1+t)^{-1}d\tau\right.
\\
&\left.\quad+\int_{\frac{t}{2}}^{t-\sqrt{1+t}}(1+t)^{-\frac{1}{4}}(1+t)^{-\alpha-1}\left[1+\frac{\left\vert\bar{x}-\tau\right\vert^{2}}{1+t}\right]^{-1}d\tau
+\int_{t-\sqrt{1+t}}^{t}(1+t-\tau)^{-\frac{1}{2}}(1+t)^{-\alpha-1}d\tau\right]
\\[1mm]
&\lesssim (1+t)^{-2}+(1+t)^{-\frac{5}{4}-\alpha}
 \lesssim\psi_i^2+(1+t)^{-\alpha}\tilde\psi_j^{3/2}.
\end{align*}
\\
For $\bar{x}>t+\sqrt{1+t}$,
\begin{align*}
A&\lesssim (1+t)^{-\frac{1}{2}}\left[\int_{0}^{\frac{t}{2}}(1+t)^{-\frac{1}{2}}(1+\tau)^{-\alpha-1}\left[1+\frac{\bar{x}^{2}}{1+t}\right]^{-1}d\tau
\right.\\
&\quad\left.+\int_{\frac{t}{2}}^{t}(1+t-\tau)^{-\frac{1}{2}}(1+t)^{-\alpha-1}\left[1+\frac{\left\vert \bar{x}-t\right\vert^{2}}{1+t}\right]^{-1}d\tau\right]
\\
&\lesssim \psi_{i}^{2}+(1+t)^{-\alpha}\psi_{j}^{2}.
\end{align*}
Combining above four cases, we complete the proof.
\end{proof}


\begin{lemma}\label{lem:late-time-weighted-convolution}
Let $i\neq j$, $t\geq 1$, and let $\theta_k^\alpha$ and $\psi_k^\alpha$ be defined
by (1.9) and (1.10), respectively.
\begin{equation}\label{eq:late-time-weighted-convolution}
B:=\int_{\sqrt t}^{t}\int_{\mathbb R}
\theta_i^1(x-y,t-\tau)(1+\tau)^{-1/2}
\psi_j^3(y,\tau)\,dy\,d\tau \lesssim \psi_i^{3/2}(x,t)\,.
\end{equation}
\end{lemma}
\begin{proof}
We assume that $\lambda_j>\lambda_i$; the other case follows by reflection.
It suffices to consider $t\geq 8$, since bounded times are estimated directly.
 Using this change of variables and
the completing-square identity \eqref{eq-complete-square}, we have
\begin{align*}
B
&\lesssim \int_{\sqrt t}^{t}\int_{\mathbb R}
 (1+t-\tau)^{-\frac12}
 \left(1+\frac{|\bar x-y|^2}{1+t-\tau}\right)^{-N}
 (1+\tau)^{-2}
 \left(1+\frac{|y-\tau|^2}{1+\tau}\right)^{-\frac32}
 \,dy\,d\tau\\
&\lesssim (1+t)^{-\frac12}
 \int_{\sqrt t}^{t}(1+\tau)^{-\frac32}
 \left(1+\frac{|\bar x-\tau|^2}{1+t}\right)^{-1}d\tau,
\end{align*}
where $\bar x=\frac{x-\lambda_i t}{\lambda_j-\lambda_i}$.

For $\bar x\leq 0$ or $0\leq\bar x\leq\sqrt{1+t}$,
\begin{align*}
B\lesssim (1+t)^{-\frac12}
 \left(1+\frac{\bar x^2}{1+t}\right)^{-1}
 \int_{\sqrt t}^{t}(1+\tau)^{-\frac32}d\tau\lesssim (1+t)^{-\frac34}
 \left(1+\frac{\bar x^2}{1+t}\right)^{-1}
 \lesssim \psi_i^{3/2}.
\end{align*}

For $\sqrt{1+t}<\bar x<t-\sqrt{1+t}$, we choose a fixed constant
$a>2$, we obtain
\begin{align*}
B&\lesssim (1+t)^{-\frac12}\Bigg[
\left(1+\frac{\bar x^2}{1+t}\right)^{-1}
 \int_{\sqrt t}^{\max\left\{\frac{\bar{x}}{a},\sqrt{t}\right\}}(1+\tau)^{-\frac32}d\tau+\bar x^{-\frac32}
 \int_{\max\left\{\frac{\bar{x}}{a},\sqrt{t}\right\}}^{\frac{\bar{x}}{a}+(1-\frac{1}{a})t}
 \left(1+\frac{|\bar x-\tau|^2}{1+t}\right)^{-1}d\tau\\
&\quad+(1+t)^{-\frac32}
 \int_{\frac{\bar{x}}{a}+(1-\frac{1}{a})t}^{t}
 \left(1+\frac{|\bar x-\tau|^2}{1+t}\right)^{-1}d\tau
\Bigg]\\
&\lesssim (1+t)^{-\frac34}
 \left(1+\frac{\bar x^2}{1+t}\right)^{-1}
 +(\bar x)^{-\frac32}+(1+t)^{-\frac32}
 \\
 &\lesssim\psi_i^{3/2}.
\end{align*}

For $t-\sqrt{1+t}<\bar x<t+\sqrt{1+t}$,
\begin{align*}
B&\lesssim (1+t)^{-\frac12}\Bigg[
(1+t)^{-1}\int_{\sqrt t}^{t/2}
 (1+\tau)^{-\frac32}d\tau+(1+t)^{-\frac32}\int_{t/2}^{t-\sqrt{1+t}}
 \left(1+\frac{|\bar x-\tau|^2}{1+t}\right)^{-1}d\tau\\
&\quad+(1+t)^{-\frac32}\int_{t-\sqrt{1+t}}^{t}
 \left(1+\frac{|\bar x-\tau|^2}{1+t}\right)^{-1}d\tau
\Bigg]\\
&\lesssim (1+t)^{-\frac74}+(1+t)^{-\frac32}
\\
 &\lesssim\psi_i^{3/2}.
\end{align*}

For $\bar x>t+\sqrt{1+t}$,
\begin{align*}
B&\lesssim (1+t)^{-\frac12}\Bigg[
\left(1+\frac{\bar x^2}{1+t}\right)^{-1}
 \int_{\sqrt t}^{t/2}(1+\tau)^{-\frac32}d\tau+(1+t)^{-\frac32}\int_{t/2}^{t}
 \left(1+\frac{|\bar x-\tau|^2}{1+t}\right)^{-1}d\tau
\Bigg]\\
&\lesssim (1+t)^{-\frac34}
 \left(1+\frac{\bar x^2}{1+t}\right)^{-1}
 +(1+t)^{-\frac32} \,\mathbf{1}_{\left\{\bar x\leq2t\right\}}+
 (\bar x)^{-2}\, \mathbf{1}_{\left\{\bar x>2t\right\}}
 \\
&\lesssim\psi_i^{3/2}.
\end{align*}
Combining the above four cases, we complete the proof.
\end{proof}

\begin{lemma}\label{lem:sqrt-time-spatial-convolution}
Let $i\neq j$, $t\geq 1$ and let $\theta_k^\alpha$ and $\psi_k^\alpha$
be defined by (1.9) and (1.10), respectively, with $N\geq\frac{3}{2}$.
Then,
\begin{equation}\label{eq:sqrt-time-spatial-convolution}
\int_{\mathbb R}
\theta_i^1(x-y,t-\sqrt t)\psi_j^3(y,\sqrt t)\,dy \lesssim \psi_i^2(x,t)
\end{equation}
\end{lemma}
\begin{proof}
We assume that $\lambda_j>\lambda_i$; the other case follows by reflection.

Using this change of variables and
the completing-square identity \eqref{eq-complete-square}, we have
\begin{align*}
\int_{\mathbb R}&\theta_i^1(x-y,t-\sqrt t)\psi_j^3(y,\sqrt t)\,dy\\
&\lesssim
(1+t-\sqrt t)^{-\frac12}(1+\sqrt t)^{-\frac32}
\int_{\mathbb R}
\left(1+\frac{|\bar x-y|^2}{1+t-\sqrt t}\right)^{-N}
\left(1+\frac{|y-\sqrt t|^2}{1+\sqrt t}\right)^{-\frac32}dy\\[2mm]
&\lesssim\left(2+t+|\bar x-\sqrt t|^2\right)^{-1}
\\[2mm]
&\lesssim (1+t+|\bar x|^2)^{-1}
\\[2mm]
&\lesssim\psi_i^2(x,t).
\end{align*}
We complete the proof.
\end{proof}

\vspace{2em}

\noindent \textbf{Acknowledgment}\hspace*{1em}
This work is partially supported by the National Key R\&D Program of China under grant 2022YFA1007300. H.-T. Wang is supported by NSFC under Grant No. 12371220 and 12031013. H.-T. Chang and K.-C. Wu are supported by the National Science and Technology Council under the grant NSTC 112-2628-M-006-006 -MY4 and National Center for Theoretical Sciences. The authors would like to thank Professor Tai-Ping Liu for his valuable suggestions and insightful discussions.
\\

\noindent \textbf{Data Availability}\hspace*{1em}
Data sharing is not applicable to this article as no datasets were generated or analyzed during the current study.
\\

\noindent \textbf{Declarations}
\\

\noindent \textbf{Conflict of interest}\hspace*{1em}
The authors declare that they have no conflict of interest.
\medskip

\noindent\textbf{Declaration of generative AI.}
During the preparation of this work, the authors used ChatGPT (OpenAI)
to assist with language editing, improving readability and organization.
The authors reviewed and verified all suggested changes and take full
responsibility for the content of the article.


\begin{thebibliography}{99}




\bibitem{[EP]} R. Ellis, M. Pinsky, The first and second fluid approximation
to the linearized Boltzmann equation. J. Math. Pure. Appl. 54 (1975),
125-156.


\bibitem{[Grad]} H. Grad, Asymptotic theory of the Boltzmann equation,
Rarefied Gas Dynamics, J. A. Laurmann, Ed. 1, 26, pp.26--59 Academic Press,
New York, 1963.




\bibitem{[Kawashima-Matsumura-Nishida]} S. Kawashima, A. Matsumura, and T.
Nishida, On the Fluid-Dynamical Approximation to the Boltzmann Equation at
the Level of the Navier-Stokes Equation, Commun. Math. Phys., 70 (1979),
97-124.

\bibitem{[Kai]} K. Koike, Time-Asymptotic Expansion with Pointwise Remainder Estimates for 1D Viscous Compressible Flow. Arch Rational Mech Anal 247, 81(2023).

\bibitem{[LeeLiuYu-07]} M.-Y. Lee, T.-P. Liu, S.-H. Yu, Large-time behavior of solutions for the Boltzmann equation with hard potentials, Comm. Math. Phys. 269 (2007), no. 1, 17–37.


\bibitem{[LiYangZhong-20]} H.-L. Li, T. Yang, M.-Y. Zhong, Green's function and pointwise space-time behaviors of the Vlasov-Poisson-Boltzmann system, Arch. Ration. Mech. Anal. 235 (2020), no. 2, 1011–1057.

\bibitem{[LiYangZhong-23]} H.-L. Li, T. Yang, M.-Y. Zhong, Green's function and pointwise behavior of the one-dimensional Vlasov-Maxwell-Boltzmann system, Arch. Ration. Mech. Anal. 247 (2023), no. 5, Paper No. 82, 95 pp.

\bibitem{[Haitao]}Y.-C. Lin, H.T. Wang, K.-C. Wu, Space-time Structure and particle-fluid duality of solutions for Boltzmann equation with hard potentials, Arch. Ration. Mech. Anal. 249 (2025), no. 6, Paper No. 75, 52 pp.

\bibitem{[LinWangWu2]} Y.-C. Lin, H.-T. Wang,  K.-C. Wu, 3D hard sphere
Boltzmann equation: explicit structure and the transition process from
polynomial tail to Gaussian tail, (2024), arXiv:2408.02183.



\bibitem{[Liu]}T.-P. Liu, Pointwise Convergence to Shock Waves for Viscous Conservation Laws. Communications on Pure and Applied Mathematics, 50(1997), 1113-1182.



\bibitem{[Liu-17]} T.-P. Liu, Hopf-Cole transformation, Bull. Inst. Math. Acad. Sin. (N.S.), 12 (2017), no. 1, 71–101.

\bibitem{[Liu-21]} T.-P. Liu, Shock waves, Grad. Stud. Math., 215, American Mathematical Society, Providence, RI, 2021, xvii+437 pp.

\bibitem{[Yu]}T.-P. Liu and S.-H.Yu, The Green's function and large‐time behavior of solutions for the one‐dimensional Boltzmann equation. Communications on Pure and Applied Mathematics, 57(12)(2004), 1543-1608.

\bibitem{[LiuYu-04]}  T.-P. Liu, S.-H. Yu,  Boltzmann equation: micro-macro decompositions and positivity of shock profiles, Comm. Math. Phys., 246 (2004), no. 1, 133–179.


\bibitem{[LiuYu1]} T.-P. Liu, S.-H. Yu, Green's function of Boltzmann
equation, 3-D waves, Bull. Inst. Math. Acad. Sin. 1 (2006), 1-78.

\bibitem{[LiuYu-07]}  T.-P. Liu, S.-H. Yu, Initial-boundary value problem for one-dimensional wave solutions of the Boltzmann equation, Comm. Pure Appl. Math. 60 (2007), no. 3, 295–356.


\bibitem{[LiuYu2]} T.-P. Liu, S.-H. Yu, Solving Boltzmann equation, Part I :
Green's function, Bull. Inst. Math. Acad. Sin. (N.S.), 6 (2011), 151-243.

\bibitem{[LiuYu-13]} T.-P. Liu, S.-H. Yu, Invariant manifolds for steady Boltzmann flows and applications, Arch. Ration. Mech. Anal. 209 (2013), no. 3, 869–997.

\bibitem{[Zeng]}T.-P. Liu and Y. Zeng, Large time behavior of solutions for general quasilinear hyperbolic parabolic systems of conservation laws. Memoirs of the American mathematical Society, Volume 125, Number 599(1997), ISSN-0065-9266.







%


\bibitem{[Ukai-Yang]} S. Ukai, T. Yang, The Boltzmann equation in the space $%
L^{2}\cap L_{\beta }^{\infty }$: global and time-periodic solutions, Anal.
Appl., 4 (2006), no. 3, 263-310.



\bibitem{[YangYu]} T. Yang, H. Yu, Spectrum analysis of some kinetic
equations, Arch. Ration. Mech. Anal., 222 (2016), no. 2, 731-768.

\bibitem{[Yu-10]} S.-H. Yu, Nonlinear wave propagations over a Boltzmann shock profile, J. Amer. Math. Soc. 23 (2010), no. 4, 1041–1118.

\bibitem{[Guo2010]}Y. Guo, Decay and Continuity of the Boltzmann Equation in Bounded Domains. Arch. Ration. Mech. Anal., 197(2010), 713–809.

\bibitem{[Zhong]} M.Y. Zhong, Optimal time-decay rates of the Boltzmann
equation, Sci. China Math., 57 (2014), 807-822.

\end{thebibliography}
\end{document}